\documentclass[12pt]{amsart}
\usepackage{amsmath,amsthm,amssymb,amsfonts}
\usepackage{latexsym}
\usepackage{enumerate}
\usepackage{verbatim}
\usepackage{multicol}
\usepackage{graphicx}
\usepackage{epsfig}
\usepackage{color}
\usepackage{bezier}
\usepackage{curves}
\usepackage{fullpage}
\usepackage{MnSymbol}
\usepackage{trfsigns}
\usepackage{cases}

\newtheorem{theorem}{Theorem}
\newtheorem{claim}{Claim}

\newtheorem{problem}{Problem}

\newtheorem{remark}{Remark}
\newtheorem{lemma}[theorem]{Lemma}

\newtheorem{proposition}[theorem]{Proposition}
\newtheorem{corollary}[theorem]{Corollary}
\newtheorem{example}{Example}

\newtheorem{definition}{Definition}
\newtheorem{conjecture}{Conjecture}

\newcommand{\NN}{{\mathbb N}}
\newcommand{\ZZ}{{\mathbb Z}}

\DeclareMathOperator{\ainc}{Inc}

\DeclareMathOperator{\comp}{{Comp}}

\DeclareMathOperator{\suff}{suff}
\DeclareMathOperator{\pref}{pref}
\DeclareMathOperator{\prim}{Prim}

\DeclareMathOperator{\age}{Age}
\DeclareMathOperator{\inc}{Inc}

\DeclareMathOperator{\fac}{Fac}

\title[Minimal prime ages]{Minimal prime ages, words and permutation graphs}

\author[D.Oudrar] {Djamila Oudrar}
\address{Faculty of Mathematics, USTHB, Algiers, Algeria}
\email {dabchiche@usthb.dz}

\author [M.Pouzet]{Maurice Pouzet}
\address{Univ. Lyon, Universit\'e Claude-Bernard Lyon1, CNRS UMR 5208, Institut Camille Jordan, 43, Bd. du 11 Novembre 1918, 69622
Villeurbanne, France et Department of Mathematics and Statistics, University of Calgary, Calgary, Alberta, Canada}
\email{pouzet@univ-lyon1.fr }

\author[I.Zaguia]{Imed Zaguia*}\thanks{*Corresponding author. Supported by Canadian Defence Academy Research Program, NSERC and LABEX MILYON (ANR-10-LABX-0070) of Universit\'e de Lyon within the program ''Investissements d'Avenir (ANR-11-IDEX-0007'' operated by the French National Research Agency (ANR)}
\address{Department of Mathematics \& Computer Science, Royal Military College of Canada,
P.O.Box 17000, Station Forces, Kingston, Ontario, Canada K7K 7B4}
\email{zaguia@rmc.ca}

\date{\today}

\begin{document}


\begin{abstract}
 This paper is a contribution to the study of hereditary classes of finite graphs.
We  classify these classes according to the number of prime structures they contain.
We consider such classes that are \emph{minimal prime}: classes that contain infinitely many primes but every proper hereditary subclass contains only   finitely many primes. We give a complete description of such classes. In fact, each one of these classes is a well-quasi-ordered age and there are uncountably many of them. Eleven  of these ages are almost multichainable; they    remain well-quasi-ordered when labels in a well-quasi-ordering are added, hence have finitely many bounds. Five ages among  them are exhaustible. Among the remaining ones, only countably many remain well-quasi-ordered when one  label is added, and these have finitely many bounds (except for the age of the infinite path and its complement). The others have infinitely many bounds.

Except for six examples, members of these ages we characterize are permutation graphs. In fact, every age which is not among the eleven ones is the age of a graph associated to a uniformly recurrent  word  on  the integers.

A description of minimal prime classes of posets and bichains is also provided.

Our  results  support the conjecture that if a hereditary class of finite graphs does not remain  well-quasi-ordered by adding labels in a well-quasi ordered set to these graphs, then it is not well-quasi-ordered if we add just two constants to each of these graphs

Our description  of minimal prime classes uses a description of minimal prime graphs \cite {pouzet-zaguia2009}  and previous work by Sobrani \cite{sobranithesis, sobranietat}  and the authors \cite{oudrar, pouzettr}  on properties of uniformly recurrent words and the associated graphs. The completeness of our description is based on classification results   of Chudnovsky, Kim, Oum and Seymour \cite{chudnovsky} and  Malliaris and Terry \cite {malliaris}.
\end{abstract}

\maketitle

\section{Introduction and presentation of the results}
This paper is a contribution to the study of hereditary classes of finite graphs.  We classify these classes according to their proper subclasses. With this idea, our simplest classes are those who contain finitely many proper subclasses, hence these classes are finite. At the next level, there are the classes who contain infinitely many proper subclasses, but every proper subclass contains only finitely many.  That is such classes are infinite but  proper subclasses are   finite. It is a simple exercise based on Ramsey's theorem that there are only two such classes: the class of finite cliques and the class of their complements.
% Such classes contain just one member of size $n$ for each integer $n$ (these are the ages  of chainable structures \cite {fraissetr}).
%This result, which follows from Ramsey's theorem and the compactness theorem of first order logic,  is essentially due to Fra\"{\i}ss\'{e} who,  in the late forties, following the work of Cantor, Hausdorff and Sierpinski, pointed out the role of the quasi-ordering of embeddability and hereditary classes in the theory of relations (see his book \cite{fraisse} for an illustration).
Pursuing this idea further, we would like to attach  a rank to each class, preferably an ordinal. If we do this, it turns out that a class has a rank if and only if the set of its proper subclasses ordered by set inclusion is well founded.  This latter condition amounts to the class being well-quasi-ordered (this follows from Higman's characterization of well-quasi-orders \cite{higman}). This puts forward the importance of  well-quasi-ordered hereditary classes.

A basic construction of well-quasi-ordered hereditary classes of finite graphs and more generally of finite structures goes as follows: chose a finite hereditary class  of finite binary structures and take its closure under lexicographical sums over elements of the class. The fact that this latter class is well-quasi-ordered is a  consequence of a theorem  of Higman \cite{higman} . An important property of such a class is that it contains only finitely many prime structures (see Definition \ref{def:module}). A concrete example of such a class is the class of finite cographs (prime structures in this class have cardinality at most two). A natural question then arises: under what conditions a class that contains infinitely many primes is well-quasi-ordered?

Among  hereditary classes which contain  infinitely many prime members, we show that there are  minimal ones with respect to set inclusion (Theorem \ref{thm:minimalprime}). Furthermore, we show that the minimal ones are well-quasi-ordered ages (Theorem \ref{minimal}).
We obtain some general results that we are able to refine in some special cases like  graphs, ordered sets,  and bichains. We give a complete description of minimal prime ages of graphs (Theorem \ref{thm:charact-minimal-prime-ages}),  of posets, and bichains (Corollary \ref{cor:thm:charact-minimal-prime-ages-bichains}). It turns out that there are $2^{\aleph_0}$ such ages (Corollary \ref {thm:minprimeages}). Eleven  of these ages are almost multichainable; they  remain well-quasi-ordered when labels in a well-quasi-ordering are added, five  being exhaustible. Among the remaining ones, countably many remain well-quasi-ordered when one  label is added  and these have finitely many bounds (except for the age of the infinite path and its complement). The others have infinitely many bounds (Theorem \ref{thm:bound-uniform}).

Except for six examples, members of these ages we characterize are permutation graphs. In fact, every age which is not among the eleven ones is the age of a graph associated to a uniformly recurrent word  on  the integers (this is a consequence of Theorems \ref{thm:permutation-graph}, \ref{thm:recurrent-word} and \ref{thm:uniformly-ages}). This result  supports the conjecture that if a hereditary class of finite graphs does not remain  well-quasi-ordered by adding labels in a well-quasi ordered set to these graphs, then it is not well-quasi-ordered if we add just two constants to each of these graphs.

Our description  of minimal prime classes uses a description of minimal prime graphs \cite {pouzet-zaguia2009}  and previous work by Sobrani \cite{sobranithesis, sobranietat}  and the authors \cite{oudrar, pouzettr}  on properties of uniformly recurrent words and the associated graphs. The completeness of our description is based on classification results   of Chudnovsky, Kim, Oum and Seymour \cite{chudnovsky} and  Malliaris and Terry \cite {malliaris}.

\section{Organisation of the paper}
 In section \ref{sec:preq} we present some prerequisites on graphs, posets and words. In section \ref{section:min-hered-class} we consider binary relational structures with a finite signature,  we give the definition of a minimal prime hereditary class of binary structures and prove their  existence. Section  \ref{section:min-hered-class} contains also the proof of Theorem \ref{minimal} (see subsection \ref{subsection:infinitely-prime}) and a  proof of Theorem \ref{thm:main1} (see subsection \ref{sec:proof-thm:main1}). In Section \ref{section:minimalprimesgraphs} we start with the  classification results   of Chudnovsky, Kim, Oum and Seymour \cite{chudnovsky} and  Malliaris and Terry \cite {malliaris}. Then, we present   our  main results on minimal prime ages.
In Section \ref{section:bornes} we look at the number of bounds of our minimal prime ages. In section \ref{sec:proof-thm:permutation-graph} we provide a proof  of Theorem \ref{thm:permutation-graph} and a characterization of order types of realizers of transitive orientations of $0$-$1$ graphs. In section \ref{sec:modules} we characterize the modules of a $0$-$1$ graph. We prove among other things, that if $G_\mu$ is not prime, then $\mu$ contains large factors of $0$'s or $1$'s. Section \ref{section:embeddings} is devoted to the study of the relation between embeddings of $0$-$1$ words and their corresponding graphs. Results obtained in this section will be used in the proof of Theorem \ref{thm:recurrent-word}. In section \ref{sec:proof-thm:recurrent-word} we give a proof of Theorem \ref{thm:recurrent-word}. Theorem \ref{thm:uniformly-ages} is proved in section \ref{sec:proof-thm:uniformly-ages}. In section \ref{sec:proof-thm:bound-uniform} we investigate bounds  of $0$-$1$ graphs and give a proof of Theorems \ref{thm:bound-uniform}.
% and \ref{thm:bound-uniformages}.

\section{Prerequisites}\label{sec:preq}

\subsection{Graphs, posets and relations}
This paper is mostly about  graphs and posets. Sometimes, we will need  to consider binary relational structures, that is ordered pairs $R:=(V,(\rho_{i})_{i\in I})$ where each $\rho_i$ is a binary relation or a unary relation on $V$. The set $V$, sometimes denoted by $V(R)$,  is the  \emph{domain} or \emph{base} of $R$. The sequence $s:= (n_i)_{i\in I}$ of arity  $n_i$ of $\rho_i$ is the \emph{signature} of $R$ (this terminology is justified since we may identify a unary relation on $V$, that is a subset $U$ of $V$,  with the binary relation  made of pairs $(u,u)$ such that  $u\in U$). We denote by $\Omega_s$ the collection of finite structures of signature $s$. In the sequel we will suppose the signature finite, i.e. $I$ finite. For example, we will consider \emph{bichains}, i.e., relational structures $R:= (V, (\leq', \leq''))$ made of a set $V$ and two linear orders $\leq'$ and $\leq''$ on $V$.

The framework of our study is the theory of relations as developed by Fra\"{\i}ss\'e and subsequent investigators.  At the core  is the notion of embeddability, a quasi-order between relational structures. We recall that a relational structure $R$ is \emph{embeddable} in a relational structure $R'$,  and we set $R\leq R'$,   if $R$ is isomorphic to an induced substructure of $R'$. Several  important notions in the study of these structures, like hereditary classes, ages, bounds, derive from this quasi-order. For example, a class $\mathcal C$ of relational structures, of signature $s$, is \emph{hereditary} if it contains every relational structure that embeds into a member of $\mathcal C$. The \emph{age} of a relational structure $R$ is the class $\age(R)$  of all finite relational structures, considered up to isomorphy, which embed into $R$. This is an \emph{ideal} of $\Omega_s$ that is  a  nonempty,  hereditary  and  \emph{up-directed} class $\mathcal{C}$ (any pair of members of $\mathcal C$ are embeddable in some element of $\mathcal C$). A characterization of ages was given by Fra\"{\i}ss\'e (see chapter 10 of  \cite{fraissetr}). Namely, a class $\mathcal C$ of finite relational structures is the age of some  relational structure  if and only if $\mathcal C$  is an ideal of $\Omega_s$. We recall that a \emph{bound} of a hereditary class $\mathcal C$ of finite relational structures (e.g. graphs, ordered sets) is any relational structure $R\not \in \mathcal C$ such that every proper induced substructure of $R$ belongs to $\mathcal C$.
For a wealth of information on these notions see \cite{fraissetr}.

\subsubsection{Graphs}Unless otherwise stated, the graphs we consider are undirected, simple and have no loops. That is, a {\it graph} is a
pair $G:=(V, E)$, where $E$ is a subset of $[V]^2$, the set of $2$-element subsets of $V$. Elements of $V$ are the {\it vertices} of
$G$ and elements of $ E$ its {\it edges}. The {\it complement} of $G$ is the graph $\overline{G}$ whose vertex set is $V$ and edge set
${\overline { E}}:=[V]^2\setminus  E$. If $A$ is a subset of $V$, the pair $G_{\restriction A}:=(A,  E\cap [A]^2)$ is the \emph{graph
induced by $G$ on $A$}.
A \emph{path} is a graph $\mathrm P$ such that there exists a one-to-one map $f$ from the set $V(\mathrm P)$ of its vertices into an
interval $I$ of the chain $\ZZ$ of integers in such a way that $\{u,v\}$ belongs to $E(\mathrm P)$, the set of edges of $\mathrm P$,  if and only if $|f(u)-f(v)|=1$ for every $u,v\in V(\mathrm P)$.  If $I=\{1,\ldots,n\}$, then we denote that path by $\mathrm P_n$; its \emph{length} is $n-1$ (so, if $n=2$, $\mathrm P_2$ is made of a single edge, whereas if $n=1$, $\mathrm P_1$ is a single vertex. %A graph is $\mathrm P_n$-\emph{free} if it has no $\mathrm P_n$ as an induced subgraph. %A \emph{cycle} is obtained from a finite path $\mathrm P_n:= \{v_{1},...,v_{n}\}$ by adding the edge $\{v_{n},v_{1}\}$ and $n\geq 2$. The integer $n$ is called the \emph{length} of the cycle.

% If $G:= (V, E)$ is  a graph, and $x,y$ are two vertices of $G$, we denote by $d_G(x,y)$ the length of the shortest path joining $x$ and $y$ if any, and $d_G:= \infty$ otherwise. This defines a distance on $V$, the \emph{graphic distance}.
%The \emph{diameter} of $G$, denoted by  $\delta_{G}$, is the supremum of the set of $d_G(x,y)$ for $x,y\in V$. If $A$ is a subset of $V$, the graph $G'$ induced by $G$ on $A$  is an \emph{isometric subgraph} of $G$  if $d_{G'}(x,y)=d_G(x,y)$ for all $x,y\in A$. The supremum of the length of induced finite paths of $G$, denoted by $D_G$,  is sometimes called the \emph{detour} of $G$ \cite{buckley-harary}. We denote by $D_G(x,y)$ the supremum of the length of the induced path joining $x$ to $y$. Evidently, $d_G(x,y)\leq D_G(x,y)$. In Proposition \ref{prop:oscillation} of Subsection \ref{posetswidth2} we give an upper bound for $D_G$, when $G$ is a permutation graph.
%
\subsubsection{Posets} \label{subsubsection:posets} Throughout, $P :=(V, \leq)$ denotes an ordered set (poset), that is
a set $V$ equipped with a binary relation $\leq$ on $V$ which is
reflexive, antisymmetric and transitive. We say that two elements $x,y\in V$ are \emph{comparable} if $x\leq y$ or $y\leq x$, otherwise,  we say they are \emph{incomparable}. The \emph{dual} of $P$ denoted $P^{*}$ is the order defined on $V$ as follows: if $x,y\in
V$, then $x\leq y$ in $P^{*}$ if and only if $y\leq x$ in $P$. %Let $P :=(V, \leq)$ be a poset. A set of pairwise comparable elements is called a \emph{chain}. On the other hand, a set of pairwise incomparable elements is called an \emph{antichain}. The \emph{width} of a poset is the maximum cardinality  of its antichains (if the maximum does not exist, the width is set to be infinite).

%Dilworth's celebrated theorem on finite posets \cite{dilworth} states that the maximum cardinality of an antichain in a finite poset equals the minimum number of chains needed to cover the poset. This result remains true even if the poset is infinite but has finite width. If  the poset $P$ has width $2$ and the incomparability graph of $P$ is connected, the partition of $P$ into two chains is unique (picking  any vertex $x$,  observe that the set of  vertices at odd distance from $x$ and the set of vertices at even distance from $x$ form a partition into two chains). This paper is essentially about those posets that we call coverable by two chains  instead of having width two.

 According to  Szpilrajn \cite{szp},  every order $\leq$ on  a set  $V$ has a \emph{linear extension}, that is a linear (or total) order $\preceq$ on the $V$ such that $x\preceq y$ whenever $x\leq y$, for all $x,y\in V$. Let $P:=(V,\leq)$ be a poset.  A \emph{realizer} of $P$ is a family $\mathcal{L}$ of linear extensions of the order of $P$ whose intersection is the order of $P$. Observe that the set of all linear extensions of $P$ is a realizer of $P$. The \emph{dimension} of $P$, denoted $dim(P)$, is the least cardinal $d$ for which there exists a realizer of cardinality $d$ \cite{dushnik-miller}. It follows from the  Compactness Theorem of First Order Logic that an order is intersection of  at most $n$ linear orders ($n\in \NN$) if and only if every finite restriction of the order has this property. Hence, the class of posets with dimension at most $n$ is determined by a set of finite obstructions, each obstruction is a poset $Q$  of dimension $n+1$  such that  the deletion of any vertex of $Q$ leaves a poset of dimension $n$; such a poset is said \emph{critical}.  For $n\geq 2$ there are infinitely many critical posets of dimension $n+1$. For $n=2$,  critical posets of dimension three (and hence finite comparability graphs of critical posets  of dimension three) were characterized by Kelly \cite{kelly77}. Beyond, the task is considered as hopeless.

%We present in Figure \ref{fig:critique3} the critical bipartite posets of dimension $3$ extracted from  the list of Kelly. The comparability graph of the first one is commonly called a  \emph{spider}.
%\begin{figure}[h]
%\begin{center}
%\leavevmode \epsfxsize=3in \epsfbox{CriticalPosetsDim3}
%%scalebox{1}{
%%\input{fig4.pstex_t}
% %}
%\end{center}
%\caption{Critical bipartite posets  of dimension three.} \label{fig:critique3}
%\end{figure}
\subsubsection{Comparability and incomparability graphs}The \emph{comparability graph}, respectively the \emph{incomparability graph}, of a poset $P:=(V,\leq)$ is the graph, denoted by $\comp(P)$, respectively $\ainc(P)$, with vertex set $V$ and edges the pairs $\{u,v\}$ of comparable distinct vertices (that is, either $u< v$ or $v<u$) respectively incomparable vertices. A graph $G:= (V, E)$ is a \emph{comparability graph} if the edge set is the set of comparabilities of some order on $V$. From the Compactness Theorem of First Order Logic, it follows that a graph is a comparability graph if and only if  every finite induced subgraph is a comparability graph. Hence, the class of comparability graphs is determined by a set of finite obstructions. The complete list of minimal obstructions was determined by Gallai \cite{gallai} (see \cite{maffray} for an English translation). The list can also be found in \cite{trotter-moore} Figures 4(a) and 4(b).

\subsubsection{Permutation graphs} A graph $G:= (V, E)$ is a \emph{permutation graph} if there is a linear order $\leq $ on $V$ and a permutation $\sigma$ of $V$ such that the edges of $G$ are the pairs  $\{x, y\}\in [V]^2$ which are reversed by $\sigma$.

Denoting by $\leq_{\sigma}$ the set of oriented pairs $(x, y)$ such that $\sigma(x) \leq \sigma (y)$, the graph is the comparability graph of the poset whose order is the intersection of $\leq$ and the dual of $\leq_{\sigma}$.  Hence, a permutation graph is the comparability graph of  an order intersection of two linear orders, that is the comparability graph of an order of dimension at most two \cite{dushnik-miller}. The converse  holds  if the graph is finite. As it is well known, a finite graph $G$ is a permutation graph if and only if $G$ and $\overline G$ are comparability graphs \cite{dushnik-miller}; in particular, a finite graph  is a permutation graph if and only if its complement is a permutation graph.

 The comparability graph of  an infinite order which is intersection of two linear orders is not necessarily a permutation graph.
A one way infinite path is a permutation graph, but the complement of this infinite path is not a permutation graph. There are examples of infinite posets which are intersection of two linear orders and whose comparability and incomparability graphs  are not  permutation graphs. For an example see Figure \ref{fig:omega-z}.
However, via the Compactness Theorem of First Order Logic, an infinite graph is the comparability graph  of  a poset intersection of two linear orders if an only if each finite induced subgraph is a permutation graph (sometimes these graphs are called permutation graphs, while there is no possible permutation involved).
For more about permutation graphs, see \cite{klazar}, \cite{vatter}.
\subsubsection{Initial segment, ideal} An \emph{initial segment} of a
poset $P:= (V, \leq)$ is any subset $I$ of $V$ such that $x\in V$, $y\in I$ and $x\leq y$ imply $x\in I$. An \emph{ideal} is
any nonempty  initial segment $J$ of $P$ which is up-directed (that is $x, y\in J$ implies  $x,y\leq z$ for some $z\in J$). If $X$ is a
subset of $V$, the set $\downarrow X:=\{y\in V: y\leq x \; \text{for some}\; x\in X\}$ is the least initial segment containing $X$, we
say that it is \emph{generated} by $X$. If $X$ is a singleton, say $X=\{x\}$, we denote by $\downarrow x$, instead of $\downarrow X$,
this initial segment and say that it is \emph{principal}.
We denote by $\mathbf {I}(P)$, resp. $\mathbf {Id}(P)$,  the set of initial segments, respectively ideals,  of $P$, ordered by set inclusion.
%\emph{Final segments} and \emph{filters} of $P$ are defined as initial segments and ideals of $P^{*}$.

\subsection{Well-quasi-order}\label{section:w.q.o. -thm2-thm6}
We present the notion of well-quasi-order and introduce the notion of better-quasi-order; we refer to \cite{milner}.
A poset is \emph{well-founded} if every nonempty subset has some minimal element. Such a poset has a \emph{level decomposition} $(P_{\alpha})_{\alpha< h(P)}$ indexed by ordinal numbers. Level $P_{\alpha}$ is the set of minimal elements of $P\setminus \bigcup \{P_{\beta}: \beta< \alpha\}$ and $h(P)$, the \emph{height} of $P$, is  the least ordinal $\alpha$ such that $P_{\alpha}= \emptyset$. The poset is \emph{level-finite} if each level $P_{\alpha}$ is finite.   A quasi-ordered-set  (quoset) $Q$  is \emph{well-quasi-ordered} (w.q.o.), if every infinite sequence of elements of $Q$ contains an infinite increasing subsequence. If $Q$ is an ordered set, this amounts to say that every nonempty subset of $Q$ contains finitely many minimal elements (this number being non zero). Equivalently, $Q$ is w.q.o. if and only if it contains no infinite descending chain and no infinite antichain.
\subsubsection{Better-quasi-order}\label{subsubsection:bqo}  Proofs that some classes of countable  structures  are w.q.o.   under embeddability may require a strengthening of  that notion, e.g; the notion of \emph{better-quasi-order} (b.q.o) (see Subsection \ref{subsection:finitelymanyprimes}). We just recall that   b.q.o.'s are w.q.o.s.  As  for w.q.o.'s, finite sets and  well-ordered sets are b.q.o.'s,   finite unions, finite products, subsets and images of b.q.o.s by order preserving maps are b.q.o.'s.  (see \cite{fraissetr} for more).  Nash-Williams 1965 \cite{nashwilliams1} p.700,  asserted  that "one  is inclined to  conjecture  that  most  w.q.o. sets which  arise in a reasonably  'natural'  manner  are likely to be b.q.o." It is not known if the answer is positive for hereditary classes of finite graphs. The first classes to consider are probably those which are minimal prime. Due to the description of these classes,  the answer is positive.%
\subsubsection{Labelled classes} Among classes of structures which are w.q.o. under the embeddability quasi-order some remain  w.q.o. when the structures are labelled by the elements of a quasi-order. Precisely, let $\mathcal{C}$ be a class of relational structures, e.g.,  graphs, posets, etc., and $Q$ be  a quasi-ordered set or a poset. If $R\in \mathcal C$, a \emph{labelling of $R$ by $Q$} is any map $f$ from  the domain of $R$ into $Q$. Let   $\mathcal{C}\cdot Q$ denotes the collection of $(R,f)$ where $R\in \mathcal{C}$ and $f: R\rightarrow Q$ is  a labelling. This class is quasi-ordered  by $(R,f)\leq (R',f')$ if there exists an embedding $h: R\rightarrow R'$ such that $f(x)\leq (f'\circ h)(x)$ for all $x\in R$. We say that $\mathcal{C}$ is \emph{very well-quasi-ordered} (vw.q.o. for short) if for every finite $Q$, the class $\mathcal{C}\cdot Q$ is w.q.o. The class $\mathcal{C}$ is \emph{hereditary w.q.o.} if $\mathcal{C}\cdot Q$ is w.q.o. for every w.q.o. $Q$. The class $\mathcal{C}$ is $n$-w.q.o. if for every $n$-element poset $Q$, the poset $\mathcal{C}\cdot Q$ is w.q.o.  The class $\mathcal{C}$ is $n^{-}$-w.q.o. if the class $\mathcal C_{n^-}$ of $(R,a_1,\ldots,a_n)$ where $R\in \mathcal {C}$ and $a_1,\ldots,a_n\in R$ is w.q.o.

We do not know if these four notions are different. In the case of posets covered by two chains (that is of width at most two) we proved that they are identical \cite{pouzet-zaguia-wqo-2022}.

We will use the notion of hereditary well-quasi-ordering  in  Theorems \ref{thm:finite-prime} and  \ref{thm:main2} and the notion of $1^-$-well-quasi-ordering in Lemma \ref{lem:noyau1}.
We recall the following result (Proposition 2.2 of \cite{pouzet72}).

\begin{theorem}\label{thm:bounds} Provided that the signature $s$ is bounded, the cardinality of bounds of  every hereditary and hereditary w.q.o. subclass of  $\Omega_{s}$ is bounded.
\end{theorem}

%The notion of better-quasi-ordering is a tool for proving that classes of infinite structures are w.q.o.   The operational definition of b.q.o. is not intuitive. Since we are not going to  prove properties of   b.q.o.s, we use  the following  definition,  based on the idea of labelling considered above. Let $Q$ be a quasi-ordered set and $Q^{<\omega_1}$,  the set of maps $f: \alpha \rightarrow Q$, where $\alpha$ is any countable ordinal. If $f$ and $g$ are two such maps, we set $f\leq g$ if there is a one-to-one preserving map $h$ from the domain $\alpha$ into the domain $\beta $ of $g$ such that $f(\gamma)\leq g(h((\gamma))$ for all $\gamma< \alpha$. This relation is  a quasi-order; the quasi-ordered set  $Q$ is a \emph{better-quasi-order} if $Q^{<\omega_1}$ is w.q.o. B.q.o.'s are w.q.o.s. Indeed, according to Higman's Theorem on words, a set $Q$ is w.q.o., if and only if  the set $Q^{<\omega}$ of finite sequences of elements of $Q$ is w.q.o. As  w.q.o.'s, finite sets and  well-ordered sets are b.q.o.'s (do not try to prove it using  the  definition of b.q.o. given here), finite unions, finite products, subsets and images of b.q.o.s are b.q.o.'s.  But, contrarily to w.q.o.'s,  if $Q$ is b.q.o. then the set $\mathbf{I} (Q)$ of initial segments of $Q$  is b.q.o. Instead of labelling ordinals, we may label chains and compare them as above.
%
\subsubsection{J\'onsson posets}
\begin{definition}
A poset $P$ is a \emph{J\'onsson} poset if it is infinite and every proper initial segment has a strictly smaller cardinality than $P$.
\end{definition}

J\'onsson posets were introduced by Oman and Kearnes \cite{kearnes}. Countable J\'onsson posets were studied and described in \cite{pouzettr, pouzet-sauer, assous-pouzet}. We recall (see Proposition 3.1 \cite{assous-pouzet}):

\begin{theorem}\label{minimalposet}Let $P$ be a countable poset. The following propositions are equivalent.
\begin{enumerate}[$(i)$]
  \item $P$ is J\'onsson;
  \item $P$ is well-quasi-ordered and each ideal distinct from $P$ is finite;
  \item $P$ is level-finite,  has height  $\omega$, and  for each $n<\omega$, there is  $m<\omega$ such that each element of height at most $n$ is below every element of height at least $m$.
\end{enumerate}
\end{theorem}

\begin{lemma} \label{lem:contains minimal}Every infinite well-founded poset $P$ which is level finite contains an  initial segment which is J\'onsson.
\end{lemma}
\begin{proof}
We  apply Zorn's Lemma to the set  $\mathcal J$ of infinite initial segments of $P$ included in the first $\omega$-levels. For that, we prove that $\mathcal J$ is closed under intersections of nonempty chains. Indeed,  let $\mathcal C$ be a nonempty chain (with respect to set inclusion) of  members of $\mathcal J$. Set $J:= \cap\, \mathcal C$. Let $n<\omega$, let $P_n$ be the $n$-th level of $P$ and $\mathcal {C}_n:=  \{C\cap P_n: C \in \mathcal C\}$. The members of $\mathcal {C}_n$ are finite, nonempty and linearly ordered by set inclusion. Hence,  $J_n:=\cap\, \mathcal {C}_n$ is nonempty. Since $J=\cup \{ J_n: n\in \NN\}$, $J\in \mathcal J$.
\end{proof}

J\'onsson posets are behind  the study of minimal prime hereditary classes (See Theorem \ref{thm:minimalprime} in Section \ref{section:min-hered-class}).

\subsection{Words}Let $\Sigma$ be a finite set. A  $\Sigma$-sequence is any map $u$ from an interval  $I$  of the set $\ZZ$ of integers in $\Sigma$. The set $I$ is the \emph{domain} of $u$. Two $\Sigma$-sequences $u$ and $u'$ are \emph{isomorphic} if there is a translation $t$ on $\ZZ$ mapping the domain $I$ of $u$ onto the domain $I'$ of $u'$ so that $u(i)= u'(t(i))$ for all $i\in I$.  If the domain of a $\Sigma$-sequence $u$ is $\{0, \dots, n-1\}$,  $\NN$, $\NN^*:= \{0, -1, \ldots, -n \ldots \}$ or $\ZZ$, the sequence is a \emph{word}. Words appear as representatives of equivalence classes  of sequences. Except if their  domain is $\ZZ$,   the representatives are unique. The elements of $\Sigma$ are called \emph{letters} and $\Sigma$ is the \emph{alphabet}. When the  alphabet is  $\{0, 1\}$, we use the terminology  $0$-$1$ sequences or $0$-$1$ word. If $u$ is a $0$-$1$ sequence with domain $I$ and  if $I'$ is a subset of $I$, the \emph{restriction} of $u$ to $I'$ is denoted by  $u_{\restriction I'}$. If $I$ is finite, the sequence $u$ is finite and  the \emph{length} of $u$, denoted $\vert u \vert$, is the number of elements of $I$. We denote by $\Box$ the empty sequence. If $u$ is  a finite word and $v$ is a word, finite or infinite with domain $\NN$, the \emph{concatenation} of $u$ and $v$ is the word $uv$ obtained by writing $v$ after $u$. If $v$ has domain $\NN^*$, the word $vu$ is defined similarly. A word $v$ is a \emph{factor} of $u$ if $u= u_1vu_2$. This defines an order on the collection $\Sigma^*$ of finite words, the \emph{factor ordering}.
\subsubsection{Hereditary classes of words}
A subset $\mathcal C$ of $\Sigma^{*}$ is \emph{hereditary} if it contains every factor of every member of $\mathcal C$. In other words, this is an initial segment of $\Sigma^*$ ordered with the factor ordering. The \emph{age} of a word $u$ is the set $\fac (u)$ of all its finite factors endowed with the factor ordering. This is a hereditary subset of $\Sigma^{*}$. In fact,   if the alphabet is at most countable,  a set $\mathcal C$ of finite words is the age of a word $u$ if and only if  $\mathcal C$ is an \emph{ideal}  for the factor ordering. Note that the domain of $u$ is not necessarily $\NN$.

A  nonempty subset $\mathcal C$ of $\Sigma^{*}$   is \emph{inexhaustible} if it is not reduced to the empty word  and if for every $v\in \mathcal C$ there is some $w$ such that $vwv\in \mathcal C$.

%\begin{lemma}\label{lem:rec-class-age}Let $\mathcal{C}$ be a hereditary class of finite words.
%\begin{enumerate}[$(1)$]
%
%  \item $\mathcal{C}$ is the union of the maximal ideals included in it.
%  \item $\mathcal{C}$ is inexhaustible if and only if $\mathcal{C}$ is an union of inexhaustible ideals.
%\end{enumerate}
%\end{lemma}

\begin{lemma}\label{lem:rec-class-age}A hereditary class $\mathcal C$ of finite words
 is inexhaustible if and only if $\mathcal{C}$ is an union of inexhaustible ages.

\end{lemma}
\begin{proof}
$\Rightarrow$
\begin{claim}\label{claim1} If $\mathcal{D}$ is an inexhaustible subset of $\Sigma^*$  then $\downarrow \mathcal{D}$  is inexhaustible.
\end{claim}

\noindent{\bf Proof of Claim \ref{claim1}.}  Let $u\in \downarrow \mathcal{D}$. There exists $u'\in D$ such that $u$ is a factor of $u'$. We write $u':= u'_1uu'_2$. Since $\mathcal{D}$ is inexhaustible  there exists $w'$ such that $u'w'u'\in D$. Let $w:= u'_2w'u'_1$. The word $uwu$ is a factor of $u'w'u'$ hence is in $\downarrow \mathcal{D}$. \hfill                  $\Box$
\begin{claim}\label{claim2}  If $\mathcal{D}$ is   an inexhaustible subset of $\Sigma^*$, then for all $u\in \mathcal{D}$ there is a sequence $u_0, \ldots u_n\ldots $ satisfying   $u_0= u$ , $u_{n+1}:= u_n v_n u_n$, $u_n\in \mathcal{D}$ and $u_{n+1}\in \mathcal{D}$. By construction, the set $E:= \{u_n: n\in \NN\}$ is an  inexhaustible set.
\end{claim}
\begin{claim}\label{claim3} If $\mathcal{C}$ is inexhaustible and hereditary, then for all $u\in \mathcal{C}$ there exists a inexhaustible age $\mathcal{A}$  such that $u\in \mathcal{A} \subseteq \mathcal{C}$. \end{claim}
\noindent {\bf Proof of Claim \ref{claim3}.} We apply Claim \ref{claim2} with $\mathcal D:= \mathcal C$. The set $E$ defined in Claim \ref{claim2} is inexhaustible. The set   $\downarrow E$ is the  age of the word $u_{\infty}$ having the words $u_n$ as prefixes for $n\geq 0$.

$\Leftarrow$ Obvious. \end{proof}

% If an age is inexhaustible this is the age of a word on $\NN$ (see Theorem \ref{thm:rec-age-word}).

 A  word $u$  is \emph{recurrent} if every finite factor occurs infinitely often.  This amounts to the fact that $\fac(u)$ is inexhaustible. In fact:

%A nonempty set of words is an age if it is hereditary with respect to taking factors and if it is \emph{up-directed}. An age of words $\mathcal{A}$ is \emph{minimal} if it is infinite and every proper hereditary subclass is finite.
%\begin{definition}
%An set of words $\mathcal{A}$ is \emph{recurrent} if for all $\beta, \beta'\in \mathcal{A}$ there exists $\gamma\in \mathcal{A}$ such that $\beta+\gamma+\beta'\in \mathcal{A}$.
%\end{definition}

%The following result is Proposition 2 in section II-2.2 page 37 of \cite{pouzetetat}. This result was never published. For completeness we include its proof here.
%
%\begin{proposition}An age of words $\mathcal{A}$ is recurrent if and only if for all $\beta\in \mathcal{A}$ there exists $\gamma\in \mathcal{A}$ such that $\beta+\gamma+\beta\in \mathcal{A}$.
%\end{proposition}
%\begin{proof} The condition is necessary. We prove that it is sufficient. Let $\beta$ and $\beta'$ in $\mathcal{A}$. Since $\mathcal{A}$ is up-directed there exists $\delta$ in $\mathcal{A}$ such that $\beta$ and $\beta'$ are factors of $\delta$. There exists $\gamma$ such that $\gamma \delta \gamma$ is in $\mathcal{A}$. The word $\delta$ can be written as $\delta_1 \beta \delta_2$ and as $\delta'_1 \beta' \delta'_2$. Hence, $\mathcal{A}$ which is hereditary, contains $\beta \gamma_2 \delta \gamma'_1 \beta'$.
%\end{proof}
%
%

% This result was never published. For completeness we include its proof here.

\begin{theorem}\label{thm:rec-age-word} Let $\mu$ be a nonempty $0$-$1$ sequence on an interval of $\ZZ$. The following are equivalent.
\begin{enumerate}[$(i)$]
\item $\fac(\mu)$ is inexhaustible;
\item $\mu$ is recurrent;
\item There exists a word $\nu$ on $\ZZ$ so that
$\fac(\mu) =\fac (\nu)=\fac (\nu_{\restriction \NN})=\fac (\nu_{\restriction \NN^*})$.
\end{enumerate}
\end{theorem}

This is a variation of  Proposition 2 in section II-2.3 page 40 of \cite{pouzettr}. For results along this lines, see \cite{beauquier-nivat}.

%\begin{theorem}Let $\mu$ be a uniformly recurrent $0$-$1$ sequence. Then $\fac(\mu)$ is recurrent and minimal.
%\end{theorem}
A word $u$  is \emph{uniformly recurrent} if for every $n\in \NN$ there exists $m\in \NN$ such that each factor $u(p)\ldots u(p+n)$ of length $n$ occurs as a factor of every factor of length $m$.
%   , that is a non-empty set of finite words which is an \emph{initial segment} (that is if $v\in \mathcal A$, then every factor of $v$ is also in $\mathcal A$) and  is up-directed (that is if $v,v'\in \mathcal A$, then there exists $v''\in \mathcal A$ such that $v$ and $v'$ are factors of $v''$).

%There is a relation between uniformly recurrent words and well-quasi-ordering via the notion of J\'onsson poset.
%An ordered set (poset) $P$ is a \emph{J\'onsson} poset if it is infinite and every proper initial segment has a strictly smaller cardinality than $P$. J\'onsson posets were introduced by Oman and Kearnes \cite{kearnes}. Countable J\'onsson posets were studied and described in \cite{pouzetthesis, pouzet-sauer, assous-pouzet}. In particular, a countable poset $P$ is J\'onsson if and only if it is well-quasi-ordered, has height  $\omega$, and  for each $n<\omega$, there is  $m<\omega$ such that each element of
%height at most $n$ is below every element of height at least $m$.

%Their appearance in the domain of symbolic dynamics is due to the following result.
 The fact that a word is uniformly recurrent can be expressed in terms of properties of its age ordered by the factor ordering.
\begin{theorem}\label{unif-recurrent} Let $u$ be a word with domain $\NN$ over a finite alphabet. The following properties are equivalent:
\begin{enumerate}[$(i)$]
\item $u$ is uniformly recurrent;
\item $\fac (u)$ is inexhaustible and well-quasi-ordered;
\item $\fac(u)$  is a countable J\'onsson poset.
\end{enumerate}
\end{theorem}
The only nontrivial implication is $(ii)\Rightarrow (iii)$ (see Lemma II-2.5 of Ages belordonn\'es in \cite{pouzettr} p. 47). See Theorem 5 in \cite{pouzet-zaguia-dmtcs21}.

Let $u$ and $v$ be two words.  The word  $v$ is a \emph{prefix} of $u$ if $u=vu'$. This is a \emph{suffix} of $v$ if $v= u''v$. These relations  define two orders on the collection $\Sigma^*$ of finite words, the \emph{prefix} and the \emph{suffix} orders. The prefix order, as well as the suffix order are (ordered) trees (for each $u\in \Sigma^*$, the set of elements below is a chain). A first consequence is that every ideal is a chain. A more significant consequence, is that  if  $\mathcal C$ is an initial segment of $\Sigma^*$ for one of these orders, $\mathcal C$ is  a  w.q.o. if and only if it is an union of finitely many chains (indeed, if $\mathcal C$ is a w.q.o. then as every w.q.o.  this is a finite union of ideals (see \cite{milner}).  From this we deduce:

\begin{theorem}\label{thm:ulti-periodic} Let $u$ be a word with domain $\NN$ over a finite alphabet. Then $\fac(u)$ is well-quasi-ordered for the prefix order, respectively, the  suffix order,  if and only if $u$ is ultimately periodic, respectively, periodic.
\end{theorem}

\begin{proof}Suppose that $u$ is w.q.o. for the prefix order.  Then $\fac(u)$ is a finite union of ideals. For each integer $n\in \NN$, the set $\pref_n (u):= \{u_{\restriction [n, m[}: n\leq m\}$ is an ideal. This ideal been included in a finite union of initial segments is included in one of them, and in fact equal. Thus there are only finitely many sets of the form $\pref_n (u)$. If for $n<n'$, $\pref_n(u)=\pref_{n'}(u)$ then the sequences $\pref_n (u)$  and $\pref_{n'} (u)$ give the same word. It follows that  $u$ is ultimately periodic.  If  $u$ is w.q.o. for the suffix order, we observe first that $u$ is recurrent. Then we apply Theorem \ref {thm:rec-age-word}: there is a word $w$ on $\ZZ$ such that $\fac(w)= \fac(u)$. With the same argument as above,  the set  for $n\in \ZZ$ of $\suff_n (w):= \{w_{\restriction [m, n[}: m\leq n\}$  is finite. This ensures that $w$ is periodic. It follows that $u$ is periodic too.  The converses of these implications are obvious. \end{proof}

\subsubsection{Bounds of hereditary classes of words}

Let $\mathcal C$ be a hereditary class of finite words. A \emph{bound} of $\mathcal C$ is any finite word $v\not \in \mathcal C$ such that every proper factor of $v$ belongs to $\mathcal C$. Equivalently, if $v:=v_0\dots v_{n-1}$,  then $v$ is a bound of $\mathcal{C}$ if and only if $v\not \in \mathcal C$ and the words $v_0\dots v_{n-2}$ and $v_1\dots v_{n-2}$ belong to $\mathcal C$.

Let $u$ be a word and  $p$ be a nonnegative integer. The word  $u$ is \emph{periodic} and $p$ is a \emph{period} if $u(i)= u(i+p)$ whenever  $i$ and $i+p$ belong to the domain of $u$.

The following result is Proposition 3 in section II-2.6 page 54 of \cite{pouzettr}. This result was never published. For completeness we include its proof here.

\begin{theorem}\label{thm:periodic-bounds}Let $\mu$ be an infinite  periodic word of period $p>0$. Then the bounds of $\fac(\mu)$ have length at most $p$.
\end{theorem}
\begin{proof} The proof is based on the following remark due to Roland Assous. Namely, a word $v$ is periodic, with  period  $p>0$, if and only if every two factors  $w'$ and $w''$ of $v$, both having length $p$, contain the same letters and each of these letters occur the same number of times in $w'$ and $w''$.\\
We now prove the theorem. Let $v$ be a finite word of length at least $p+1$ so that each factor of length $p$ is a factor of the word $\mu$.  It follows from   the above remark  that the word $v$ is periodic with period $p$,  hence $v$ is of the form $w\dots ww'= (wn)w'$ where $w'$ is a prefix of $w$. Since $\mu$ is periodic of period $p$ and $w$ is a factor of $\mu$ we infer that  $w(n+1)$ is  a factor of $\mu$, hence  $v$ is a  factor of $\mu$. Thus $v$ is not a bound of $\fac(\mu)$.
\end{proof}

%Let $p$ be a given positive integer. A word $\beta$ of length $q$, with $q\geq 1$, can be written $\beta=\gamma\delta\gamma\delta\ldots \gamma\delta \gamma=((\gamma\delta) n)\gamma$ so that $|\gamma|<p$ and $|\gamma\delta|=p$ if and only if every two subwords $\beta'$ and $\beta''$ of $\beta$, of length $p$, contain the same letters and each of these letters occur the same number of times in $\beta'$ and $\beta''$.\\

The following is a consequence of Proposition 6 in II-2.6. page 60 of  \cite{pouzettr}. This result was never published. For completeness we include a  proof (in fact two) here.

\begin{theorem}\label{thm:inf-bound-uniform-seq}Let $\mu$ be a uniformly recurrent and non periodic word. Then $\fac(\mu)$ has infinitely many bounds.
\end{theorem}
\begin{proof} We give two proofs.\\
1) Let $\mu$ be a uniformly recurrent and non periodic word. Suppose for a contradiction that $\fac(\mu)$ has finitely many bounds and let $m$ be the maximum length  of bounds of $\fac(\mu)$. Let $v\in \fac(\mu)$ of length at least $m$. Since $\mu$ is uniformly recurrent, $\fac(\mu)$ is inexhaustible (Theorem \ref{unif-recurrent}) it contains a word of the form $vwv$. Since the bounds of $\fac(\mu)$ have lengths at most $m$ we infer that $\fac(\mu)$ contains all periodic words of the form $(vw)\ldots (vw)$ and hence contains the set of factors of the infinite periodic word $\mu':=(vw)vw\ldots$. Since $\fac(\mu)$ is J\'onsson (Theorem \ref{unif-recurrent}) and $\fac(\mu')$ is infinite,  $\fac(\mu)=\fac(\mu')$. It follows  that $\mu$ is periodic. A contradiction.\\

\noindent 2) Our second  proof is based on the following properties of regular languages.

\noindent {\bf Claim}. \label{lem:reg}Let $\mathcal{C}$ be a hereditary class of finite words (ordered by the factor relation). If $\mathcal{C}$  has a finite number of bounds, then $\mathcal{C}$ is a regular language. And if $\mathcal C$ is an infinite regular language it contains $\fac(\mu)$ where $\mu$ is an infinite periodic word.

 \noindent{\bf Proof of the Claim.}
The first implication is immediate. Indeed, for every finite word $v$, the set $\uparrow v:=\{w : v\;  \mbox{is a factor of } w\}$ is regular language (in fact $\uparrow v=\Sigma^*v\Sigma^*$, where $\Sigma$ is the alphabet). It follows from Kleene's Theorem that the complement of $\uparrow v$, that is $\mathcal{C}\setminus \uparrow v$ is a regular language. If $\mathcal{C}$ has a finite number of bounds, then
\[\mathcal{C}=\bigcup \{\Sigma^* \setminus \uparrow v : v \mbox{ bound of } \mathcal{C}\}\]
is a finite intersection of regular languages and is therefore regular. \\For the second implication we use the Pumping Lemma for regular languages \cite{Bar-Hillel-Perles-Shamir}. Since $\mathcal{C}$ is a regular language, there are finite words $u,v_1,v_2$ such that for all $n\in \NN$, $v_1u^nv_2\in \mathcal{C}$. Let $\mu:=uuuu\ldots$. Then  $\fac(\mu)\subseteq \mathcal{C}$.
\hfill $\Box$

The existence of infinitely many bounds to a non periodic and uniformly recurrent word follows from the Claim. Indeed, let $\mu$ be a non periodic uniformly recurrent word. Then $\fac(\mu)$ cannot be regular. Otherwise,  it follows from the second part of the Claim that $\fac(\mu)$ contains the set of factor of an infinite periodic word $w$. Since $\mu$ is uniformly recurrent $\fac(\mu)=\fac(w)$ and therefore $\mu$ is periodic, contradicting our assumption. The required conclusion now follows from the first part of the Claim.
\end{proof}

 \section{Minimal prime hereditary classes}\label{section:min-hered-class}

 In this section we present the definition and properties of minimal prime hereditary classes of finite binary structures. We introduce  the notion of minimal prime structure and we conclude with the notion of almost chainability.  Results for graphs, given in the subsequent sections, are more precise. Most of the results presented here were included in Chapter 5 of the thesis of the first author \cite{oudrar}.

We start  with the  notion of a module.

\begin{definition}\label{def:module}
Let  $R:=(V,(\rho_i)_{i\in I})$ be a binary relational structure. A  \emph{module} of $R$ is any subset $A$ of $V$ such that  $$(x\rho_i a \Leftrightarrow x\rho_i a')  \; \text{and} \; (a\rho_i x \Leftrightarrow a'\rho_i x) \; \text{for all} \; a,a'\in A \;\text{and}\;
x\notin A \; \text{and} \; i\in I.$$
\end{definition}

The empty set, the singletons in $V$ and the whole set $V$ are modules and are called \textit{trivial}. (sometimes in the literature, modules  are called \emph{interval},
 \emph{autonomous} or \emph{partitive sets}). If $R$ has no nontrivial module, it is called \emph{prime} or \textit{indecomposable}.

\noindent For example, if $R:=(V,\leq)$ is a chain, its modules  are the ordinary intervals of the chain. If $R:=(V,(\leq,\leq'))$ is a bichain then $A$ is a  module of $R$ if and only if $A$ is an  interval of $(V,\leq)$ and $(V,\leq')$.

The notion of module  goes back to Fra\"{\i}ss\'e  \cite{fraisse3} and Gallai  \cite{gallai}, see also \cite{fraisse84}. A fundamental decomposition result of a binary structure into modules was obtained by Gallai \cite{gallai} for finite binary relations (see \cite{ehren} for further extensions).
We recall the  compactness result of Ille \cite{ille}.

\begin{theorem}\label{ille-theorem}
A binary structure $R$  is prime if and only if every finite subset $F$ of its domain extends to a finite set $F'$ such that  $R_{\restriction F'}$ is prime.
\end{theorem}

We consider the class $\prim_{s}:=\prim (\Omega_s)$ of finite  binary structures of signature $s$ which are prime. We set $\prim (\mathcal C):= \prim_s\cap \, \mathcal C$ for every $\mathcal C\subseteq \Omega_s$.

We say that a subclass $\mathcal{D}$ of $\prim_{s}$ is \emph{hereditary} if it contains every member of $\prim_{s}$  which can be embedded into some member of $\mathcal{D}$.

\subsection{Hereditary classes containing finitely many prime structures}\label{subsection:finitelymanyprimes}

The following result (see Proposition 5.2  of \cite{oudrar-pouzet2016}) improves  a result of  \cite{albert-atkinson} for  hereditary classes of finite permutations.

\begin{theorem}\label{thm:finite-prime}Let $\mathcal {C}$ be a hereditary class of finite binary structures containing only finitely many prime structures. Then $\mathcal {C}$ is hereditarily well-quasi-ordered. In particular, $\mathcal {C}$ has finitely many bounds.
\end{theorem}

The following result, due independently to  Delhomm\'e \cite{delhomme1} and McKay
 \cite{mckay1} extends Thomass\'e's result on the well-quasi-order   character of the class of countable series-parallel posets \cite{thomasse},  which extends the  famous Laver's theorem \cite{laver} on the  well-quasi-order  character of the class of countable chains.

\begin{theorem}\label{thm:delhomme-mckay}
Let $\mathcal C$ be a hereditary class of $\Omega_{s}$. If $\prim (\mathcal C)$ is finite, then  the collection  ${\mathcal C^{\leq \omega}}$   of countable $R$ such that $\age(R)\subseteq \mathcal C$ is well-quasi-ordered  by embeddability.
\end{theorem}

In fact, Delhomm\'e and McKay obtain a  stronger conclusion of  Theorem \ref{thm:delhomme-mckay}. If $\prim(\mathcal C)$ is finite, and $Q$ is a better-quasi-order then, the class of members of ${\mathcal C^{\leq \omega}}$ labelled by $Q$ is better-quasi-ordered (this implication is false if b.q.o. is replaced by w.q.o). In particular, if $Q$ is finite, this class is w.q.o. This case follows from Theorem \ref {thm:delhomme-mckay} above. Indeed, we may view structures labelled by $Q$ as binary structures. In this new class, say $\mathcal D$, modules are unchanged, hence there are only finitely many primes and thus the class  $\mathcal D^{\leq \omega} $ is w.q.o. We will use this observation in the proof of Theorem \ref{thm:noninexhaustible} below.

\subsection{Hereditary classes containing infinitely many prime structures}\label{subsection:infinitely-prime}

In this subsection, we report some results included in \cite{oudrar}. We consider hereditary classes containing infinitely many prime structures. We show that each  such a class contains one which is minimal with respect to inclusion.

\begin{definition} A hereditary class $\mathcal C$ of $\Omega_{\mu}$ is \emph{minimal prime} if it contains infinitely many prime structures, while every  proper  hereditary subclass contains only finitely  many prime structures.
\end{definition}
This notion appears in the thesis of the first author \cite{oudrar} (see Theorem 5.12, p. 92, and Theorem 5.15, p. 94 of  \cite{oudrar}).

Due to their definition, minimal prime ages ordered by set inclusion form an antichain with respect to set inclusion.

%An ordered set $P$ is \emph{J\'{o}nsson} if $P$ is infinite and the cardinality of every proper initial segment of $P$ is strictly less than the cardinality of $P$ \cite{kearnes}. We say that $P$ is \emph{minimal} if it is infinite and every proper initial segment of $P$ is finite. This amounts to say that $P$ is a countable J\'{o}nsson poset. These posets appear quite naturally in symbolic dynamic. In fact,  an infinite word $u$ on some finite alphabet $A$ is uniformly recurrent (\cite{All-Sha}, \cite{lothaire}) if and only if the set $Fac(u)$ of its finite factors is minimal once it is ordered with the factor order.
%
% We list below some equivalent properties  see Proposition 4.1 of \cite{pouzet-sauer}, or Proposition 3.1 of \cite{assous-pouzet}.
%
%\begin{theorem} \label{minimalposet} Let $P$ be an
%infinite  poset. Then, the following properties are  equivalent:
%\begin{enumerate}[{(i)}]
%\item Every proper initial segment of $P$ is finite.
%\item $P$ is w.q.o.  and all ideals distinct from $P$ are principal;
%\item $P$ has no infinite antichain and all ideals distinct from $P$ are finite;
%%\item Every linear extension of $P$ has order type $\omega$.
%\item $P$ is level-finite, of height  $\omega$, and
% for each $n<\omega$ there is  $m<\omega$ such that each element of
%height at most $n$ is below every element of height at least $m$.
% \end{enumerate}
%\end {theorem}

We have immediately (cf. Th\'eor\`eme 5.14 p.93 of \cite{oudrar}).

\begin{theorem} \label{thm:minimalprime}A hereditary class $\mathcal C$ of $\Omega_s$ is minimal prime if and only if $\prim  (\mathcal C)$
is a J\'onsson poset which is cofinal in $\mathcal C$.
\end{theorem}

\begin{proof} Let $\mathcal  C$ be a minimal prime class. By definition, $\prim  (\mathcal C)$ is infinite.
Let $\mathcal I$ be a proper hereditary subclass of $\prim  (\mathcal C)$. The initial segment $\downarrow \mathcal I$ in $\Omega_s$ is a proper subclass of $\mathcal C$. Hence $\mathcal I$ is finite. Thus $\prim (\mathcal C)$ is J\'onsson. Let $\mathcal C':= \downarrow \prim  (\mathcal C)$. If $\mathcal C'\not =\mathcal C$ then since $\mathcal C$ is minimal prime, $\prim  (\mathcal C')= \prim  (\mathcal C)$ is finite, which is impossible. This proves that the forward implication holds

 Conversely,  suppose that $\prim  (\mathcal C)$ is a J\'onsson poset which is cofinal in $\mathcal C$. Then
$\mathcal C$ is  infinite. If $\mathcal C$ is not  minimal prime there is a	proper hereditary subclass $\mathcal C'$  of $\mathcal C$ such that $\prim  (\mathcal C')$ is infinite. Since $\prim (\mathcal C)$  is J\'onsson, $\prim  (\mathcal C')= \prim  (\mathcal C)$. Since $\mathcal C'= \downarrow \prim  (\mathcal C')$ and $\prim  (\mathcal C)$ is cofinal in $\mathcal C$, this yields $\mathcal C'= \mathcal C$.  A contradiction.
\end{proof}

We have:

\begin{theorem}\label {minimal}
Every   hereditary class of   finite binary  structures (with a given finite signature), which contains  infinitely many prime structures   contains a minimal prime hereditary subclass.
\end{theorem}

 For the proof of Theorem $\ref{minimal}$ we will need the following lemma which is a special case of Theorem 4.6 of \cite{assous-pouzet}.

\begin{lemma}\label{lem:levelfinite-prims}$\prim_{s}$ is level finite.
\end{lemma}
\begin{proof}Suppose for a contradiction that there exists an integer $n\geq 0$ such that the level ${\prim_s}(n)$ of $\prim_s$ is infinite and choose $n$ smallest with this property. Define
$$\mathcal{C}:=\{R\in \Omega_{s}: R< S \text{ for some } S\in {\prim}_{s}(n)\}.$$
Then $\mathcal{C}$ is a hereditary class of $\Omega_{s}$ containing only finitely many prime structures. It follows from Theorem \ref{thm:finite-prime} that $\mathcal{Cer}$ is hereditary well-quasi-ordered and hence has finitely many bounds. This is not possible since the elements of ${\prim_s}(n)$ are  bounds of $\mathcal{C}$.
\end{proof}

The proof of Theorem $\ref{minimal}$ goes as follows. Let $\mathcal C$ be a hereditary class of $\Omega_s$ such that $J:= \prim_s(\mathcal C)$ is infinite. Since $\prim_{s}$ is level finite,   Lemma \ref {lem:contains minimal} ensures that $J$ contains an initial segment $D$ which is J\'onsson. According to Theorem \ref{thm:minimalprime}, $\downarrow D$ is minimal prime. This completes the proof.\hfill $\Box$

\subsubsection{Another proof of Lemma \ref{lem:levelfinite-prims}}
We prove the finiteness of the levels of $\prim_{s}$ via the properties of critical primality. A binary structure $R$ is \textit{critically prime} if it is prime and   $R_{\restriction V(R)\setminus \{x\}}$ is not prime for every $x\in V(R)$.  Note that $\vert V(R)\vert$ has at least four elements.
This notion of critical primality was introduced by Schmerl and Trotter \cite{S-T}. Among results given in their paper, we have the following theorem (this is Theorem 5.9, page 204):

\begin{theorem}\label{theo:indec}
 Let $R$ be a prime binary structure of order $n\geqslant 7$. Then there are distinct $c, d\in V(R)$ such that $R_{\restriction V(R) \setminus \{c,d\}}$ is prime.
\end{theorem}

%\begin{corollary}\label{cor:indec}
% Suppose $R$ is a prime structure of order $n$ which is not critically prime, and suppose $5\leqslant m\leqslant n.$  Then $R$ has an prime substructure of order $m.$
%\end{corollary}
%
In their paper, Schmerl and Trotter give examples of critically prime structures within the class of graphs, posets, tournaments, oriented graphs and binary relational structures. The set of critical prime structures within each of these classes  is a finite union of chains.
%If $\mathcal C$ is a hereditary subclass of $\Omega_s$, let $\mathcal{C}:=Ind(\Omega_s)$ be the set of prime members of $\Omega_\mu$.

Decompose $Prim_s$  into levels;  in level $i$, with $i\leq 2$,
are the structures of order zero, one or two.

For structures $R$ in $Prim_s$ of order at least $2$, we have the following  relationship between the height $h(R)$ in $Prim_s$   and its order, $\vert V(R) \vert$ (which is the height of $R$ in $\Omega_s$).

\begin{equation} \label{ineq:1}
    h(R) \leqslant \vert V(R) \vert  \leqslant 2(h(R)-1).
      \end{equation}

The first inequality is obvious. For the second, we use induction on $n:= h(R)\geq 2$. The basis step $n=2$ is trivially true. Suppose $n>2$. Let $S$ be prime such that $S$ embeds in $R$ with $h(S)=n-1$. From the induction hypothesis, $\vert V(S)\vert \leq 2(h(S)-1)= 2(n-2)$. According to Theorem \ref {theo:indec}, $\vert V(R)\vert -2 \leq \vert V(S) \vert $. Hence $\vert V(R)\vert -2 \leq 2(n-2)$. Therefore $\vert V(R)\vert \leq 2(n-1)$.

Lemma \ref{lem:levelfinite-prims} follows from the second inequality in (\ref{ineq:1}) since there are only finitely many structures of a given order.

With Theorem \ref{ille-theorem} and \ref{minimal}, one gets:
\begin{corollary} The age of any infinite prime structure contains a minimal prime age.
\end{corollary}

%we   recall that a class $\mathcal C$ of graph  is \emph{well-quasi-ordered} if every sequence $G_0, \ldots, G_n, \ldots$ contains an increasing subsequence $G_{n_0}\leq  \ldots, \leq G_{n_k}, \ldots$ with respect to embeddability. If the class $\mathcal C$ does  not contain infinite descending chains, this amounts to the nonexistence of  infinite antichains.

With  Lemma \ref{lem:contains minimal} and Theorem \ref{thm:finite-prime} we get:

\begin{theorem} \label{thm:main1} Every  minimal prime hereditary class is the age of some prime structure; furthermore this age is well-quasi-ordered.
\end{theorem}

\begin{proof}\label{sec:proof-thm:main1}  Let $\mathcal{C}$ be a minimal prime hereditary class. We first prove that it is the age of a prime structure. It follows from Theorem \ref{thm:minimalprime} that $\mathcal{C}=\downarrow D$ where $D$ is J\'{o}nsson. Since $D$ is J\'{o}nsson, it is up-directed. Thus $\mathcal{C}$ is an age. Since $D$ is up-directed and countable,  it contains a cofinal sequence $R_0\leq R_1\leq \ldots <R_n\leq \ldots$. We may define the limit $R$ of these $R_n$. Since the $R_n$'s are prime,  $R$ is prime and $\age(R)=\mathcal{C}$.

Next we prove that $\mathcal{C}$ is w.q.o.  Since $D$ is J\'{o}nsson,  it is w.q.o. To prove that $\mathcal{C}$ is w.q.o., let $R\in \mathcal{C}$ and consider $\mathcal{C}\setminus (\uparrow \{R\}$). In order to prove that $\mathcal{C}$ is w.q.o.  it is enough to prove that $\mathcal{C}\setminus (\uparrow \{R\}$)  is w.q.o.  by embeddability. Indeed, an antichain that contains $R$ must be in $\mathcal{C}\setminus (\uparrow \{R\}$). Now to prove that $\mathcal{C}\setminus (\uparrow \{R\}$) is w.q.o.  we note that since $\mathcal{C}\setminus (\uparrow \{R\}$) is a proper hereditary class in $\mathcal{C}$,  it contains only finitely many primes. It follows from Theorem \ref{thm:finite-prime}  that $\mathcal{C}\setminus (\uparrow \{R\})$  is w.q.o. . \end{proof}

As mentioned in subsection \ref{subsubsection:bqo},  it is not known if a hereditary class of finite graphs which is w.q.o. is b.q.o.

\begin{problem}
Is every minimal prime hereditary class of finite binary structures b.q.o.?
\end{problem}

It is known that J\'onsson posets are b.q.o. \cite{pouzettr, carroy-pequignot} but the argument in the proof  of Theorem \ref{thm:main1}  does no give the b.q.o. character of the class $\mathcal C$.
In the case of graphs, minimal prime hereditary classes  divide into two types. Those which are almost multichainable and those which are the ages of some special graphs. The b.q.o. character of these classes can be obtained by an extension of Higman's theorem to b.q.o.
(see Remark \ref{rem:b.q.o.}).
We give below an improvement  of Theorem \ref{thm:main1} based on properties  of the  kernel of a relational structure.

\subsection{Inexhaustibility, kernel and minimality}

The \emph{kernel} of  a relational structure $R$ with domain $V$ is the set
\[\ker(R): \{x\in V : \age(R_ {\restriction {V\setminus \{x\}}})\not = \age (R)\}.\]
The kernel is an invariant of the age in the sense that if $R$ and $R'$ have the same age then there is an isomorphism $f$ from $\ker(R)$ onto $\ker(R')$ such that $(a)$ every restriction of $f$ to every finite subset $F$ of $\ker(R)$ extends to every finite superset  $\overline F$ of $F$ to an embedding of $R_{\restriction \overline F}$ in $R'$  and $(b)$ the same property holds for $f^{-1}$.
An age $\mathcal A$ is \emph{inexhaustible},  or has the  \emph{disjoint embedding property}, if two arbitrary members of the age can be embedded into a third member in such a way that their domains are disjoint. As it is easy to see, the kernel of a relational structure $R$ is empty if and only if $\age (R)$ is inexhaustible. We say that an age $\mathcal C$ which is not inexhaustible is \emph{exhaustible}. It is \emph{almost inexhaustible} if the kernel of some $\emph{R} $ with age $\age (R)= \mathcal C$ is finite.

The notion of inexhaustibility was introduced by Fra\"{\i}ss\'e in the sixties.
The  notion of kernel was  introduced in \cite{pouzettr} and studied in several papers \cite{pouzet-minimale}, \cite{pouzet.81}, and \cite{pouzet-impartible} (see Lemme IV-3.1 p. 37), first for structures with finite signature. The general case was considered in \cite{pouzet-sobranisa}.

We prove:
\begin{theorem} \label{thm:main2}If $\mathcal C$ is a minimal prime class of binary structures,  then $\mathcal C$ is almost inexhaustible.
\end{theorem}

In order to prove Theorem \ref {thm:main2} we recall two facts below. The first one is in  \cite{pouzet-minimale} see III.1.3, p. 323.
%Let $a, b\in V(R)$. Set $a\leq_R b$ if $\age (R_{-a}) \leq \age(R_{-b})$.

\begin{lemma}\label{lem:0} An element $a\in V(R)$ belongs to $\ker (R)$ if and only if  there is some finite subset $A$ of $V(R)$ containing all the images of $a$ by  the local automorphisms defined on $A$. \end{lemma}

We extract the second  fact from \cite {pouzettr} Corollaire p.6 in  "Caract\'erisation combinatoire et topologique des \^ages les plus simples". For reader's convenience, we give a proof.

\begin{lemma}\label{lem:noyau1}Let $R:= (V, (\rho_i)_{i\in I})$ be a relational structure made of finitely many binary relations. If $\age (R)_{1^-}:= \{(S, a): S\in \age \mathcal (R),  a\in V(S) \}$ is well-quasi-ordered,  then $\ker(R)$ is finite.
\end{lemma}
\begin{proof}Suppose that $\ker (R)$ is infinite. We built a sequence  $(R_{\restriction A_n}, a_n)$ of elements  of $\age (R)_{1^-}$ such that no two members of the sequence have a common extension belonging to $\age (R)_{1^-}$.  In particular these  members form an infinite antichain of  $\age (R)_{1^-}$.   We pick  $a_0\in \ker(R)$ and $A_0$ given by Lemma \ref{lem:0}. Suppose $(A_n, a_n)$ defined for $n<m$,  pick $a_m\in V\setminus \cup \bigcup_{n<m} A_m$, select $A$ given by Lemma \ref{lem:0} and set $A_{m}:= A\cup \bigcup_{n<m} A_n$.  \end{proof}

%\begin{lemma}\label{lem:2}
%Let $R$ be a relational structure,  $K' \subseteq \ker (R)$ and $R'$ be a relational structure  on   $V':= V \setminus \ker(R)$. Suppose that the local isomorphisms of $R'$ are those of $R$ extendable by the identity on $K'$. If $\ker (R)$ is finite, then  $\ker (R')= \emptyset$.
%\end{lemma}
%
%\begin{proof}
%Let $a'\in V'$ and $A'$ be a finite subset of $V'$ containing $a'$. Since $\ker (R)$ is finite there is a	 local isomorphism of $R$ defined on  $\ker (R)\cup A'$  fixing $\ker (R)$ pointwise and sending $A'$ in $V'\setminus \{a'\}$.
%See my papers. Theorem 4.5. of
%\end{proof}
Let $\mathcal C$ be a class of finite binary structures $\mathcal S: = (F, (\rho_i)_{i\in I})$ with a finite signature $s$. Denote by $\mathcal C^{+1}$ the class of $S: = (F, (\rho_i)_{i\in I})$ such that there is some $a\in F$ such that $S_{\restriction F\setminus \{a\}}\in \mathcal C$.

The following   lemma is Proposition 5.32 p. 105 of \cite{oudrar} and Theorem 4.5 page 20 of \cite{brignal-vatter}. A similar fact, but non explicit, appears in the proof of Theorem 4.24 p.267 of \cite{pouzet.2006}. For reader's convenience,  we give a proof.

\begin{lemma}\label{lem:extensionwqo1} Let $\mathcal C$ be a hereditary class of binary structures. If the members of $\mathcal C$ are not necessarily finite and  if these members  when labelled  by any better-quasi-order form a  well-quasi-order,   then $\mathcal C^{+1}$ has the same property. If $\mathcal C$ is made of finite structures and is hereditarily well-quasi-ordered,  then $\mathcal C^{+1}$ is hereditarily well-quasi-ordered.  \end{lemma}

\begin{proof}
Let $I$ be such that each $\mathcal S\in \mathcal C$  is of the form $\mathcal S: = (F, (\rho_i)_{i\in I})$. Let W be a w.q.o. By hypothesis, the set $(2\times 2\times 2)^{I}$  is finite, hence with the equality ordering it is wqo.  The direct product $W'$  of $W$ with  $(2\times 2\times 2)^{I}$  is w.q.o.  We code members of $\mathcal C^{+1}$ labelled by $W$ by members of $\mathcal C$ labelled by $W'$. Indeed, for each $\mathcal S: = (F, (\rho_i)_{i\in I})\in \mathcal C^{+1}$ we select $a\in F$ such that $\mathcal S_{\restriction F\setminus \{a\}}\in \mathcal C$ and we label $\mathcal S_{\restriction F\setminus \{a\}}$ by the map $g_a$ defined for $x\in F\setminus \{a\}$ by $g_a(x):=(\rho_i(a,x), \rho_i(x,a), \rho_i(a, a))_{i\in I}$. Now if $f$ is a labelling of $F$  in $W$,  we  associate the labelling $f'$ of $F\setminus \{a\}$  by setting $f':= (f_{\restriction F\setminus \{a\}}, g_a)$. By construction, if $\mathcal S, \mathcal S'\in \mathcal C$, an embedding $h$ from the labelled structure  $\mathcal S_{\restriction F\setminus \{a\}}$ in the  labelled $\mathcal S'_{\restriction F'\setminus \{a'\}}$ will extend to  an embedding of the labelled structure  $\mathcal S$  in the labelled structure $\mathcal S'$ with $a$ mapped to $a'$. The conclusion follows. \end{proof}

We deduce:
\begin{corollary}\label {cor:extensionwqo}
 Let $R:= (V, (\rho_i)_{i\in I})$ be a relational structure made of finitely many binary relations and let $a\in V$. If $\age( R_{\restriction V\setminus \{a\}})$  is hereditarily well-quasi-ordered,  then $\age( R)$ is hereditarily well-quasi-ordered.  \end{corollary}

\begin{proof}
If $a\not \in \ker (R)$, there is nothing to prove. If $a\in \ker ( R)$, we set $\mathcal C:= \age( R_{\restriction V\setminus \{a\}})$. We observe that $\age (R) \subseteq \mathcal C^{+}$ and we apply Lemma \ref{lem:extensionwqo1}.
\end{proof}

%In Fra{\"\i}ss{\'e}'s terminology, ages with the disjoint embedding property are said \emph{inexhaustible} and relational structures whose age  is inexhaustible are said \emph{age-inexhaustible}; we say that relational structures with finite kernel are \emph{almost age-inexhaustible}. \footnote{In order to agree with Fra{\"\i}ss{\'e}'s terminology, we disagree with the terminology of our papers, in which inexhaustibility, resp. almost inexhaustibility, is used for relational structures with empty, resp. finite, kernel, rather than for their ages.}
\noindent{\bf Proof of Theorem \ref {thm:main2}.}
Let $\mathcal C$ be a minimal prime class and  $R$ such that $\age (R)=\mathcal C$. Suppose that $\ker (R)$ is nonempty. Let $a\in \ker (R)$. Then $\age(R_{\restriction V\setminus \{a\}})\not = \age (R)= \mathcal C$. Since $\mathcal C$ is minimal prime, $\age(R_{\restriction V\setminus \{a\}})$ contains only finitely many primes. Theorem \ref{thm:finite-prime} asserts that $\age (R_{\restriction V\setminus \{a\}})$ is hereditarily wqo. Corollary \ref{cor:extensionwqo} asserts that $\age(R)$ is hereditarily w.q.o. Lemma \ref{lem:noyau1}  asserts that $\ker (R)$ is finite. With that the proof is complete. \hfill $\Box$

%Let $\mathcal C$ be a minimal prime class. If $\mathcal C$ is noninexhaustible  then $\mathcal C$ is hereditarily well-quasi-ordered  and, in particular,  the kernel $\ker(\mathcal{R})$ of every $R$ with $\age(R)= \mathcal C$ is finite.

Since each hereditary well-quasi-ordered  class has finitely many bounds (Theorem \ref{thm:bounds}) we have only countably many exhaustible minimal prime classes.

\begin{corollary} There are at most countably many minimal prime classes $\mathcal C$  such that $\mathcal C$ is exhaustible.
\end{corollary}

\begin{problem}
\begin{enumerate}[$(1)$]
\item Is it true that $\vert \ker(\mathcal{R})\vert \leq 2$ if  $\age (R)$ minimal prime?
\item Is the number of exhaustible minimal prime ages finite?
\end{enumerate}
\end{problem}

As we will see, the  answers are positive if one considers minimal prime classes of graphs. In this case, there are only five examples with a nonempty kernel.
\subsection{Links with an other notion of minimality}

\begin{definition} A binary relational structure $R$ is \emph{minimal prime} if $R$ is prime and $R$ embeds in every induced indecomposable substructure with the same cardinality.  \end{definition}
Several examples of graphs and posets are given in \cite{pouzet-zaguia2009}.

\begin{problem} Is it true that  the age of a minimal prime binary structure is necessarily minimal prime? \end{problem}

Even in the case of graphs  we do not know the answer. The converse is false in the sense that there are minimal prime ages of graphs such that no graph with that age is minimal prime.

We prove:
\begin{theorem}\label{thm:noninexhaustible} If $\mathcal C$  is minimal prime and  exhaustible,  then every binary prime structure $R$ with $\age (R)=\mathcal C$ embeds a minimal prime  structure. \end{theorem}

The proof relies on Theorem  \ref{thm:delhomme-mckay} and Lemma \ref{lem:extensionwqo1}.

We prove first the following.

\begin{lemma}\label{lem:well-founded}
If $\mathcal C$ is minimal prime and  exhaustible then $\mathcal C^{\leq \omega} $ is well-quasi-ordered.
\end{lemma}
\begin{proof} Let $R$ with $\age (R)= \mathcal C$. Pick $a\in \ker (R)$. Let $\mathcal D:= \age( R_{\restriction V(R) \setminus \{a\}})$. This age contains only finitely many primes. From Theorem \ref{thm:delhomme-mckay}, $\mathcal D^{\leq \omega}$ is well-quasi-ordered. Furthermore,  members of $\mathcal D^{\leq \omega}$ when labelled by any  finite set form a well-quasi-ordered set. According to Lemma \ref{lem:extensionwqo1},  $(\mathcal D^{\leq \omega})^{+1} $ has the same property. Next,   $\mathcal C^{\leq \omega} \subseteq (\mathcal D^{\leq \omega})^{+1}$. Indeed, every member   of $\mathcal C^{\leq \omega}$ has  a copy  $R'$ in a countable  extension $R''$ of  $R$ having the same age as $R$ hence $\age (R'_{\restriction V(R') \setminus \{a\}})\subseteq \mathcal C$.    Hence,   $\mathcal  C^{\leq \omega} $ is well-quasi-ordered.
%According to Lemma $\age (\mathcal D)$ is hereditarily well quasi ordered. From Corollary \ref{cor:extensionwqo},  $\age (R)$ it is hereditarily well quasi ordered. In fact,
\end{proof}

Next,
\begin{lemma}\label{lem:well-founded-prime} Let $\mathcal C$ be a hereditary class of $\Omega_{s}$. If $\mathcal C^{\leq \omega}$ is well founded then every prime member of  $\mathcal C^{\leq \omega}$, if any, embeds  a minimal one.
\end{lemma}

\noindent{\bf Proof of Theorem \ref{thm:noninexhaustible}.} Let $R$ be a prime structure with $\age (R)=\mathcal C$. According to Lemma \ref{lem:well-founded},
 $\mathcal C^{\leq \omega} $ is well-quasi-ordered.  According to Lemma \ref{lem:well-founded-prime},  $R$   embeds  a minimal prime member.
\hfill $\Box$

\subsection{Primality and almost multichainability}
A relational structure $R$ is \emph{almost multichainable} if its domain $V$ is the disjoint union of a finite set $F$ and a set  $L\times K$ where  $K$ is a  finite set, for which there is a linear order $\leq$ on $L$, satisfying  the following condition:

 $\bullet$ For every local isomorphism $h$  of the chain $C:= (L, \leq)$  the map $(h, 1_K)$ extended by the identity on $F$ is a local isomorphism of $R$ (the map $(h, 1_K)$ is defined by $(h, 1_K)(x, y):= (h(x), y)$ ).

The notion of almost multichainability was introduced in \cite{pouzettr} (see \cite {pouzet.2006} for further references and discussions).  The special case $\vert K\vert =1$ is
the notion of almost chainability introduced by Fra\"{\i}ss\'e. The use of this notion in relation with the notion of primality is illustrated in several papers, notably \cite{boudabbous-pouzet-2010}, \cite{pouzet-sikaddour-zaguia-2005}.

%It seems to be a
%little bit hard to swallow, still it is the key in proving the second part of Theorem \ref{thm:poly-expo1} as well that Theorem \ref{theo:base}.

We recall $1.$  of Theorem 4.19 p.265 of \cite{pouzet.2006}.

\begin{proposition}\label{multichainablewqo}The age of an almost multichainable structure  is hereditarily well-quasi-ordered.
\end{proposition}

The proof of Proposition \ref{multichainablewqo} given in \cite {pouzet.2006}
 consists to interpret members of the age by words over a finite alphabet and apply Higman's Theorem on words. In fact, the extension of Higman's Theorem to b.q.o. tells us that the age of  an almost multichainable structure  is hereditarily b.q.o.

With Theorem \ref{thm:bounds}, we have:

\begin{theorem}
If the signature is bounded, the cardinality of bounds of the age of an almost multichainable structure is bounded.
\end{theorem}

 Proposition \ref{multichainablewqo} extends a little bit.

\begin{proposition}\label{multichainablebqo}If $\mathcal C$ is the age of a almost multichainable structure,  then the collection  $\mathcal C^{\leq\omega}$ of countable structures whose ages are included in $\mathcal C$ is b.q.o. and  in fact hereditary b.q.o. \end{proposition}

For a proof,  see \cite{pouzet-zaguia-wqo-2022}.

Applying Lemma \ref {lem:well-founded-prime} and Proposition \ref {multichainablebqo}, we have:

\begin{theorem}\label{thm:multichain}
If $R$ is almost multichainable,  then   $\age(R)$ is hereditarily well-quasi-ordered.  Hence, it has finitely many bounds.  Every prime $R'$ with the same age (if any) contains a minimal prime structure.
\end{theorem}

\begin{problem} If a minimal prime age is $2^{-}$-well-quasi-ordered,  this is the age of an almost multichainable binary relational structure.
\end{problem}

The answer is positive for graphs. Indeed, the minimal prime ages which are not ages of multichainable graphs are ages of some special graphs, the $G_{\mu}$'s, their ages are not $2^-$-w.q.o.  Some are $1^-$-w.q.o. (when $\mu$ is periodic). For more, see Remark \ref{rem:nonw.q.o.}.

\section{Minimal prime ages of graphs}\label{section:minimalprimesgraphs}

Our description  of minimal prime ages of graphs is based on several results. First a previous description of unavoidable prime graphs in large finite prime graphs of Chudnovsky, Kim, Oum and Seymour \cite{chudnovsky}, see also  Malliaris and Terry \cite {malliaris}.
Next  a  study of graphs associated to $0$-$1$ sequences.

\subsection{Unavoidable prime graphs.}

%We present first the results of \cite{chudnovsky}.

We  introduce some finite prime graphs. Fix an integer $n \geq 1$.

\begin{itemize}
  %\item A half-graph of height $n$ is a bipartite graph with $2n$ vertices $a_1,\ldots,a_n$, $b_1\dots,b_n$ such that $a_i$ is adjacent to $b_j$ if and only if $i \leq j$.

\item \emph{The bipartite half-graph of height $n$} $H_n$, is a graph with $2n$ vertices $a_1\ldots,a_n$, $b_1,\dots,b_n$ such
that $a_i$ is adjacent to $b_j$ if and only if $i \leq j$ and such that $\{a_1\dots,a_n\}$ and $\{b_1,\dots,b_n\}$ are
independent sets.

\item \emph{The half split graph of height $n$} $H'_n$, is a graph with $2n$ vertices  $a_1,\ldots,a_n$, $b_1\dots,b_n$ such that
$a_i$ is adjacent to $b_j$ if and only if $i\leq j$ and such that $\{a_1\dots,a_n\}$ is an independent set and
$\{b_1,\dots,b_n\}$ is a clique (a graph is a split graph if its vertices can be partitioned into a clique and an independent set).

\item Let $H'_{n,I}$ be the graph obtained from $H'_n$ by adding a new vertex adjacent to $a_1,\ldots,a_n$ (and no others). Let $H^*_n$  be the graph obtained from $H'_n$  by adding a new vertex adjacent to $a_1$ (and no others).

\item \emph{The thin spider with $n$ legs} is a graph with $2n$ vertices $a_1,\ldots,a_n$, $b_1\dots,b_n$ such that $\{a_1\dots,a_n\}$ is an independent set and $\{b_1,\dots,b_n\}$ is a clique, and $a_i$ is adjacent to $b_j$ if and only if $i=j$. The \emph{thick spider  with $n$ legs} is the complement of the thin spider with $n$ legs. In particular, it is a graph with $2n$ vertices $a_1\ldots,a_n$, $b_1,\dots,b_n$ such that $\{a_1\dots,a_n\}$ is an independent set $\{b_1,\dots,b_n\}$ is a clique, and $a_i$ is adjacent to $b_j$ if and only if $i\not =j$. A \emph{spider} is a thin spider or a thick spider. In Item (4) of Theorem \ref{thm:malliaris} we consider the extension of this notion to infinite sets.

\item A sequence of distinct vertices $v_0,\ldots,v_m$ in a graph $G$ is called a \emph{chain from a set $I \subseteq V ( G )$ to
$v_m$} if $m\geq 2$ is an integer, $v_0,v_1\in I$, $v_2,\ldots,v_m \not \in I$ , and for all $i>0$, $v_{i-1}$ is either the unique
neighbor or the unique non-neighbor of $v_i$ in $\{v_0,\dots,v_{i-1}\}$. The length of a chain $v_0,\dots,v_m$ is $m$.
\end{itemize}

\begin{figure}[h]
\begin{center}
\leavevmode \epsfxsize=4in \epsfbox{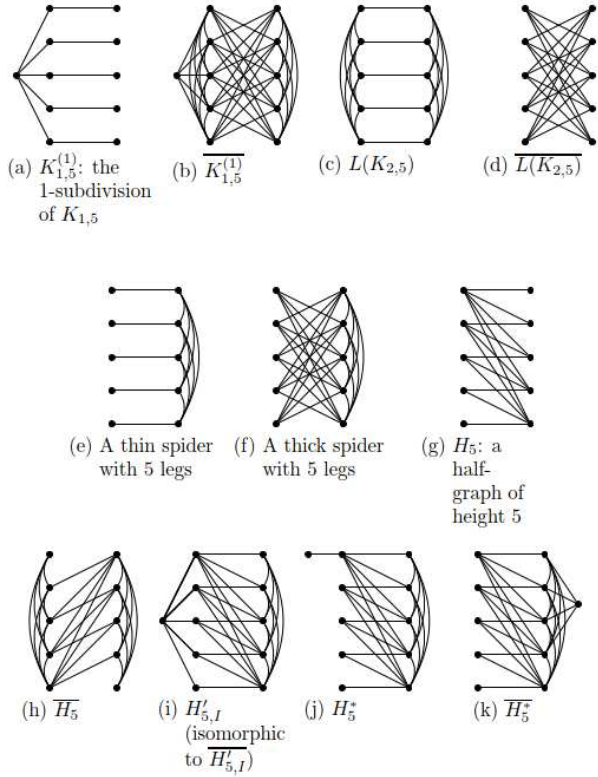}
\end{center}
\caption{Unavoidable prime finite graphs (this is Figure 1  from \cite{chudnovsky})}
\label{fig:liste-chudnowski}
\end{figure}

The following is due to Chudnovsky et al \cite{chudnovsky}:

\begin{theorem}[Theorem 1.2 of \cite{chudnovsky}]\label{thm:chudnovsky}
For every integer $n\geq 3$ there is $N$ such that every prime graph
with at least $N$ vertices contains one of the following graphs or their complements as an induced
subgraph.
\begin{enumerate}[$(1)$]
  \item  The 1-subdivision of $K_{1,n}$ (denoted by $K^1_{1,n}$ ).
  \item The line graph of $K_{2,n}$.
  \item The thin spider with $n$ legs.
  \item The bipartite half-graph of height $n$.
  \item The graph $H'_{n,I}$.
  \item the graph $H^*_n$.
  \item A prime graph induced by a chain of length $n$.
\end{enumerate}

\end{theorem}

% Note that the infinite one way path is the graph corresponding to the graph obtained from the constant $0$-$1$ word $\mu$ on $\NN$ defined by $\mu:=111\ldots$.

Malliaris and Terry prove in  \cite{malliaris} an infinitary  version of Theorem \ref{thm:chudnovsky} for infinite  graphs, then use it to prove Theorem \ref{thm:chudnovsky}. Their result is the following.
\begin{theorem}[Theorem 6.8 of \cite{malliaris}]\label{thm:malliaris}

 An infinite prime graph $G$ contains one of the following.
 \begin{enumerate}[$(1)$]
   \item Copies of $H_n$, $\overline{H_n}, H^*_n, \overline{H^*_n}, H'_{n,I}, \overline{H'_{n,I}}$  for arbitrarily large finite $n$,
   \item Prime graphs induced by arbitrarily long finite chains,
   \item $K^1_{1,\omega}$ or its complement,
   \item The line graph of $K_{2,\omega}$ or its complement,
   \item A spider with $\omega$ many legs.
   \end{enumerate}
\end{theorem}

The graphs mentioned in the last three items and some infinite versions of the graphs in Item 1 were considered in \cite{pouzet-zaguia2009}. In addition, the following
characterization of unavoidable infinite prime graphs without infinite clique (or infinite independent set) was given.

\begin{theorem}[Theorem 2 of \cite{pouzet-zaguia2009}]\label{thm:pouzet-zaguia}
An infinite prime graph which does not contain an infinite clique embeds one of the following:
\begin{enumerate}[$(1)$]
\item The bipartite half-graph of height $\omega$.
\item The infinite one way path.
\item The $1$-subdivision of $K_{1,\omega}$.
\item The complement of the line graph of $K_{2,\omega}$.\end{enumerate}
\end{theorem}

The graphs mentioned in  Theorem \ref{thm:pouzet-zaguia} are depicted in Figure \ref{fig:list-graph-min-1}.

\subsection{Eleven almost multichainable graphs and their ages}
Let $\mathcal M$ be the graphs $G_0, G_1, G_3, G_4$, $G_5$ and $G_6$ depicted in Figures \ref{fig:list-graph-min-1} and  \ref{fig:list-graph-min-2}. Let $\tilde{\mathcal M}$ be the list of these graphs and their complements. Let $\mathcal L$ be the set of the ages of these graphs and of their complements.  It should be noted that the graphs $G_5$, $\overline G_5$ have the same age, hence $\mathcal L$ has eleven  members.

%xxxxxxx
%\begin{figure}[ht]
%        \input{figures/figures/figure6-1.tex}
%        \caption{} \label{fig:list-graph-min-1}
%    \end{figure}
%
%\begin{figure}[ht]
%        \input{figures/figures/figure6-2.tex}
%        \caption{} \label{fig:list-graph-min-2}
%    \end{figure}

\begin{figure}[h]
\begin{center}
\leavevmode \epsfxsize=5in \epsfbox{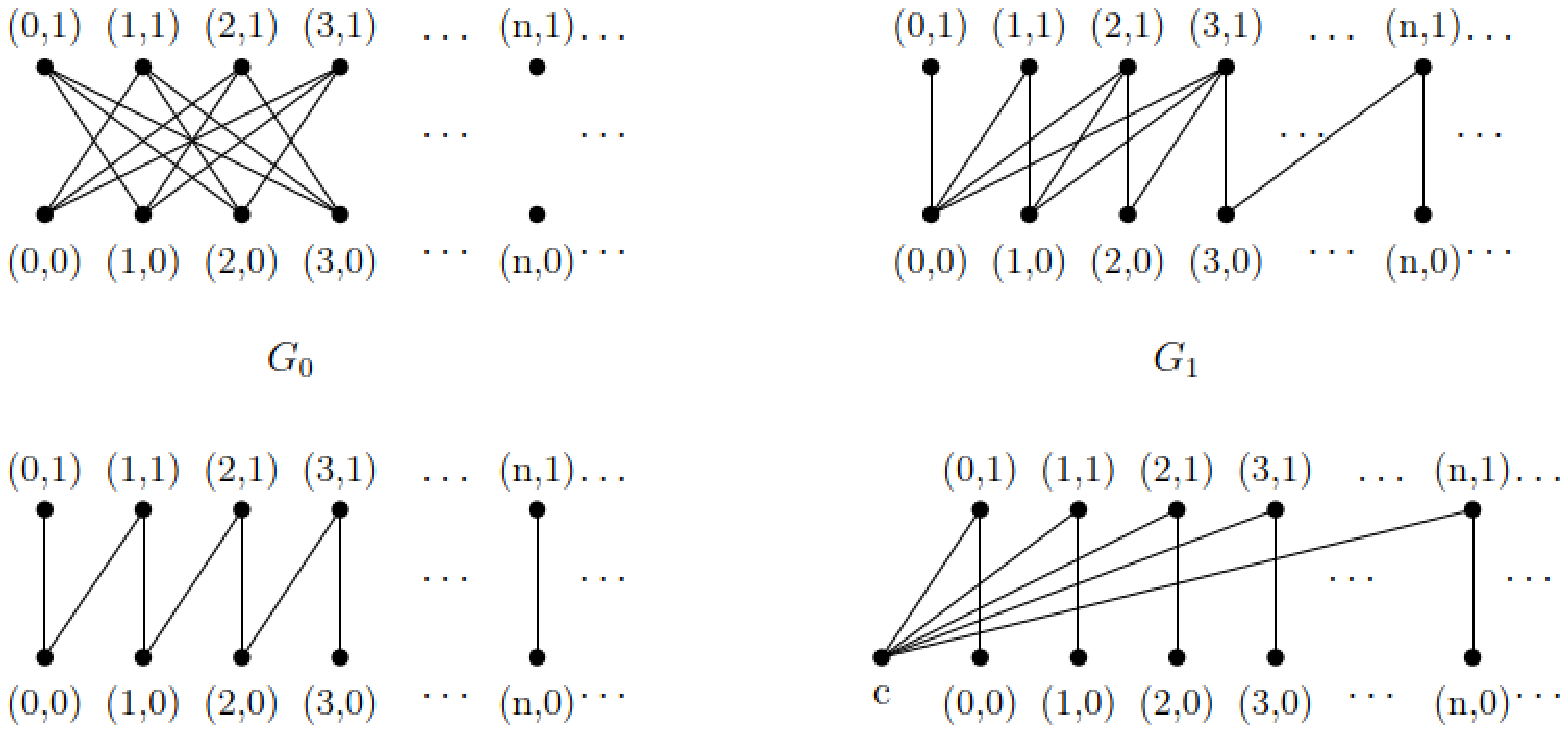}
\end{center}
\caption{}
\label{fig:list-graph-min-1}
\end{figure}

\begin{figure}[h]
\begin{center}
\leavevmode \epsfxsize=5in \epsfbox{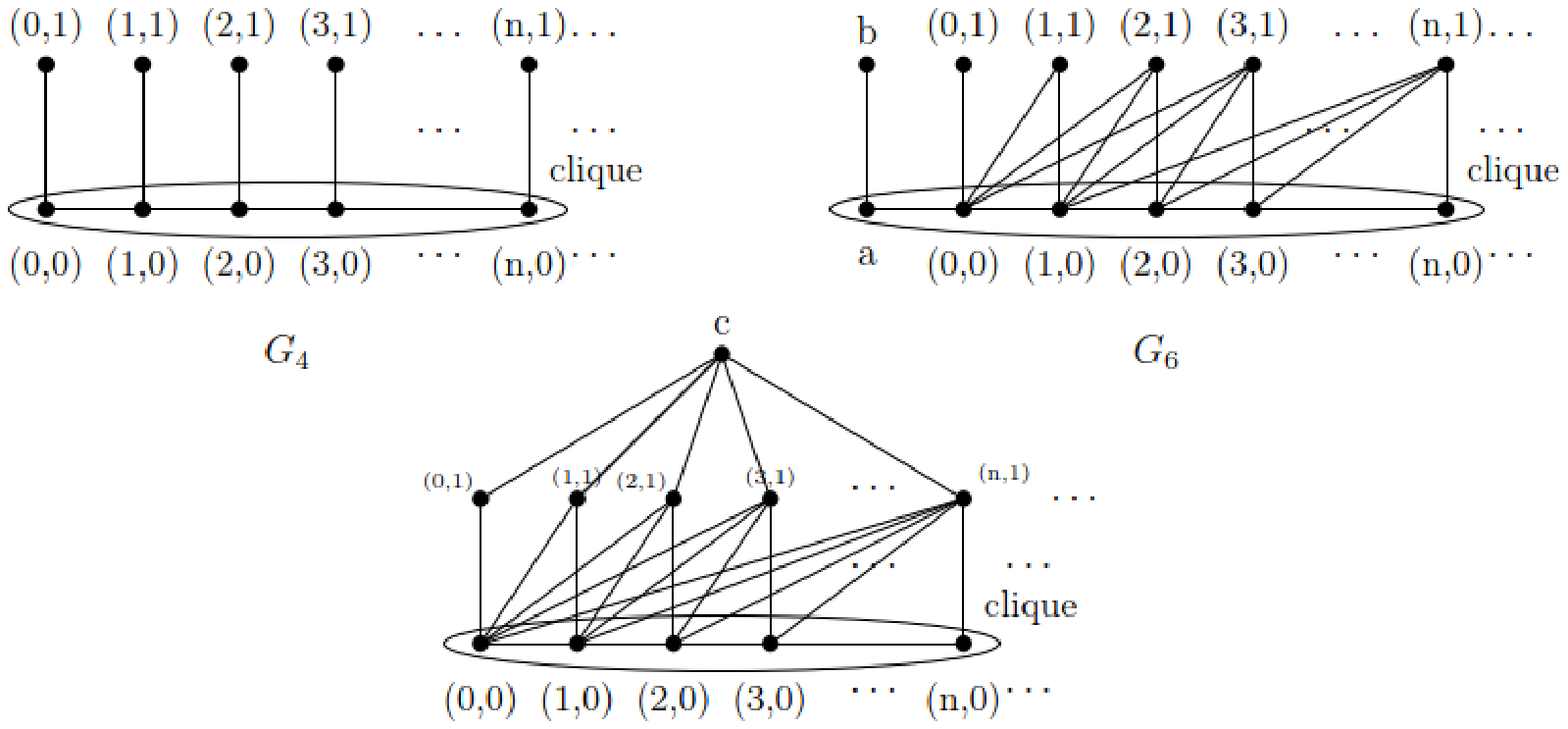}
\end{center}
\caption{}
\label{fig:list-graph-min-2}
\end{figure}

\begin{theorem}\label{thm:almostmultigraphs} Members of $\mathcal M$ and their complements are almost multichainable and minimal prime. Members of $\mathcal L$  are distinct  and minimal prime, five of them are exhaustible.
\end{theorem}
\begin{proof}

An inspection of the six members of $\mathcal M$ shows that $G_0$, $G_1$ and $G_4$ are multichainable with an empty kernel, the three others are almost multichainable with a one-element kernel, in the case of $G_3$ and $G_5$, and a two-element kernel in the case of $G_6$. This gives three exhaustible ages; with the ages of their complements added (and since $G_5$ and $\overline G_5$ have the same age) this gives five exhaustible ages. The fact that these graphs are minimal prime is given in \cite {pouzet-zaguia2009}.
The second part of the theorem, notably the   fact that the ages are distinct and minimal prime  is detailed in Chapter 6 page 109 of the first author's thesis \cite{oudrar}.
\end{proof}

The only prime graphs occurring in Theorem \ref{thm:chudnovsky} and \ref{thm:malliaris} and not in Theorem \ref {thm:almostmultigraphs} are chains. Chains can be represented  by   words on the alphabet $\{0,1\}$. They will give rise to uncountably many minimal prime ages. We study these graphs and their ages in the next subsection.

\subsection{Graphs associated to $0$-$1$ sequences.}\label{subsection:01graphs}

%In this text, a $0$-$1$ sequence is a map $\mu$  from an interval $I$ of the set $\ZZ$ of integers into $\{0,1\}$. The restriction of $\mu$ to an interval of $I$ is a \emph{factor} of $\mu$. For $i\in I$ we will sometimes denote $\mu(i)$ by $\mu_i$.  If $I$ is finite, with $n$ elements, we may view $\mu$  as a word  $\mu:=  u_0\ldots   u_{n-1}$. If the sequence is infinite we may view it as a $0$-$1$ sequence over $\NN$, over $\NN^{*}:= \{\ldots, -n, \ldots, -2, -1, 0\}$, or over $\ZZ$.
%
%

\begin{definition}To a word $\mu$ we associate the graph $G_{\mu}$ whose vertex set $V(G_\mu)$ is $\{-1,  0, \ldots,  n-1\}$  if  the domain of $\mu$ is $\{0, \ldots,  n-1\}$, $\{-1\} \cup \NN$ if the domain of $\mu$  is $\NN$, and $\NN^{*}$  or $\ZZ$ if the domain of $\mu$ is $\NN^{*}$  or $\ZZ$. For two vertices $i,j$ with $i<j$ we let $\{i,j\}$ be an edge of $G_{\mu}$ if and only if
\begin{align*}
  \mu_j=1 & \mbox{ and } j=i+1,\mbox{or} \\
  \mu_j=0 &\mbox{ and } j\neq i+1.
\end{align*}
\end{definition}

%Hence, $G_\mu(i-1,i)=\mu_i$ and for $i<j$, $G_\mu(i,j)=\mu_j \dot +1$.

For instance, if $\mu$ is the word defined on $\NN$ by setting $\mu_i=1$ for all $i\in \NN$, then  $G_\mu$ is the infinite one way path on $\{-1\}\cup \NN$. Note that if $\mu'$ is the word defined on $\NN$ by setting $\mu'_i=1$ for all $i\in \NN \setminus \{1\}$ and $\mu'_1=0$, then  $G_{\mu'}$ is also the infinite one way path. In particular,  the graphs $G_{\mu}$ and $G_{\mu'}$ have the same age but $\mu$ and $\mu'$ do not have the same sets of finite factors.

\begin{figure}[h]
\begin{center}
\leavevmode \epsfxsize=5in \epsfbox{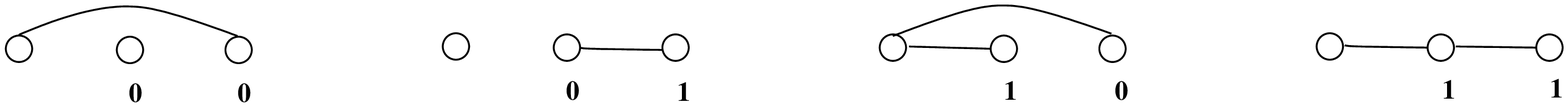}
\end{center}
\caption{$0$-$1$ words of length two and their corresponding graphs.} \label{fig:gmu-three}
\end{figure}

\begin{figure}[h]
\begin{center}
\leavevmode \epsfxsize=4in \epsfbox{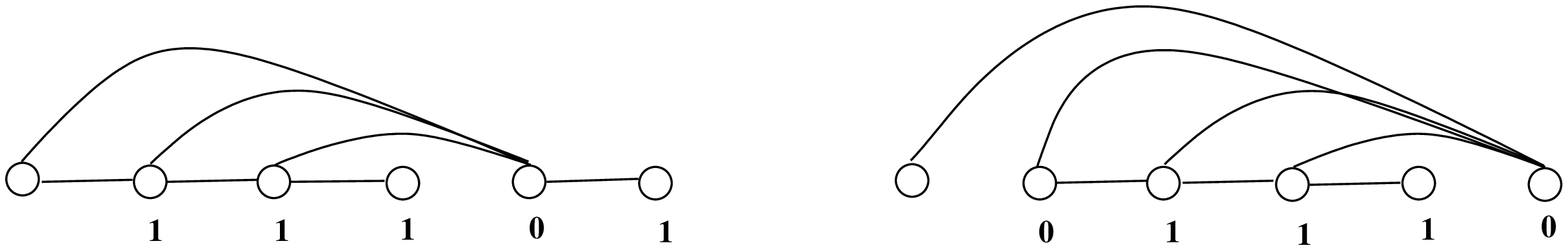}
\end{center}
\caption{Two distinct $0$-$1$ sequences with isomorphic corresponding graphs.} \label{fig:gmu-same-graphs}
\end{figure}

This correspondence between $0$-$1$ words and graphs was first considered in \cite{sobranithesis}, \cite{sobranietat}; see also  \cite{zverovich} and \cite{chudnovsky}.
\begin{remark}\label{lem:comp}
If $I$ is an interval of $\NN$ and $\mu:=(\mu_{i})_{i\in I}$ is a $0$-$1$ sequence, then $\overline{G_\mu}=G_{\overline{\mu}}$, where $\overline{\mu}:=(\overline{\mu}_{i})_{i\in I}$ is the $0$-$1$ sequence defined by $\overline{\mu}(i):=\mu(i)\dot +1$ and $\dot +$ is the addition modulo $2$.
\end{remark}

\begin{figure}[h]
\begin{center}
\leavevmode \epsfxsize=3in \epsfbox{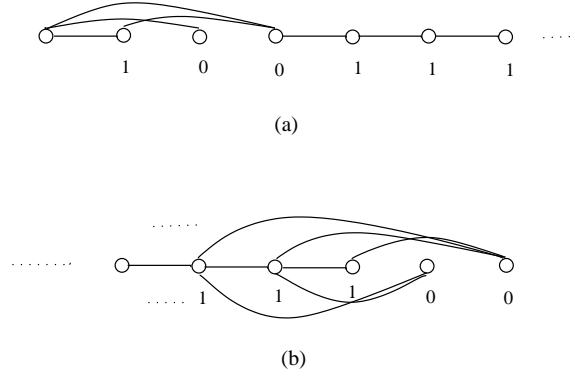}
\end{center}
\caption{$0$-$1$ graphs nonrealizable by a sequence on $\ZZ$.} \label{fig:gmu-nonrealizable}
\end{figure}

\begin{remark} Given a $0$-$1$ graph defined on $\NN\cup \{-1\}$ or on $\NN^*$ there does not exist necessarily a $0$-$1$ graph on $\ZZ$ with the same age.\\
\emph{Indeed,
$(a)$ Let $\mu:=100111\ldots$ be an infinite word on $\NN$ (the corresponding graph is depicted in (a) of  Figure $\ref{fig:gmu-nonrealizable}$). There does not exist a word $\mu'$ on $\NN^*$ or $\ZZ$ such that $\age(G_\mu)=\age(G_{\mu'})$. \\
$(b)$ Let $\nu:=\ldots 11100$ be an infinite word on $\NN^*$ (the corresponding graph is depicted in (b) of  Figure $\ref{fig:gmu-nonrealizable}$). There does not exist a word $\nu'$ on $\NN$ or $\ZZ$ such that $\age(G_\nu)=\age(G_{\nu'})$.\\
\emph{Proof of (a):} Every vertex of the graph $G_\mu$ has finite degree. Suppose for a contradiction that there exists a word $\mu'$ on $\NN^*$ or $\ZZ$ such that $\age(G_\mu)=\age(G_{\mu'})$. Then there exists $i \in \ZZ$ such that $\mu'(i)=0$ because otherwise $G_{\mu'}$ would be a path and hence $\age(G_\mu)\neq \age(G_{\mu'})$. But then the vertex $i$ of $G_{\mu'}$ would have infinite degree which is impossible since every vertex of the graph $G_\mu$ has finite degree. $\hfill$ $\blacksquare$\\
\emph{Proof of (b):} The graph $G_\nu$ has two vertices of infinite degree. Suppose for a contradiction that there exists a word $\nu'$ on $\NN$ or $\ZZ$ such that $\age(G_\nu)=\age(G_{\nu'})$. Then $\nu'$ must take the value $0$ on an infinite subset of $I$ of $\ZZ$ because otherwise every vertex of $G_{\nu'}$ would have finite degree which is impossible since $\age(G_\nu)=\age(G_{\nu'})$. Let $I'\subseteq I$ be an infinite set of nonconsecutive integers. Then $G_{\nu'}$ induces an infinite clique on $I'$. This is not possible since the only cliques of $G_\nu$ have cardinality $3$}.   $\hfill$ $\blacksquare$
\end{remark}

\begin{remark} Given a word $\nu$ we associate the graph $G{\nu}$ whose vertex set $V(G^\nu)$ is $\{-n+1, \ldots,  0,1\}$  if  the domain of $\nu$ is $\{-n+1, \ldots,  0\}$, $\NN$ or $\ZZ$ if the domain of $\nu$  is $\NN$ or $\ZZ$ respectively, and $\NN^{*}\cup \{1\}$   if the domain of $\nu$ is $\NN^{*}$. For two vertices $i,j$ with $i<j$ we let $\{i,j\}$ be an edge of $G^{\nu}$ if and only if
\begin{align*}
  \nu_i=1 & \mbox{ and } j=i+1,\mbox{or} \\
  \nu_i=0 &\mbox{ and } j\neq i+1.
\end{align*}
If $\nu$ is of domain $\{0, \ldots,  n-1\}$, $\NN$, $\NN^*$ or $\ZZ$ define $\nu^*$ to be the sequence of domain is $\{-n+1, \ldots,  0\}$, $\NN^*$, $\NN$ or $\ZZ$ respectively by setting $\nu^*(i):=\nu(-i)$. Then ${G}^{\nu^*}$ and $G_\nu$ are isomorphic.
\end{remark}

\begin{remark}\label{remark:union-clique-inde}
For every word $\mu$  the graph $G_\mu$ is the union of at most two infinite cliques and at most two infinite independent sets.

\emph{To see that, let $\mu$ be a $0$-$1$ sequence on an infinite interval $J$ of $\ZZ$.
\begin{enumerate}[$(1)$]
  \item  If $\mu$ takes the value $0$ or the value $1$ finitely many times, then there exists a finite interval $K$ of $J$ such that ${G_\mu}_{\restriction J\setminus K}$ has at most two connected components and either each connected component is an infinite path or the complement of an infinite path.
 \item If  $\mu$ takes the values $0$ and $1$ infinitely many times, let $J_0:=\{j\in J : \mu(j)=0\}$ and $J_1:=\{j\in J : \mu(j)=1\}$. For $i\in \{0,1\}$ let $C_i:=\{\min(J_0)+i+2k : k\in \NN\}$. Note that it is possible for $C_0$ or $C_1$ to be empty, for an example consider the periodic sequence $\mu:=011011\ldots$. Then $\{C_0,C_1\}$ is a partition of $J_0$ and $G_\mu$ induces a clique on $C_0$ and on $C_1$. Similarly, for $i\in \{0,1\}$ let $I_i:=\{\min(J_1)+i+2k : k\in \NN\}$. Note that it is possible for $I_0$ or $I_1$ to be empty, for an example consider the periodic sequence $\mu:=100100\ldots$. Then $\{I_0,I_1\}$ is a partition of $J_1$ and $G_\mu$ induces an independent set on $I_0$ and on $I_1$.
\end{enumerate}}
\end{remark}

%A graph $G:= (V, E)$ is a \emph{permutation graph} if there is a linear order $\leq $ on $V$ and a permutation $\sigma$ of $V$ such that the edges of $G$ are the pairs  $\{x, y\}\in [V]^2$ which are reversed by $\sigma$. A graph is  the  comparability graph of a two-dimensional poset if and only if it is also the  incomparability graph of a two-dimensional poset \cite{dushnik-miller}. If the graph is finite,  this amounts to the fact that this is a permutation graph.
%
%During the last fifteen years, several  studies have been devoted to permutation graphs and some variants, in relation with the Stanley-Wilf conjecture and its solution by Marcus and Tard\"os \cite{marcus-tardos}. An emphasis was put on hereditary classes of finite permutation graphs and a classification of these classes, notably in terms of their profile. The role of the notions of primality and of well-quasi-order has been particularly investigated,  see \cite {klazar, vatter}.
%

Here is our first result.

\begin{theorem}\label{thm:permutation-graph}For every $0$-$1$ word $\mu$ the age  $\age(G_\mu)$ consists of permutation graphs.
\end{theorem}

The proof of Theorem \ref{thm:permutation-graph} is given in Section \ref{sec:proof-thm:permutation-graph}. It follows from the  Compactness Theorem of First Order Logic and Lemma \ref{lem:one-extension}. It was brought to us by  Brignall \cite{brignall:private} that  chains are  the same objects as  \emph{pin sequences} (see \cite {brignall:thesis} Subsection 2.6. p.41).

The next result is about the number of hereditary classes  of  finite permutation graphs.
It is easy to prove and well known that there are $2^{\aleph_0}$ such classes.   This is due to the existence of infinite antichains among finite permutation graphs.

In general, it is not true that two words with different sets of finite factors give different  ages. But, we prove:

%Let us recall two basic notions of Symbolic Dynamic. A  word $u$  is \emph{recurrent} if every finite factor occurs infinitely often; the word
% $u$   is \emph{uniformly recurrent} if for every $n\in \NN$ there exists $m\in \NN$ such that each factor $u(p)\ldots u(p+n)$ of length $n$ occurs as a factor of every factor of length $m$.

\begin{theorem}\label{thm:recurrent-word} Let $\mu$ and $\mu'$ be two words. If $\mu$ is recurrent and   $\age (G_\mu) \subseteq \age(G_{\mu'})$, then $Fac(\mu)\subseteq Fac(\mu')$.
\end{theorem}

%
%
%\begin{theorem}\label{thm:recurrent-class} Let $\mathcal{C}$ and $\mathcal{C'}$ be two classes of words. If $\mathcal{C}$ is recurrent and  $\downarrow G(\mathcal{C}) \subseteq \downarrow G(\mathcal{C})$, then $\mathcal{C}\subseteq \mathcal{C'}$.
%\end{theorem}

Using this result and the fact that  there are $2^{\aleph_0}$ $0$-$1$ recurrent words with distinct sets of factors, we obtain the following.

\begin{corollary}\label{thm:cont-ages} There are $2^{\aleph_0}$ ages of permutation graphs.
\end{corollary}

The ages we obtain in Theorem \ref{thm:cont-ages} are not necessarily well-quasi-ordered. To obtain well-quasi-ordered ages,  we consider graphs associated to uniformly recurrent sequences.

\begin{theorem}\label{thm:uniformly-ages}Let $\mu$ be a $0$-$1$ sequence on an infinite interval of $\ZZ$. The following propositions are equivalent.
\begin{enumerate}[$(i)$]
  \item $\mu$ is uniformly recurrent.
  \item $\mu$ is recurrent and $\age(G_\mu)$ is minimal prime.
\end{enumerate}
\end{theorem}

The proofs of Theorem \ref {thm:recurrent-word} and \ref {thm:uniformly-ages} are given in Section \ref{sec:proof-thm:recurrent-word} and \ref{sec:proof-thm:uniformly-ages}.

As it is well known,  there are $2^{\aleph_0}$ uniformly recurrent words with distinct sets of factors (e.g. Sturmian words with different slopes, see Chapter 6 of  \cite{pytheas}). With Theorem  \ref{thm:recurrent-word} we get:

\begin{corollary}\label{thm:minprimeages} There are $2^{\aleph_0}$ ages of permutation graphs which are minimal prime.
\end{corollary}

Theorem \ref{thm:main1} asserts that
minimal prime ages are well-quasi-ordered. Since minimal prime ages are incomparable when ordered by set-inclusion,  it follows from Corollary  \ref{thm:minprimeages} that the set of well-quasi-ordered ages of permutation graphs, when ordered by set inclusion, has an uncountable antichain. On the other hand, observe that the chains are  countable.

\begin{problem}
Does every uncountable set of ages of permutation graphs, when ordered by set inclusion, contain an uncountable antichain of ages?
\end{problem}

\begin{remark}\label{rem:b.q.o.} \emph{When $\mu$ is uniformly recurrent,  the age $\mathcal C$ of $G_{\mu}$ is w.q.o. since it is minimal prime. In fact,  it is b.q.o. Indeed, since $\fac(\mu)$ is J\'onsson,  it is b.q.o.  (see \cite {pouzettr, carroy-pequignot}). From the extension of Higman's Theorem  to b.q.o, the set $(\fac(\mu))^*$ of finite sequences of members of $\fac(\mu)$, once equipped with the Higman's ordering of finite sequences,  is b.q.o. If  $s:= (u_0, \dots u_k)\in (\fac(\mu))^*$, we  may represent it by  a sequence $u'_0, \dots u'_k$ of factors of $\mu$ in such a way that $u'_i $ is before $u'_{i+1}$ and not contiguous to it. The graph induced by $G_{\mu}$ on this union of factors does not depend of the representation. Denote it by $G(s)$. Observe that the map which associate $G(s)$ to each $s$ is order preserving. It follows that its range is b.q. o. Once observed that this range is $\mathcal C$, the result follows.}
\end{remark}
\begin{remark}\label{rem:nonw.q.o.}
\emph{If $\mu$ is periodic, the collection $\mathcal C^{\leq \omega}$, of countable $G$ such that $\age (G) \subseteq \mathcal C$ is $1^-$-w.q.o. (and, in fact,  $1^-$-b.q.o. But if $\mu$ is uniformly recurrent and not periodic, $\mathcal C^{\leq \omega}$ is not w.q.o. (indeed, the sequence of $G_n:= G_{\mu_ \restriction  [n, \rightarrow[}$, $n\in \NN$,  is strictly decreasing). This simple fact is a reason for using uniformly recurrent sequences in the theory of relations}.
\end{remark}

If $\mu$ is an infinite word, then $\age( G_{\mu})$ is not $2^{-}$-w.q.o. However,

\begin{theorem}\label{thm-1-wqo}  If $\mu$ is an infinite word, $\age( G_{\mu})$ is  $1^{-}$-w.q.o. if and only if $\mu$ is periodic.
\end{theorem}
\begin{proof}
Suppose that $\age( G_{\mu})$ is  $1^{-}$-w.q.o. We claim   that $\fac(u)$ is w.q.o. for the suffix  order.  According to Theorem \ref{thm:ulti-periodic} this implies that  $\mu$ is periodic. The proof of our claim is based on the following observation. Let  $w:= w_0\dots w_n$ and   $w':= w'_0\dots w’_{n’}$ be  two finite words. Then  $w$ is a suffix  of  $w’$ if and only if  the labelled graph  $(G_{w}, n)$ embeds in the labelled graph  $(G_{w'}, n’)$. Indeed, if $n$ is mapped to  $n'$, then since  $n-1$ is the unique neighbour or nonneighbour of $n$ in  $G_w$ we infer that   $n-1$ is mapped to the  unique neighbour or nonneighbour of $n'$ in  $G_{w'}$. Hence, the labelled graphs obtained by deleting $n$ and $n'$ and labelling them  $n-1$ and $n'-1$ embed in each other.  Conversely, if $\mu$ is periodic, then according to  Theorem \ref{thm:ulti-periodic}, $\fac (\mu)$ is w.q.o. for the prefix and the suffix order. We prove first that the collection of $(G_{w}, a_w)$, where  $w\in \fac (\mu)$ and $a_w$ is a constant, is w.q.o. (decompose  each $G_{w}$ into an initial  part and a final part  containing only the label $a_w$. An infinite  sequence   of such labelled graphs yields two infinite sequences;  extract an increasing sequence from the first and then an increasing sequence from the corresponding sequence. This yields an increasing sequence). From that fact, the proof that $\age  (G_{\mu})$ is w.q.o. is as in Remark \ref{rem:b.q.o.}. \end{proof}

Permutation graphs come from posets and  from bichains. Let us recall that a \emph{bichain} is relational structure $R:= (V, (\leq', \leq''))$ made of a set $V$ and two linear orders $\leq'$ and $\leq''$ on $V$. If $V$ is finite and has $n$ elements, there is a unique permutation $\sigma$ of $\{1, \ldots, n\}$ for which $R$ is isomorphic to the  bichain $C_{\sigma}:= (\{1, \ldots, n\}, \leq, \leq_{\sigma})$ where $\leq$ is the natural order on $\underline n:=\{1, \ldots, n\}$ and  $\leq_{\sigma}$ is the linear order defined by $i\leq_{\sigma} j$ if $\sigma(i)\leq \sigma (j)$.

%For an example, let $\sigma$ be the permutation of $\underline {10}:= \{1, 2, 3, 4, 5, 6, 7, 8, 9, 10\}$ given by  the sequence of its values: $2,  4, 6, 8, 10, 1, 3,5,7,9$. The sequence of elements of  $\underline {10}$ ordered  according to  $\leq_{\sigma}$ is: $6<_{\sigma}1<_{\sigma}7<_{\sigma}2<_{\sigma}8<_{\sigma}3<_{\sigma}9<_{\sigma}4<_{\sigma}10<_{\sigma}5$. Note that  this is the sequence of values of $\sigma^{-1}$.
%Let us  represent $\sigma$ by its graph  in the product $\underline n\times \underline n$, that is the set $G(\sigma):=\{(i,\sigma(i)): i\in \underline n\}$  and  order this set componentwise, that is set $(i,\sigma(i))\leq (j, \sigma(j))$ if $i\leq j$ and $\sigma(i)\leq \sigma (j)$. Since $\sigma$ is bijective, the poset $G(\sigma)$ is the intersection of two linear orders,  given respectively by the natural order on the first and on the second coordinate. If we identify each $i$ to $(i, \sigma (i))$, the order induced on $\underline n$ is the intersection of $\leq$ and $\leq_{\sigma}$.

If we represent bichains by permutations, embeddings between bichains is equivalent to the \emph{ pattern   containment}  between the corresponding permutations, see Cameron \cite{cameron}.

To a bichain $R:= (V, (\leq', \leq''))$, we may associate the intersection order $o(R):= (V, \leq'\cap\leq'')$ and to $o(R)$ its comparability graph.

% An order of the form $o(R)$ is called    \emph{two-dimensional}. A graph is  the  comparability graph of a two-dimensional poset if and only if it is also the  incomparability graph of a two-dimensional poset \cite{dushnik-miller}. If the graph is finite,  this amounts to the fact that this is a permutation graph.
The following is Theorem 67 from \cite{pouzet-zaguia-wqo-2022}.

\begin{theorem}\label{thm:minimalprime-graph-poset}
\begin{enumerate}[$(1)$]
\item Let $P:=(V,\leq)$ be a  poset.  Then  $\age(\inc (P))$ is minimal prime if and only if
$\age(\comp(P))$ is minimal prime. Furthermore,  $\age(P)$ is minimal prime if and only if $\age (\ainc(P))$ is minimal prime and $\downarrow  \prim (\age(P))= \age (P)$.
\item  Let $B:=(V,(\leq_1, \leq_2))$ be a bichain and $o(B):= (V, \leq_1\cap \leq_2)$. Then  $\age (B)$ is minimal prime if and only if $\age (o(B))$ is minimal prime and $\downarrow  \prim (\age (B))= \age (B)$.
\end{enumerate}
\end{theorem}

%

%
%From this result we deduce

%We now state a similar result to Theorem \ref{thm:w.q.o. ages} for bichains and two-dimensional posets.
%
%\begin{theorem}\label{thm:generalminprimeages} There are $2^{\aleph_0}$ ages of bichains and ages of two-dimensional orders  which are minimal prime.
%\end{theorem}

With $(2)$ of Theorem \ref{thm:minimalprime-graph-poset} and Corollary \ref {thm:minprimeages} we  have:

\begin{theorem}\label{thm:minprimeages-bichains}
There are $2^{\aleph_0}$ ages of bichains and permutation orders  which are minimal prime.
\end{theorem}

% \begin{enumerate}[$(1)$]
%\item Let $P:=(V,\leq)$ be a  poset.  Then  $\age(\inc (P))$ is minimal prime if and only if
%$\age(\comp(P))$ is minimal prime. Furthermore,  $\age(P)$ is minimal prime if and only if $\age (\ainc(P))$ is minimal prime and $\downarrow  \prim (\age(P))= \age (P)$.
%\item  Let $B:=(V,(\leq_1, \leq_2))$ be a bichain and $o(B):= (V, \leq_1\cap \leq_2)$. Then  $\age (B)$ is minimal prime if and only if $\age (o(B))$ is minimal prime and $\downarrow  \prim (\age (B))= \age (B)$.
%\end{enumerate}

\subsection{A complete characterization of minimal prime ages of graphs}

\begin{theorem}\label{thm:charact-minimal-prime-ages}
A hereditary class $\mathcal{C}$ of finite graphs is minimal prime if and only if $\mathcal{C}=\age(G_\mu)$ for some uniformly recurrent word on $\NN$, or $\mathcal{C}\in \mathcal {L}$.
\end{theorem}

\begin{proof}
$\Leftarrow$. Follows from Theorems \ref {thm:almostmultigraphs}and  \ref {thm:uniformly-ages}.  and Chapter 6 page 109 of the first author's thesis \cite{oudrar}.

 $\Rightarrow$ Follows essentially from Theorem \ref{thm:chudnovsky}.  Let $\mathcal C$ be a minimal prime age. Then $\mathcal C$ contains infinitely many prime graphs  of one of the types  given in Theorem \ref{thm:chudnovsky}.   If for an example, $\mathcal C$ contains infinitely many chains,  that is graphs of the form $G_{\mu}$ for $\mu$  finite, then, since it is minimal prime, we claim that this is the age of some $G_{\mu}$ with $\mu$ uniformly recurrent. Indeed, let $\mathcal{A}$ be an age containing $0$-$1$ graphs $G_w$ for arbitrarily long finite words $w$. We prove that  $\mathcal{A}$ contains the age of a graph $G_\mu$ where $\mu$ is a uniformly recurrent word. Indeed, let $W$ be the set of finite words $w$ such that $G_w\in  \mathcal{A}$. Clearly, $W$ is an infinite hereditary set of finite words. It follows from Lemma \ref{lem:contains minimal} that $W$ contains an initial segment $U$ which is J\'onsson. It follows from the equivalence $(i) \Leftrightarrow  (iii)$ of Theorem \ref{unif-recurrent} that $U=\fac(\mu)$ where $\mu$ is a uniformly recurrent word. We now prove that $\age(G_\mu)\subseteq  \mathcal{A}$. Let $H\in  \age(G_\mu)$. There exists then $w\in \fac(\mu)$ such that $H$ is an induced subgraph of $G_w$. But $w\in \fac(\mu)\subseteq W$. Thus $G_w\in \mathcal{A}$ as required. For the other cases, use the structure of the infinite graphs described in  Figures \ref{fig:list-graph-min-1} and  \ref{fig:list-graph-min-2}.
\end{proof}

\begin{theorem}\label{thm:charact-minimal-prime-ages-hered}
\begin{enumerate}[$(1)$]
  \item A minimal prime hereditary class $\mathcal{C}$ of finite graphs is hereditary  well-quasi-ordered if and only if $\mathcal{C}\in \mathcal L$.
  \item A minimal prime hereditary class $\mathcal{C}$ of finite graphs remains well-quasi-ordered when just one  label is added if and only if $\mathcal{C}=\age(G_\mu)$ for some periodic $0$-$1$ word on $\NN$, or $\mathcal{C}\in \mathcal L$.
  \end{enumerate}
\end{theorem}

 The corresponding characterization of minimal prime ages of posets and bichains will follow from Theorem \ref{thm:minimalprime-graph-poset} and a careful examination of our list of graphs to decide which graphs are comparability graphs.

\begin{figure}[h]
\begin{center}
\leavevmode \epsfxsize=4in \epsfbox{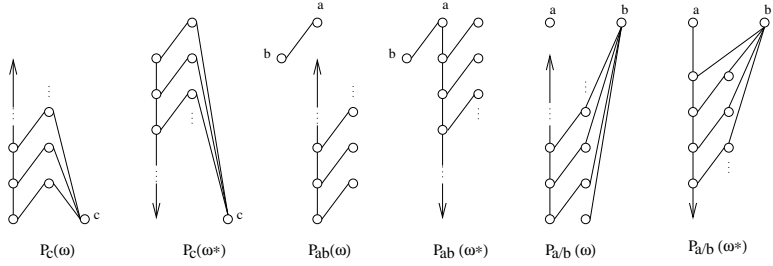}
\end{center}
\caption{Transitive orientations of the graphs $G_5$, $\overline G_5$, $G_6$ and $\overline G_6$}
\label{fig:min-age-dim2}
\end{figure}

\begin{corollary}\label{cor:thm:charact-minimal-prime-ages-bichains}
\begin{enumerate}[$(1)$]
  \item A hereditary class $\mathcal{C}$ of finite comparability graphs is minimal prime if and only if $\mathcal{C}=\age(G_\mu)$ for some uniformly recurrent word on $\NN$, or \\
   $\mathcal{C}\in  \{\age(G_0), \age (G_1), \age(\overline G_1), \age(G_3), \age(G_5), \age(G_6), \age(\overline G_6)\}$.

  \item A hereditary class $\mathcal{C}$ of finite permutation graphs is minimal prime if and only if $\mathcal{C}=\age(G_\mu)$ for some uniformly recurrent word on $\NN$, or \\
   $\mathcal{C}\in  \{\age (G_1), \age(\overline G_1), \age(G_5), \age(G_6), \age(\overline G_6)\}$.
\end{enumerate}
\end{corollary}
We end this section with the following conjecture.

\begin{conjecture}Every infinite prime graph embeds one of the graphs depicted in Figures \ref{fig:list-graph-min-1} and  \ref{fig:list-graph-min-2}, or a graph $G_\mu$ for some $0$-$1$ sequence $\mu$ on an interval of $\ZZ$.
\end{conjecture}
\subsection{Bounds of minimal prime hereditary classes}\label{section:bornes}

We recall that a \emph{bound} of a hereditary class $\mathcal C$ of finite structures (e.g. graphs, ordered sets) is any structure $R\not \in \mathcal C$ such that every proper induced substructure of $R$ belongs to $\mathcal C$.

As we have seen in Theorem \ref{thm:charact-minimal-prime-ages} minimal prime ages of graphs belong either to $\mathcal L$, in which case they have finitely many bounds since they are ages of multichainable graphs (Theorem \ref{thm:multichain}), or they are of the form $\age (G_{\mu})$ with $\mu$ uniformly recurrent.

If $\mu$ is a $0$-$1$ periodic word, $\age(G_\mu)$ may have infinitely many bounds. This is the case if $\mu$ is constant. For the remaining cases within uniformly recurrent sequences, we have the following.

\begin{theorem}\label{thm:bound-uniform} Let $\mu$ be a $0$-$1$ uniformly recurrent word.
\begin{enumerate} [$(1)$]
\item  If $\mu$ is  non periodic,  then $\age(G_\mu)$ has infinitely many bounds;
\item If $\mu$ is  periodic and non constant, then $\age(G_{\mu})$ has finitely many bounds.
\end{enumerate}
\end{theorem}

%For nonconstant $0$-$1$ word, we propose the following conjecture.

%As for Theorem  \ref{thm:bound-uniform}, we have.
%
% \begin{theorem}\label{thm:bound-uniformages}Let $\mu$ be a uniformly recurrent  and non-periodic $0$-$1$ word; let $B_{\mu}$ be a bichain such that the comparability graph of the intersection order is $G_{\mu}$ and $P_{\mu}$  be a transitive orientation of $G_{\mu}$. Then $\age(B_\mu)$ and $\age(P_\mu)$ have infinitely many bounds.
%\end{theorem}

In \cite{brignall-engen-vatter}, Brignall et al  provided an example of a hereditary class of permutation graphs which are w.q.o., have finitely many bounds,  but are not labelled w.q.o. solving negatively a conjecture of  Korpelainen et al  \cite{korpelainen-lozin-razgon}.

As stated in $(2)$ of Theorem \ref {thm:bound-uniform}, ages of $0$-$1$ graphs corresponding to periodic and non constant words provide  infinitely many examples of such classes. Note  these classes are $1^-$-w.q.o.

\section{A proof  of Theorem \ref{thm:permutation-graph} and a characterization of order types of
realizers of transitive orientations of $0$-$1$ graphs}\label{sec:proof-thm:permutation-graph}

Let $P: (V, \leq)$ be a poset. An element $x\in V$ is \emph{extremal} if it is maximal or minimal.
%\subsection {A Proof  of Theorem \ref{thm:permutation-graph}.}

\begin{lemma}\label{lem:one-extension}Let $w:=w_0\ldots w_{n-1}$ be a finite word with $n\geq 2$ and $w':=w_0\ldots w_{n-2}$. Then every realizer $(L_{w'}, M_{w'})$ of a transitive orientation of $G_{w'}$ on $\{-1, 0, \dots, n-2\}$ (if any) such that $n-2$ is extremal in $L_{w'}$ or in $M_{w'}$ extends to a realizer $(L_{w}, M_{w})$ of a transitive orientation of $G_{w}$ on $\{-1, 0, \dots, n-1\}$   such that $n-1$ is extremal in $L_{w}$ or in $M_{w}$.
\end{lemma}
\begin{proof}
Let    $(L_{w'}, M_{w'})$ be a realizer of  a transitive orientation $P_{w'}$ of $G_{w'}$ on $\{-1, 0, \dots, n-2\}$ such that $n-2$ is extremal in $L_{w'}$ or in $M_{w'}$.
 We may assume without loss of generality that $n-2$ is maximal in $L_{w'}$ or in  $M_{w'}$. Otherwise, consider $P^*_{w'}$ and the pair $(L^*_{w'}, M^*_{w'})$. Note that $P^*_{w'}$ is a transitive orientation of $G_{w'}$,  the pair  $(L^*_{w'}, M^*_{w'})$  is a realizer of $P^*_{w'}$,     and $n-2$ is maximal in  $L^*_{w'}$ or in $M^*_{w'}$ (this is because  $n-2$ is minimal in $L_{w'}$ or in $M_{w'}$). We then extend   $(L^*_{w'}, M^*_{w'})$  to a realizer of  $P^*_{w'}$  with the desired property. The dual of this realizer is a realizer of $P_{w}$ with the required property. We may also suppose that $n-2$ is maximal in $L_{w'}$, because  otherwise, we interchange the roles of  $L_{w'}$ and $M_{w'}$. \\
$\bullet$ If $w_{n-1}=1$, then $\{n-2,n-1\}$ is the unique edge of $G_{w}$ containing $n-1$. Clearly $P_w:=P_{w'}\cup \{(n-1,n-2)\}$ is a transitive orientation of $G_{w}$. Let $L_w$ be the linear order obtained from $L_{w'}$ so that $n-1$ appears immediately before $\max(L_{w'})=n-2$ and larger than all other elements and let $M_w$ be the linear order obtained from $M_{w'}$ by letting $n-1$ smaller than all elements of $M_{w'}$. Clearly, $(L_w, M_w)$ is a realizer of  $P_{w}$ and by construction $n-1$ is minimal in $P_{w}$ and in $M_w$.\\
$\bullet$ Else if $w_{n-1}=0$, then $\{n-2,n-1\}$ is the unique non edge of $G_{w}$ containing $n-1$. Since $n-2$ is maximal in $P_{w'}$ we infer that $P_w:=P_{w'}\cup \{(x,n-1)) : x\in \{-1,0,\ldots,n-3\}\}$ is a transitive orientation of $G_{w}$ in which $n-1$ and $n-2$ are incomparable. Let $L_w$ be the linear order obtained from $L_{w'}$ so that $n-1$ appears immediately before $\max(L_{w'})=n-2$ and larger than all other elements and let $M_w$ be the linear order obtained from $M_{w'}$ by letting $n-2$ larger than all elements of $M_{w'}$. Clearly, $(L_w, M_w)$ is a realizer of  $P_{w}$ (indeed,  $n-2$ and $n-1$ are incomparable in $L_w\cap M_w$ and for all $x\in \{-1,0,\ldots,n-3\}$, $x< n-1$ in $L'_w \cap M'_w$ proving that $\{L_w,M_w\}$ is a realizer of $P_w$). By construction $n-1$ is maximal in $P_w$ and $M^{'}_w$. The proof of the lemma  is now complete.
\end{proof}

\begin{lemma}\label{lem:gw-permut-graph}
Let $\mu$ be a $0$-$1$ sequence defined on an interval $I$ of $\ZZ$. Then $G_{\mu}$ is a comparability graph and an incomparability graph. In particular, if $I$ is finite,  then $G_{\mu}$ is a permutation graph.
\end{lemma}
\begin{proof}
We consider two cases.\\
$(a)$ $I$ is finite or $I$ is a final segment of $\ZZ$ bounded below. Write  $I:= \{i_0, \dots, i_n,  \dots \}$ and define for every $n$ a realizer $(L _{n},  M _{n})$ of a transitive orientation of the  restriction of $G_{\mu}$ to  $\{i_0-1, i_0, \dots, i_n\}$.  For that,  use Lemma \ref{lem:one-extension} and induction on  $n$. Note that for  $n= 0$,  the restriction of $G_{\mu}$  to $\{i_0-1, i_0\}$ is either a 2-element independent set or a 2-element clique,  and these are permutation graphs. Then $(L_{\mu}, M_{\mu})$ where  $L_{\mu}:= \bigcup_{n\in I}L_n$ and $M_{\mu}:= \bigcup_{n\in I}M_n$  is a realizer of a transitive orientation of  $G_{\mu}$. Hence  $G_{\mu}$ a comparability and an incomparability graph, and a permutation graph if $I$ is finite.

\noindent $(b)$ $I$ is an initial segment of $\ZZ$.  In this case,  if $F$ is any finite subset of $I$, let   $J$ be a finite interval of $I$ containing $F$. Let $w$ be the restriction of $\mu$ to $J\setminus \{\min(J)\}$. Then the graph induced by $G_\mu$ on $J$ is $G_w$. It follows from (a) that  $G_w$ is a permutation graph, hence  $ {G_{\mu}}_ {\restriction F}$ is permutation graph. It follows   from the Compactness Theorem of First Order Logic that $G_{\mu}$ is a comparability and an incomparability graph.
\end{proof}

Theorem \ref{thm:permutation-graph} readily follows from Lemma \ref {lem:gw-permut-graph}.

The remainder of this section is devoted to characterizing the order types of linear extensions in a realizer of a transitive orientation of the graph $G_{\mu}$ in the case  $\mu$ is a $0$-$1$ word on~$\NN$.

\begin{lemma}\label{lem:finiteinterval}Let $w:=w_0\ldots w_{n-1}$ be a finite word with $n\geq 3$. If  $(L_{w}, M_{w})$ is a realizer of a transitive orientation of $G_{w}$ on $\{-1, 0, \dots, n-1\}$ constructed step by step by means of Lemma \ref{lem:one-extension}, then
 for all $0\leq k \leq n-3$ the set $\{k+2,\dots,n-1\}$ does not meet the intervals of $L_w$ and $M_w$ generated by $\{-1,0,\dots, k\}$.
\end{lemma}
 \begin{proof}

 Let $k\in \{0,\ldots, n-3\}$ and $j\in \{k+2,\ldots, n-1\}$.\\
\textbf{Case 1:} $j>k+2$.\\
Suppose $w_j=0$. Then $j$ is adjacent to all vertices of $\{-1,0,\ldots, k+1\}$. It follows from the algorithm described in Lemma \ref{lem:one-extension} that in a transitive orientation of $G_w$ the vertex $j$ is larger than all elements of $\{-1,0,\ldots, k+1\}$ or the vertex $j$ is smaller than all elements of $\{-1,0,\ldots, k+1\}$. Hence, if $\{L_w,M_w\}$ is a realizer of a transitive orientation of $G_w$, then $j$, in both $L_w$ and  $M_w$,  is either above all elements of $\{-1,0,\ldots, k+1\}$ or is below all elements of $\{-1,0,\ldots, k+1\}$. Hence, $j\not \in I$. We now consider the case $w_j=1$. Then $j$ is not adjacent to any vertex of $\{-1,0,\ldots, k+1\}$. Hence, if $\{L_w, M_w\}$ is a realizer of a transitive orientation of $G_w$, then $j$ is either above all elements of $\{-1,0,\ldots, k+1\}$ in $L_w$ and below all elements of $\{-1,0,\ldots, k+1\}$ in $M_w$,  or $j$ is below all elements of $\{-1,0,\ldots, k+1\}$ in $L_w$ and above all elements of $\{-1,0,\ldots, k+1\}$ in $M_w$. Hence, $j\not \in I$.\\
\textbf{Case 2:} $j=k+2$.\\
We may assume without loss of generality that $k+1$ is maximal in the restriction of a transitive orientation $P$ of $G_{w}$ and $L_{w}$ to $\{-1,0,\dots,k+1\}$ (otherwise consider the dual of $P$ which is a transitive orientation of the restriction of $G_{w}$ to $\{-1,0,\dots,k+1\}$). It follows from the algorithm described in Lemma \ref{lem:one-extension} that $k+2\not \in I$ as required.
\end{proof}

As it is customary, we denote by $\omega$ the order type of $\NN$, by $\omega^*$ the order type of its dual and by $\omega^*+ \omega$ the order type of $\ZZ$.

The proof of the following Lemma is easy and is left to the reader.

\begin{lemma}
  \begin{enumerate}[$(1)$]
    \item The intersection of two linear orders of order type $\omega$ is a w.q.o.
    \item The intersection of two linear orders of order types $\omega$ and $\omega^*$ has no infinite chains.
    \item The intersection of two linear orders of order types $\omega$ and $\omega^*+\omega
   $ is well founded.
  \end{enumerate}
\end{lemma}

\begin{corollary}Let $\mu$ be a word on $\NN$. If $(L,M)$ is a realizer of a transitive orientation of $G_\mu$, then the order types of $L$ and $M$ embed into $\omega^*+\omega$. Furthermore, if $\mu$ has finitely many $0$'s or $1$'s, then the order types of $L$ and $M$ embed into $\omega$ or $\omega^*$, else at least one of $L$ and $M$ have order type $\omega^*+\omega$.
\end{corollary}
\begin{proof}
Let $(L,M)$ be a realizer of a transitive orientation of $G_\mu$. According to Lemma \ref{lem:finiteinterval}, for every $k\in \NN$,  the least interval of $L$ containing $\{-1, 0, \dots,  k\}$ is included in $\{-1, 0, \dots,  k+1\}$. Hence $L$ is a countable increasing union of finite intervals, proving that $L$ embeds in $\ZZ$.

If $\mu$ has finitely many $0$'s or $1$'s, then there exists a final interval $I$ of $\NN$ such that the restriction of $G_\mu$ to $I$ is an infinite one way path or the complement of an infinite one way path. It can be easily seen that the order types in a  realizer of transitive orientations of an infinite one way path or its complement are $\{\omega,\omega\}$ or $\{\omega,\omega^*\}$ or $\{\omega^*,\omega^*\}$. Since $\NN\setminus I$ is an initial segment of $\NN$ we have that the order types of a linear extension in a realizer $P_\mu$ are $\{\omega,\omega\}$ or $\{\omega,\omega^*\}$ or $\{\omega^*,\omega^*\}$. Next we suppose that $\mu$ has infinitely many $0$'s and $1$'s. There exists then two infinite subsets of nonconsecutive integers $J$ and $K$ so that $\mu$ is constant on $J$ and $K$, and $\mu$ takes the value $1$ on $J$ and takes the value $0$ on $K$. Then $P_\mu$ has an infinite antichain, induced by the set $J$, and an infinite chain, induced by the set $K$. The order types of a linear extension in a realizer of $P_\mu$ cannot be $\{\omega,\omega\}$ or $\{\omega^*,\omega^*\}$ because otherwise $P_\mu$ or its dual is w.q.o and hence has no infinite antichains. The order types of a linear extension in a realizer of  $P_\mu$ cannot be $\{\omega,\omega^*\}$ either because otherwise all chains of $P_\mu$ would be finite. %The order types of a linear extension in a realizer of $P_\mu$ cannot be $\{\omega,\ZZ\}$ or $\{\omega^*,\ZZ\}$ because otherwise $P_\mu$ would have minimal or maximal elements which is not possible (the intersection of $\omega$ and $\ZZ$ is well founded and so is the dual of the intersection of $\omega^*$ and $\ZZ$). Hence, we are only left with the possibility that the order types of a linear extension in a realizer of $P_\mu$  is $\ZZ$. This completes the proof of the Corollary.
\end{proof}

\begin{figure}[h]
\begin{center}
\leavevmode \epsfxsize=3in \epsfbox{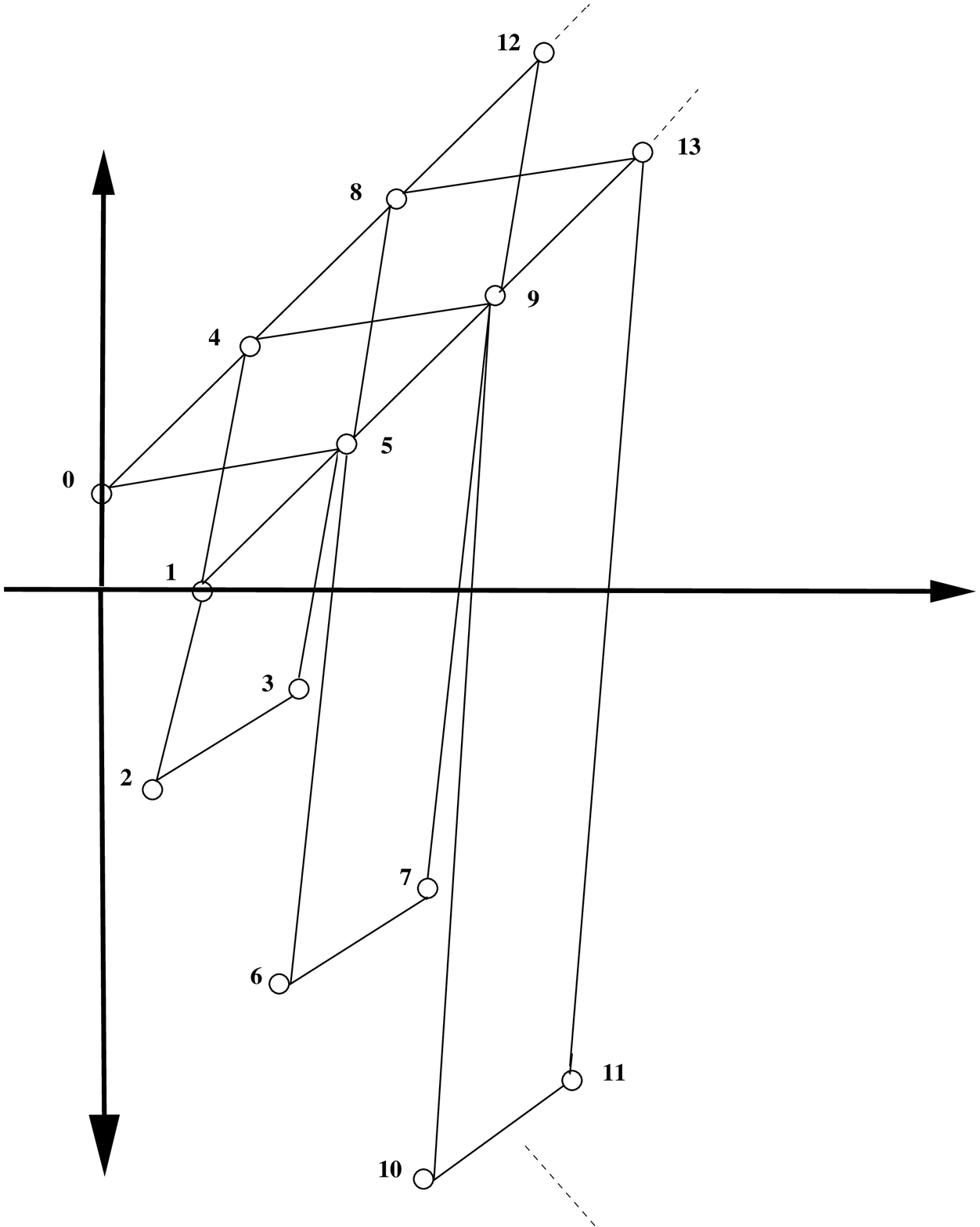}
\end{center}
\caption{An embedding into $\NN \times \ZZ$ of a transitive orientation of the graph corresponding to the periodic $0$-$1$ sequence $\mu:=001100110011\dots$.}
\label{fig:omega-z}
\end{figure}

We now provide examples of $P_\mu$ that have realizers of type $(\omega,\omega^*+\omega)$ and $(\omega^*+\omega,\omega^*+\omega)$.

\begin{example}Let $\mu:=001100110011\dots$. The order types of a linear extension in a realizer of  $P_\mu$ are $\omega$ and $\omega^*+\omega$. Indeed, an embedding of $P_\mu$ into $\NN \times \ZZ$ is depicted in Figure \ref{fig:omega-z}. It follows easily that $P_\mu$ has a realizer of type $(\omega,\omega^*+\omega)$. Since $P_\mu$ is prime it has a unique realizer up to a transposition.
\end{example}

\begin{figure}[h]
\begin{center}
\leavevmode \epsfxsize=5in \epsfbox{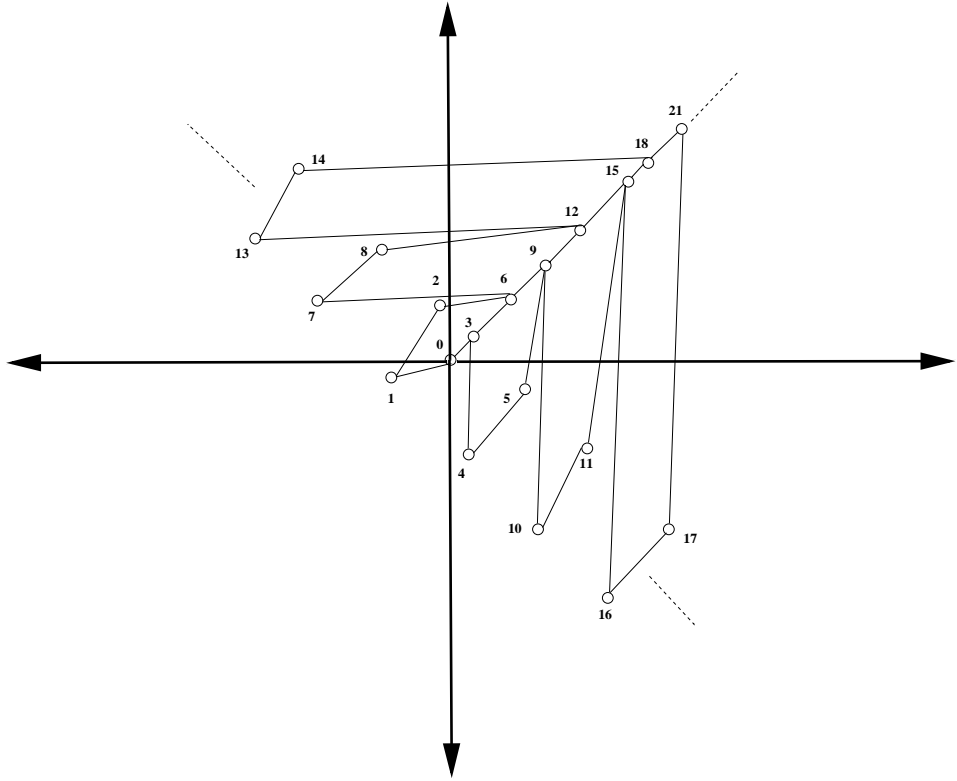}
\end{center}
\caption{An embedding into $\ZZ \times \ZZ$ of a transitive orientation of the graph corresponding to the periodic $0$-$1$ sequence $\mu:=011011011\dots$.}
\label{fig:z-z}
\end{figure}

\begin{example}Let $\mu:=011011011\dots$. The order types of a linear extension in a realizer of  $P_\mu$ is $\ZZ$. Indeed, an embedding of $P_\mu$ into $\ZZ \times \ZZ$ is depicted in Figure \ref{fig:z-z}. It follows easily that $P_\mu$ has a realizer of type $(\omega^*+\omega,\omega^*+\omega)$. Since $P_\mu$ is prime it has a unique realizer up to a transposition.
\end{example}

We should mention that in the first example $G_{\nu}$ nor its complement are permutation graphs, while in the second example both are.

\section{Modules in $G_\mu$}\label{sec:modules}

%Let $G:=(X,E)$ be a graph. A subset $M$ of $X$ is called a \emph{module} in $G$ if for every $x\not \in M$, either $x$ is adjacent to all vertices of $M$ or $x$ is not adjacent to any vertex of $M$. Clearly, the empty set, the singletons in $X$ and the whole set $X$ are modules in $G$; they are called \emph{trivial}. A graph is called \emph{prime} if all its modules are trivial. With this definition, graphs on a set of size at most two are prime. Also, there are no prime graphs  on a three-element set.

The aim of this section is to characterize the modules of a $0$-$1$ graph. We prove among other things, that if $G_\mu$ is not prime, then $\mu$ contains large factors of $0$'s or $1$'s. Results of this section will be used in Section \ref{section:embeddings} to derive properties of embeddings between $0$-$1$ graphs.

We recall that if $G:=(X,E)$ is a graph, then a subset $M$ of $X$ is called a \emph{module} in $G$ if for every $x\not \in M$, either $x$ is adjacent to all vertices of $M$ or $x$ is not adjacent to any vertex of $M$.

%Given a $0$-$1$ sequence $\mu$ on an interval of $\ZZ$  we define $\overline{\mu}:=(\overline{\mu}_{i})_{i\in I}$ to be the $0$-$1$ sequence such that $\overline{\mu}(i):=\mu(i)\dot +1$ where $\dot +$ is the addition modulo $2$.

The following lemma will be useful.

\begin{lemma}\label{lem:module-complement}A graph and its complement have the same set of modules. In particular, $G_{\mu}$ and $G_{\overline{\mu}}$ have the same modules.
\end{lemma}

%\begin{lemma}\label{lem:comp} Let $I$ be an interval of $\NN$ and $\mu:=(\mu_{i})_{i\in I}$ be a $0$-$1$ sequence. Then $\overline{G_\mu}=G_{\overline{\mu}}$.
%\end{lemma}
%\begin{proof}The proof is easy and is left to the reader.
%\end{proof}

Lemma \ref{lem:module-complement} and Remark \ref{lem:comp} of subsection \ref{subsection:01graphs} combined together will allow us to simplify proofs. Indeed, if we are arguing on the value of $\mu$ on a particular integer $i$ we may only consider the case $\mu(i)=0$ (or $\mu(i)=1$).

We recall some properties of modules in a graph. The proof of the following lemma is easy and is left to the reader (see \cite{fraisse84}).

\begin{lemma}\label{lem:prop-modules}Let $G=(V,E)$ be a graph. The following propositions are true.
\begin{enumerate}[$(1)$]
  \item The intersection of a nonempty set of modules is a module (possibly empty).
  \item The union of two modules with nonempty intersection is a module.
  \item For two modules $M$ and $N$, if $M\setminus N\neq \varnothing$, then $N\setminus M$ is a module.
\end{enumerate}
\end{lemma}

Let $G:=(X,E)$ be a graph and $\{x,y,z\} \subseteq X$. We say that $z$ \emph{separates} $x$ and $y$ if $\{z,x\}$ is an edge and $\{z,y\}$ is not and edge, or vice versa. For instance,
\begin{itemize}
\item if $I$ is an interval of $\NN$, $\mu$ is a $0$-$1$ sequence on $I$ and $i\in I$, then $i$ separates $i-1$ and $j$ for all $j<i-1$ in $G_\mu$. (Indeed,  $\{j,i\}$ is an edge if and only if $\{i-1, i\}$ is not an edge).
\end{itemize}

\begin{lemma}Let $G$ be a graph and $\{x,y,z\}\subseteq V(G)$. If $z$ separates $x$ and $y$ and if $x$ and $y$ belong to a module in $G$, then $z$ belongs to that module.
\end{lemma}

\begin{lemma}\label{lem:intial-interval-module}Let $\mu$ be a $0$-$1$ sequence on an interval $I:=\{i_1,\ldots,i_n,\ldots\}$ of $\NN$ and let $i_0:=i_1-1$. Let $J\subseteq \{i_0\}\cup I$ be a nonempty subset.   Let $J^{-}$ be the maximal initial segment of $J$ which is an interval of $\{i_0\}\cup I$. If $J$ is not an interval of $\{i_0\}\cup I$, then $J^{-}$ is a module of ${G_{\mu}}_{\restriction J}$. In particular, if $J^{-}$ is not a singleton, then ${G_{\mu}}_{\restriction J}$ is not prime.
\end{lemma}
\begin{proof} If $J$ is not an interval of $\{i_0\}\cup I$, $J\setminus J^{-}$ is nonempty. Furthermore,   no element of $J\setminus J^{-}$ separates two elements of $J^{-}$ (indeed, the element $i_k:= \max (J^{-}) +1$ does not belong to $J^{-}$ and is the only element of $I$ that separates two elements of $J^{-}$).  Therefore $J^{-}$ is a module of ${G_{\mu}}_{\restriction J}$. If $J^{-}$ is not a singleton, then since it is distinct from $J$,  it is a nontrivial module of ${G_{\mu}}_{\restriction J}$ and therefore  ${G_{\mu}}_{\restriction J}$ is not prime.
\end{proof}

\begin{corollary}\label{lem:5} Let $\mu$ be a $0$-$1$ word on $\NN$ and let $F\subseteq \{-1\}\cup \NN$ be such that ${G_{\mu}}_{\restriction F}$ is prime. Then $F\setminus \{\min(F)\}$ is an interval of $\NN$.
\end{corollary}
\begin{proof} We apply Lemma \ref{lem:intial-interval-module} with $J:=F\setminus \{\min(F)\}$. It follows that if $J$ is not an interval of $I$, then $J^{-}$ is a module of ${G_{\mu}}_{\restriction J}$. Since no element of $J\setminus J^{-}$ separates two elements of $\{\min(F)\}\cup J^{-}$ we infer that $\{\min(F)\}\cup J^{-}$ is a module of ${G_{\mu}}_{\restriction F}$ which is prime. Hence, if $J^{-}$ is not empty $\{\min(F)\}\cup J^{-}$ is a nontrivial module of ${G_{\mu}}_{\restriction F}$ which is impossible. This proves that $F\setminus \{\min(F)\}$ is an interval of $\NN$ as required.
\end{proof}

\begin{lemma}\label{lem:module}Let $I:=\{i_1,\ldots,i_n, \ldots\}$ be an interval of $\NN$ of cardinality at least $3$, $\mu$ be a $0$-$1$ sequence on $I$ and $i_0=i_1-1$. For $k\geq 2$, $\{i_0,\ldots,i_{k-1}\}$ is a module of $G_\mu \setminus \{i_j\}$ if and only if $k=j$.
\end{lemma}
\begin{proof}We only need to prove the forward implication. Suppose that $\{i_0,\ldots,i_{k-1}\}$ is a module of $G_\mu \setminus \{i_j\}$. Then $i_j\not \in \{i_0,\ldots,i_{k-1}\}$ and hence $k\leq j$. Since $i_k$ is the only vertex that separates $i_{k-1}$ and $i_{k-2}$ (recall that $k\geq 2$) we infer that $j=k$.
\end{proof}

\begin{corollary}\label{cor:module}Let $I:=\{i_1,\ldots,i_n\}$ be an interval of $\NN$ of cardinality at least $3$, $\mu$ be a $0$-$1$ sequence on $I$ and $i_0=i_1-1$. We suppose $G_\mu$ is prime and let $x\in \{i_0\}\cup I$. If $G_\mu \setminus \{x\}$ is prime, then $x\in \{i_0,i_1,i_n\}$.
\end{corollary}

In the next lemma we state some properties of modules of $G_\mu$ when $\mu$ is a word on $\NN$. It follows that a nontrivial module of $G_{\mu}$ with at least three elements is necessarily the whole domain of $G_{\mu}$ minus a singleton.

\begin{lemma}\label{lem:1} Let $I:=\{i_1,\ldots,i_n,\ldots\}$ be an interval of $\NN$, $i_0=i_1-1$ and let $\mu$ be a $0$-$1$ sequence on $I$. Let $M$ be a nontrivial module of $G_{\mu}$.
\begin{enumerate}[$(1)$]
\item Let $i_j,i_k\in M$ with $j<k$.
\begin{enumerate}[$(a)$]
\item If $i_{k+1}\in I$, then $i_{k+1}\in M$.
\item $\{m\in I : i_k\leq m\}\subseteq M$.
\item exactly one of $i_0$ and $i_1$ is in $M$.

\end{enumerate}
\item The largest final segment $F$ of $I$ included in $M$ is nonempty.

\item Assume $F$ has at least two elements. Then

\begin{enumerate}[$(a)$]
  \item $\mu$ is constant on $F\setminus \{\min(F)\}$ and $\mu(\min(F))\neq \mu(\min(F)+1)$.
  \item  $\{m\in I : m\leq \min(F)-2\}\subseteq M$.
  \item $F=I$ or $F=I\setminus \{i_1\}$.
\end{enumerate}
\end{enumerate}
\end{lemma}
\begin{proof}
\begin{enumerate}[$(1)$]
\item Let $i_j,i_k\in M$ with $j<k$.
\begin{enumerate}[$(a)$]
\item Suppose $i_{k+1}\in I$. Since $i_j,i_k\in M$  and $i_{k+1}$ separates $i_j$ and $i_k$ we infer that $i_{k+1}\in M$. This proves Item $(1) (a)$. \item Since $M$ contains at least two distinct elements,  Item $(1) (b)$ now follows by repeatedly applying Item $(1) (a)$.
    \item Suppose for a contradiction that $\{i_0, i_1\} \subseteq M$. It follows from Item $(1)(b)$ that $I\subseteq M$. This is impossible since $M$ is nontrivial. This proves that at least one of $i_0$ or $i_1$ is not in $M$. Suppose that $M\cap\{i_0, i_1\}= \emptyset$. Let $k$ be the smallest positive integer such that $i_k\in M$ (note that $k\geq 2$). Since $M$ is nontrivial there exists  some $j>k$ such that $i_j\in M$. Since $i_{k-1}\not \in M$, $i_{k-1}$ cannot separate $i_k$ and $i_{j}$. Since  $i_{k-1}$ and $i_{k}$ are  consecutive, $\mu(i_k) \not= \mu (i_j)$. Hence, $i_0$ separates the two elements $i_{k}$ and  $i_{j}$ of $M$. Since $i_0\not \in M$ and $M$ is a module,  we obtain a contradiction. This proves Item $(1) (c)$.
\end{enumerate}

\item Let  $F$ be the largest final segment of $I$ included in $M$.  We prove that $F$ is nonempty. Indeed, since $M$ is nontrivial, it has at least two elements $i_j,i_k$ with $j<k$. From Item $(1) (b)$ it follows that either $i_{k+1}\in I$ and hence the final segment $\{m\in I : i_k\leq m\}$ is nonempty and is a subset  of $M$. Or, $\max (I)=i_k$ and $\{i_k\}$ is a final segment of $I$ and belongs to $M$.
\item
\begin{enumerate}[$(a)$]
\item Let $l,m\in F\setminus \{\min(F)\}$. Since $\min(F)-1\not \in M$ and $M$ is a module, we infer that the vertex $\min(F)-1$ of $G_{\mu}$ is either adjacent to both $l$ and $m$ or not adjacent to both $l$ and $m$. Thus $\mu(l)=\mu(m)$ and $\mu$ is constant on $F\setminus \{\min(F)\}$ as required. Since $\min(F)$ and $\min(F)+1$ are elements of $M$ and since $\min(F)-1$ and $\min(F)$ are consecutive in $I$ we must have $\mu(\min(F))=\mu(\min(F)+1)+1$, that is $\mu(\min(F))\neq \mu(\min(F)+1)$, proving Item $(3) (a)$.
    \item It follows from $(3)(a)$ that every element $m\in I\cup\{i_0\}$ such that $m<\min(F)-1$ is adjacent to $\min(F)$ but not adjacent to $\min(F)+1$ or vice versa. Since $\min(F)$ and $\min(F)+1$ are elements of $M$ and $M$ is a module we infer that $m\in M$. Hence, $\{m\in I : m\leq \min(F)-2\}\subseteq M$ proving item $(3) (b)$.
        \item It follows from  $(1)(c)$ and $(3)(b)$ that $\min(F)\in \{i_1,i_2\}$. If $\min(F)=i_1$, then $F=I$. Else if $\min(F)=i_2$, then $F=I\setminus \{i_1\}$. This completes the proof of $(3)(c)$ and of the lemma.
        \end{enumerate}
\end{enumerate}
The  proof of the lemma  is now complete.
\end{proof}

Let $I$ be an interval of $\NN$ and  $G_{\mu}$ be the graph on $\{\min(I)-1\}\cup I$  associated to a sequence $\mu$ defined on $I$. In the following proposition we characterize the modules of $G_\mu$. We prove that if $\mu \not \in \{011,100,001,110\}$, then $G_{\mu}$ has at most one nontrivial module. Also, if $I$ is finite, then a nontrivial module of $G_{\mu}$ has necessarily  cardinality $2$ or $|I|$.

\begin{proposition}\label{lem:6}Let $I:=\{i_1,\ldots,i_n,\ldots\}$ be an interval of $\NN$, $i_0:=i_1-1$ and let $\mu$ be a $0$-$1$ sequence on $I$.

\begin{enumerate}[$(1)$]
\item If $\mu \not \in \{011,100,001,110\}$, then $G_{\mu}$ has at most one nontrivial module.
\item If $M$ is a nontrivial module of $G_\mu$, then \underline{either} $M=I$ or $M=\{i_0\}\cup I\setminus \{i_1\}$, \underline{or} $I$ is finite, $I=\{i_1,\ldots,i_n\}$ and $M=\{i_0,i_n\}$ or $M=\{i_1,i_n\}$.
 \end{enumerate}
\end{proposition}
\begin{proof}We recall that a graph on at most two vertices is prime. Hence, if $G_{\mu}$ has a nontrivial module, then $|I|\geq 2$. \\
\textbf{Claim:} If $|I|=2$, then $G_\mu$ has exactly one nontrivial module.\\
\textbf{\emph{Proof of Claim:}} We only consider the case $\mu(i_1)=0$ and deduce the other case by considering $\overline{\mu}$. By inspection, if $\mu(i_2)=0$, then $\{i_0,i_2\}$ is the only nontrivial module of $G_\mu$. if $\mu(i_2)=1$, then $\{i_1,i_2\}$ is the only nontrivial module of $G_\mu$.\hfill $\Box$
\begin{enumerate}[$(1)$]
\item We use the characterisation of modules of $G_\mu$ found in Item $(2)$. We consider all possible pairs of such modules. %Unless $n\leq 3$, in each case we derive a contradiction.
    \begin{enumerate}[{Case} 1.]

    \item $\{i_0,i_n\}$ and $\{i_1,i_n\}$ are both modules of $G_\mu$.\\
    Since $\{i_0,i_n\}\cap \{i_1,i_n\}\neq \varnothing$ and the union of two modules with nonempty intersection is a module (see Item $(2)$ of Lemma \ref{lem:prop-modules}) we infer that $A:=\{i_0,i_n\}\cup \{i_1,i_n\}= \{i_0,i_1,i_n\}$ is a module of $G_\mu$. It follows from Item $(2)$ that $A$ is trivial. Since $A$ has $3$ elements we infer that $A=\{i_0\}\cup I$. This implies that $|I|=2$ and hence $n=2$. We derive a contradiction from the Claim. %Now suppose $\mu(i_1)=0$. Then $i_0$ is not adjacent to $i_1$. Since $\{i_1,i_2\}$ is a module of $G_\mu$ we infer that $i_0$ is not adjacent to $i_2$, that is $\mu(i_2)=1$. But then $i_1$ is adjacent to $i_2$ and not adjacent to $i_0$. This contradicts our assumption that $\{i_0,i_2\}$ is a module of $G_\mu$. We obtain a similar contradiction if $\mu(i_1)=1$.

    \item $I$ and $\{i_0\}\cup I\setminus \{i_1\}$ are both modules of $G_\mu$.\\
    Since the intersection of two modules is a module we infer that $A:=I\cap (\{i_0\}\cup I\setminus \{i_1\})= I\setminus \{i_1\}$ is a module of $G_\mu$. It follows from Item $(2)$ that $A$ is trivial. Hence, $A=\varnothing$ or $A$ is a singleton or $A=\{i_0\}\cup I$. This last case is not possible.  The case $A=\varnothing$ is also not possible because otherwise $I=\{i_1\}$, which contradicts $|I|\geq 2$. We are left with the case $A$ is a singleton, that is $I$ has two elements. We derive a contradiction from the Claim.. %A similar proof as in Case 1 yields a contradiction.

    \item $\{i_0,i_n\}$ and $I$ are both modules in $G_\mu$.\\
    We apply Item $(3)$ of Lemma \ref{lem:prop-modules} with $M:=\{i_0,i_n\}$ and $N:=I$. Since $M\setminus N\neq \varnothing$, then $A:=N\setminus M=I\setminus \{i_n\}$ is a module of $G_\mu$. It follows from Item $(2)$ that $A$ is trivial. Hence, $A=\varnothing$ or $A$ is a singleton or $A=\{i_0\}\cup I$. This last case is not possible.  The case $A=\varnothing$ is also not possible because otherwise $I=\{i_n\}$, which contradicts $|I|\geq 2$. We are left with the case $A$ is a singleton, that is $I$ has two elements. We derive a contradiction from the Claim. %A similar proof as in Case 1 yields a contradiction.

\item $\{i_1,i_n\}$ and $\{i_0\}\cup I\setminus \{i_1\}$ are both modules of $G_\mu$.\\
We apply Item $(3)$ of Lemma \ref{lem:prop-modules} with $M:=\{i_1,i_n\}$ and $N:=\{i_0\}\cup I\setminus \{i_1\}$. Since $M\setminus N\neq \varnothing$, then $A:=N\setminus M=\{i_0\} \cup I\setminus \{i_1,i_n\}$ is a module of $G_\mu$. It follows from Item $(2)$ that $A$ is trivial. Hence, $A=\varnothing$ or $A$ is a singleton or $A=\{i_0\}\cup I$. This last case and the case $A=\varnothing$ are not possible. We are left with the case $A$ is a singleton. Since $i_0\not \in I\setminus \{i_1,i_n\}$ we infer that $I\setminus \{i_1,i_n\}=\varnothing$ and hence $I$ has at most two elements. We derive a contradiction from $I|\geq 2$ in the case $I$ is a singleton, and from the Claim in the case $|I|=2$. %A similar proof as in Case 1 yields a contradiction.

    \item $\{i_1,i_n\}$ and $I$ are both modules in $G_\mu$.\\
   Since $I$ is a module it follows from (3)(c) of Lemma \ref{lem:1} that $\mu$ is constant on $I\setminus \{i_1\}$ and $\mu(i_1)\neq \mu(i_n)$. Then $n\leq 3$ because otherwise $i_2$ separates $i_1$ and $i_n$ contradicting our assumption that  $\{i_1,i_n\}$ is a module in $G_\mu$. It follows that $\mu=100$ or $\mu=011$.

    \item $\{i_0,i_n\}$ and $\{i_0\}\cup I\setminus \{i_1\}$ are both modules in $G_\mu$.\\
Since $\{i_0\}\cup I\setminus \{i_1\}$  is a module it follows from (3)(c) of Lemma \ref{lem:1} that $\mu$ is constant on $I\setminus \{i_1,i_2\}$ and $\mu(i_2)\neq \mu(i_n)$. Then $n\leq 3$ because otherwise $i_{n-1}$ separates $i_0$ and $i_n$ contradicting our assumption that  $\{i_0,i_n\}$ is a module in $G_\mu$. It follows that $\mu=110$ or $\mu=001$.

\end{enumerate}
\item Let $M$ be nontrivial module of $G_{\mu}$. Suppose first that  $M$ has cardinality at least $3$. Let $F$ be the largest  final segment of $I$ included in $M$. From $(2)$ of Lemma \ref{lem:1}, $F$ is nonempty. Since $M$ has at least $3$ elements, it follows from Item $(1)(b)$ and $(1)(c)$ of Lemma  \ref{lem:1}  that $F$ has at least two elements. It follows from $(3)(c)$ of Lemma  \ref{lem:1} that $F=I$ or $F=I\setminus \{i_1\}$. Since $F\subseteq M\subset  I\cup \{i_0\}$, if  $F=I$, then $M=I$.  Else, it follows from $(1)(c)$ of Lemma  \ref{lem:1} that $i_0\in M$. Hence, $M=\{i_0\}\cup I\setminus \{i_1\}$.\\
We now consider the case $M$ has exactly two elements. It  follows from Item $(1) (b)$ of Lemma \ref{lem:1} that $i_{n}\in M$. It follows from Item $(1) (c)$ of Lemma \ref{lem:1} that exactly one of $i_0$ and $i_1$ is in $M$. Hence, $M=\{i_0,i_n\}$ or $M=\{i_1,i_n\}$.
\end{enumerate}
The proof of the proposition is now complete.
\end{proof}

Several corollaries will now follow.

\begin{corollary}\label{cor:primegmu} Let $\mu$ be a $0$-$1$ word on $\NN$. The graph $G_{\mu}$ is prime if and only if $\mu \not \in \{011111\ldots, 100000\ldots, 0011111\ldots, 1100000\ldots\}$.
\end{corollary}
\begin{proof}We prove the following equivalence:  the graph $G_{\mu}$ is not prime if and only if \break $\mu \in \{011111\ldots, 100000\ldots, 0011111\ldots, 1100000\ldots\}$. \\
$\Rightarrow $ Let $M$ be a nontrivial module of $G_\mu$. Since  $I$ is infinite it follows from $(1)$ of Proposition \ref{lem:6} that $M=\NN$ or  $M=\{-1\}\cup \NN\setminus \{0\}$. It follows from (3) (a) of Lemma \ref{lem:1} that $\mu \in \{011111\ldots, 100000\ldots, 0011111\ldots, 1100000\ldots\}$ as required.\\
$\Leftarrow$ Easy.
\end{proof}

Since all of the $0$-$1$ sequences in $\{011111\ldots, 100000\ldots, 0011111\ldots, 1100000\ldots\}$ are not recurrent we get this.

\begin{corollary}\label{cor:prime-recurrent}Let $\mu$ be a recurrent $0$-$1$ word on $\NN$. Then the graph $G_{\mu}$ is prime.
\end{corollary}

We have a similar conclusion to the corollary if we consider words on $\NN^*$ or $\ZZ$ but not necessarily recurrent.

\begin{lemma}Let $\mu$ be a $0$-$1$ word on $\NN^*$ or on $\ZZ$. Then the graph $G_{\mu}$ is prime.
\end{lemma}
\begin{proof}
As in $(2)$ of Lemma \ref{lem:1}, if $M$ is a module of $G_{\mu}$, then the largest final segment $F$ of $\NN^*$ or of $\ZZ$ included in $M$ is nonempty. Suppose for a contradiction that $F\neq \NN^*$ and $F\neq \ZZ$ and let $n:=\min(F)$. Then $n-1\not \in M$ because otherwise $F\cup\{n-1\}$ is a final segment included in $M$ and $F\subseteq  F\cup\{n-1\}$ contradicting the maximality of $F$. Since $M$ is a module and $n-1\not \in M$ we infer that $n-1$ must be either adjacent to both $n$ and $n+1$ or nonadjacent to both $n$ and $n+1$. Since $n-1$ and $n$ are consecutive in $\ZZ$ we have $\mu(n)\neq\mu(n+1)$. But then every $k<n-1$ separates $n$ and $n+1$. It follows from our assumption that $M$ is a module tat $\{k: k<n-1\}\subseteq M$. Hence, $M=\NN^*\setminus \{n-1\}$ or $M=\ZZ^*\setminus \{n-1\}$. We get a contradiction since $n-3\in M$ and $n-1$ separates $n-2$ and $n-3$.
\end{proof}

In the next proposition we show that $G_\mu$ not being prime forces the sequence $\mu$ to have a large factor of $0$'s or of $1$'s.

\begin{proposition}\label{lem:6'}
Let $I:=\{i_1,\ldots, i_n\}$ be a finite  interval of $\NN$,  $i_0=i_1-1$,  and let $\mu$ be a $0$-$1$ sequence on $I$. Suppose $G_\mu$ is not prime and let $M$ be a nontrivial module of $G_\mu$.
\begin{enumerate}[{Case} 1.]
  \item $M$ has cardinality $2$. Then  \underline{either}  $M=\{i_0,i_n\}$ and either $(n=2$ and $(\mu=00$ or $\mu=11))$, or $n>2$  and  ($\mu=1\underbrace{00\ldots 0}_{n-3} 10$ or $\mu=0\underbrace{11\ldots 1}_{n-3} 01$), \underline{or} $M=\{i_1,i_n\}$, and either $(n=3$ and $(\mu=100$ or $\mu=011))$, or $n>3$ and $(\mu=11\underbrace{00\ldots 0}_{n-4} 10$ or $\mu=00\underbrace{11\ldots 1}_{n-4} 01)$.
  \item  $M$ has cardinality $n$. Then  \underline{either} $M=I$ and  $(\mu=1\underbrace{00\ldots 0}_{n-1}$ or $\mu=0\underbrace{11\ldots 1}_{n-1})$, \underline{or} $M=\{i_0\}\cup I\setminus\{i_1\}$ and   $(\mu=00\underbrace{11\ldots 1}_{n-2}$ or $\mu=11\underbrace{00\ldots 0}_{n-2})$. In particular, $M$ induces a path or the complement of a path in $G_{\mu}$.
   \end{enumerate}
\end{proposition}

\begin{proof}Since $G_{\mu}$ and $G_{\overline{\mu}}$ have the same modules (this follows from Lemma \ref{lem:module-complement}) we may assume without loss of generality that $\mu(i_n)=0$.
\begin{enumerate} [{Case} 1.]

  \item Suppose $M$ has exactly two elements. It follows from Item (2)  of Proposition \ref {lem:6}  that either $M=\{i_0,i_n\}$ or $M=\{i_1,i_n\}$.
      Suppose $M=\{i_0,i_n\}$. It follows from our assumption $\mu(i_n)=0$ that $i_{n}$ is not adjacent to $i_{n-1}$ and $i_{n}$ is adjacent to $i_k$ for all $k<n-1$. Since $\{i_0,i_n\}$ is a module we infer that $i_{0}$ cannot be adjacent to $i_{n-1}$ and $i_0$ is adjacent to $i_k$ for all $1\leq k<n-1$. It follows that $\mu(i_{n-1})=1$ if $i_{n-1}\neq i_1$, and $\mu(i_{n-1})=0$ if $i_{n-1}= i_1$, that is if $n=2$. Furthermore, $\mu(i_k)=0$ for all $1<k<n-1$ and $\mu(i_{1})=1$. Thus $\mu=00$ if $n=2$ and $\mu=1\underbrace{00\ldots 0}_{n-3} 10$ if $n>2$.

      Suppose $M=\{i_1,i_n\}$. It follows from our assumption $\mu(i_n)=0$ that $i_{n}$ is not adjacent to $i_{n-1}$ and $i_{n}$ is adjacent to $i_k$ for all $k<n-1$. Hence, $i_{1}$ cannot be adjacent to $i_{n-1}$ and $i_1$ is adjacent to $i_0$ and to $i_k$ for all $1<k<n-1$. It follows that ($\mu(i_{n-1})=0$ if $n=3$) and ($\mu(i_{n-1})=1$ if $n>3$) and $\mu(i_k)=0$ for all $1<k<n-1$ and $\mu(i_{1})=1$. Then $\mu=100$ if $n=3$ and $\mu=11\underbrace{00\ldots 0}_{n-4} 10$ otherwise.

  \item Suppose $M$ has exactly $n$ elements. It follows from Item (2)  of Proposition \ref {lem:6} that $M=I$ or $M=\{i_0\}\cup I\setminus \{i_1\}$.

  Suppose $M=I$. It follows from (2) (a) of Lemma \ref{lem:1} that $\mu$ is constant on $I\setminus \{i_1\}$ and $\mu(i_1)\neq \mu(i_2)$.  It follows from our assumption $\mu(n)=0$ that $\mu(i_1)=1$ and $\mu(i_k)=0$ for all $2\leq k\leq n$, in which case $\mu$ induces the complement of a path on $M$.%, else if  $\mu(i_1)=1$, then $\mu(i_k)=0$ for all $2\leq k\leq n$, in which case $\mu$ induces a complement of a path on $M$.

  Suppose $M=\{i_0\}\cup I\setminus\{i_1\}$. It follows from (2) (a) of Lemma \ref{lem:1} that  $\mu$ is constant on $I\setminus \{i_2\}$ and $\mu(i_2)\neq \mu(i_3)$. It follows from our assumption $\mu(n)=0$ that $\mu(i_2)=1$, then $\mu(i_k)=0$ for all $3\leq k\leq n$ and $\mu(i_1)=1$, in which case $\mu$ induces the complement of a path on $I$.%, else if  $\mu(i_2)=1$, then $\mu(i_k)=0$ for all $2\leq k\leq n$ and $\mu(i_1)=1$, in which case $\mu$ induces a complement of a path on $I$.
\end{enumerate}
\end{proof}

For a set $X$ of finite words let $l_i(X)$  be the supremum, over all words $\mu$ in $X$, of the length of factors of $i's$ in $\mu$. Let $l(X):= \max \{l_0(X), l_1(X)\}$. For a $0$-$1$ sequence $\mu$ we let $l(\mu):=l(\fac(\mu))$. Note that $l(\mu)=l(\overline{\mu})$. We should mention that if $\mu$  uniformly recurrent and non constant, then $l(\mu)$ is finite.

\begin{corollary}\label{cor:large-prime-restr}
Let $X$ be an infinite set of finite words such that $l(X)$ is finite. Then for every $w\in X$ such that $\vert w\vert >l(X)+4$ the graph $G_w$ is prime.
\end{corollary}
\begin{proof}
Let $w\in X$ be such that $\vert w\vert >l(X)+4$ and suppose for a contradiction that $G_w$ is not prime. It follows from Proposition \ref{lem:6'} that $w$ has a factor of $0$'s or of $1$'s of length at least $\vert w\vert -4$. Hence, $\vert w\vert -4\leq l(X)$. This contradicts our assumption $\vert w\vert> l(X)+4$.
\end{proof}

\begin{corollary}\label{cor:primeX} If  $X$ is an infinite initial segment of $\{0, 1\}^*$, then  the set $X'$ of $u\in X$ such that $G_{u}$ is prime is infinite. \end{corollary}

\begin{proof} If $X$ contains  factors of $0$'s  or factors of $1$'s of arbitrary large length, then  the corresponding graphs are clearly prime. Otherwise, $l(X)$ is finite and the conclusion follows from Corollary \ref{cor:large-prime-restr}.
\end{proof}

%\begin{corollary}\label{cor:6'}Let $I:=\{i_0,i_1,\ldots, i_n\}$ be a finite  interval of $\NN$ such that $n\geq 5$,  and let $\mu$ be a $0$-$1$ sequence on $I\setminus \{i_0\}$. Suppose $G_\mu$ is prime.
%\begin{enumerate}[$(1)$]
%\item If $G_\mu\setminus \{i_0\}$ or $G_\mu\setminus \{i_n\}$ is not prime, then $\mu_{\restriction \{i_2,\ldots, i_n\}}$, respectively $\mu_{\restriction \{i_1,\ldots, i_{n-1}\}}$,  is one of the sequences $1\underbrace{00\ldots 0}_{n-4} 10$ or $0\underbrace{11\ldots 1}_{n-4} 01$) or $11\underbrace{00\ldots 0}_{n-5} 10$ or $00\underbrace{11\ldots 1}_{n-5} 01$ or  $1\underbrace{00\ldots 0}_{n-2}$ or $0\underbrace{11\ldots 1}_{n-2})$ or $\mu=00\underbrace{11\ldots 1}_{n-3}$ or $11\underbrace{00\ldots 0}_{n-3}$.
%\item If $G_\mu\setminus \{i_1\}$ is not prime, then $\mu$ has $0^{n-5}$ or $1^{n-5}$ as a factor.
%\end{enumerate}
%\end{corollary}

\begin{corollary}\label{cor:6'}Let $I:=\{i_1,\ldots, i_n\}$ be a finite  interval of $\NN$,  $i_0:=i_1-1$,  and let $\mu$ be a $0$-$1$ sequence on $I$. Suppose $G_\mu$ is prime but at least one of $G_\mu\setminus \{i_0\}$ and  $G_\mu\setminus \{i_n\}$ and $G_\mu\setminus \{i_1\}$ is not prime. Then $\mu$ has $0^{n-6}$ or $1^{n-6}$ as a factor.
\end{corollary}
\begin{proof}
\begin{enumerate}[$(1)$]
\item The graph $G_\mu\setminus \{i_0\}$ is isomorphic to the graph $G_{\mu'}$ where $V(G_{\mu'}):=\{i_1,\ldots, i_n\}$ and $\mu':=\mu_{\restriction \{i_2,\ldots, i_n\}}$. If $G_\mu\setminus \{i_0\}$ is not prime, then the graph $G_{\mu'}$ is not prime and we can apply Proposition \ref{lem:6'} to this graph with $n':=n-1$ and deduce that $\mu'$ has $0^{n'-4}$ or $1^{n'-4}$ as a factor. Hence, $\mu$ has $0^{n-5}$ or $1^{n-5}$ as a factor.\\
The case $G_\mu\setminus \{i_n\}$ not prime can be treated similarly. Apply Proposition \ref{lem:6'} to  the graph $G_{\mu'}$ where $V(G_{\mu'}):=\{i_0,\ldots, i_{n-1}\}$ and $\mu':=\mu_{\restriction \{i_1,\ldots, i_{n-1}\}}$ and $n':=n-1$.
\item Suppose $G_\mu\setminus \{i_1\}$ is not prime and let $M$ be a nontrivial module. If $M=\{i_2,\ldots, i_n\}$, then $i_0$ must be either adjacent to all elements of $M$ or adjacent to non. Thus $\mu$ is constant on $M$, that is $\mu_{\restriction \{i_2,\ldots, i_n\}}=0^{n-1}$ or $\mu_{\restriction \{i_2,\ldots, i_n\}}=1^{n-1}$. If $M\neq \{i_2,\ldots, i_n\}$, then $M$ is a nontrivial module of $G_\mu'$ where $V(G_{\mu'}):=\{i_2,\ldots, i_n\}$ and $\mu':=\mu_{\restriction \{i_3,\ldots, i_n\}}$. We then apply Proposition \ref{lem:6'} to $G_{\mu'}$ with $n'=n-2$ and deduce that $\mu'$ has $0^{n-6}$ or $1^{n-6}$ as a factor. For the remainder of the proof we may assume that $M$ meets $\{i_2,\ldots, i_n\}$ in a singleton and since $M$ is nontrivial $i_0\in M$. Let $k\neq 0$ be such that $i_k\in M$. Since $i_{k+1}$ separates $i_k$ from $i_0$ we infer that $k+1>n$. This shows that $k=n$, that is $M=\{i_0,i_n\}$. Suppose $\mu(i_n)=1$. Then no vertex in $\{i_2,\ldots,i_{n-2}\}$ is adjacent to $i_n$. Since $M$ is a module no vertex in $\{i_2,\ldots,i_{n-2}\}$ is adjacent to $i_0$ and therefore $\mu$ is constant on $\{i_2,\ldots,i_{n-2}\}$ and takes the value $1$. Thus $\mu$ has $1^{n-3}$ as factor. If $\mu(i_n)=0$, then we obtain that $\mu$ has $0^{n-3}$ as factor.
\end{enumerate}
\end{proof}

\section{Embeddings between $0$-$1$ graphs}\label{section:embeddings}

In this section we study the relation between embeddings of words and embeddings of the corresponding $0$-$1$ graphs, see for example Proposition \ref{prop:embed-lmu-finite}.  Results obtained in this section will be used in the proof of Theorem \ref{thm:recurrent-word}.

\begin{lemma}\label{lem:p4}Let $\mu$ a $0$-$1$ sequence on  an interval of $I$ of $\NN$. Let $\{i_0,i_1,i_2,i_3\}\subseteq I$ be such that $i_0<i_1<i_2<i_{3}$. If ${G_\mu}_{ \restriction \{i_0,i_1,i_2,i_3\}}$ is isomorphic to  a  $P_4$, then $\{i_1,i_2,i_3\}$ is an interval of $\NN$ and $\mu$ can be any $0$-$1$ word of length $3$.
\end{lemma}
\begin{proof} Since a $P_4$ is prime it follows from Corollary \ref{lem:5} that $\{i_1,i_2,i_3\}$ is an interval of $\NN$. Since $P_4$ is isomorphic to its complement follows  that if ${G_\mu}_{ \restriction \{i_0,i_1,i_2,i_3\}}$ is isomorphic to $P_4$, then so is ${G_{\overline{\mu}}}_{ \restriction \{i_0,i_1,i_2,i_3\}}$. So we may assume without loss of generality that $\mu(i_3)=1$. If $\mu(i_2)=0$, then $\{i_1,i_2\}$ is not an edge of ${G_{\overline{\mu}}}_{ \restriction \{i_0,i_1,i_2,i_3\}}$ and $\{i_0,i_2\}$ is an edge of ${G_{\overline{\mu}}}_{ \restriction \{i_0,i_1,i_2,i_3\}}$. Hence, $\mu(i_1)=1$ if $i_1$ is a successor of $i_0$ and $\mu(i_1)=0$ otherwise.  If $\mu(i_2)=1$,  then $\{i_1,i_2\}$ is an edge of ${G_{\overline{\mu}}}_{ \restriction \{i_0,i_1,i_2,i_3\}}$ and $\{i_0,i_2\}$ is not an edge of ${G_{\overline{\mu}}}_{ \restriction \{i_0,i_1,i_2,i_3\}}$. Hence, $\mu(i_1)=1$ if $i_1$ is a successor of $i_0$ and $\mu(i_1)=0$ otherwise.
\end{proof}

We denote by $1^k$ the constant word of length $k$ whose all letters are $1$, that is $1^k:=\underbrace{11\ldots 1}_{k \mbox{ times}}$. Similarly we define $0^k$.

\begin{lemma}\label{lem:pk} Let $\mu$ a $0$-$1$ sequence on  an interval $I$ of $\NN$. Let $\{i_0,i_1,\ldots,i_{k-1}\}\subseteq I$ be such that  $k\geq 5$ and $i_0<i_1\ldots  <i_{k-1}$. If ${G_\mu}_{ \restriction \{i_0, \ldots, i_{k-1}\}}$ is isomorphic to $P_k$, then $\{i_1,\ldots,i_{k-1}\}$ is an interval of $\NN$ and $\mu_{\restriction i_3,\ldots, i_{k-1}} =1^{k-3}$ and $\mu_{\restriction \{i_1i_2\}}$ can be  any $0$-$1$ word of length $2$.
\end{lemma}
\begin{proof}Suppose ${G_\mu}_{ \restriction \{i_0, \ldots, i_{k-1}\}}$ is isomorphic to $P_k$. Since $P_k$ is prime for $k\geq 4$ it follows from Corollary \ref{lem:5} that $\{i_1,\ldots,i_{k-1}\}$ is an interval of $\NN$. Then $\mu(i_{k-1})=1$ because otherwise $i_{k-1}$ would be a vertex of degree at least $3$ in $P_k$ and this is impossible. Similarly, we have $\mu_{i_{k-2}}=1$. Since $i_{k-1}$ is a vertex of degree $1$ in ${G_\mu}_{ \restriction \{i_1, \ldots, i_k\}}$,  which is isomorphic to $P_k$, we infer that ${G_\mu}_{ \restriction \{i_0, \ldots, i_{k-2}\}}$ is isomorphic to $P_{k-1}$. The required conclusion follows from Lemma \ref{lem:p4} and an induction on $k\geq 5$.
\end{proof}

%\begin{lemma}Let $w=w_1\ldots w_n$ and $w'=w'_1\ldots w'_{n}$ be two words so that $3\leq n$. Let $V(G_w)=\{v_0,v_1,\ldots,v_n\}$ and $V(G_{w'})=\{v'_0,v'_1,\ldots,v'_{n}\}$. Suppose there exists an isomorphism $f$ of $G_w$ onto $G_{w'}$ such that $f(v_n)=v'_{n}$. Then for all $2\leq i\leq n$, $f(v_i)=v'_{i}$.
%\end{lemma}
%\begin{proof}We notice at once that we can suppose without loss of generality that $w_n=1$. Indeed, if $w_n=0$, then we consider $\overline w$ and $\overline w'$ and recall that $G_{\overline{w}}$ is the complement of $G_w$. Furthermore, two graphs are isomorphic if and only if their complements are isomorphic.\\
%Let $f$ be an isomorphism of $G_w$ onto $G_{w'}$ such that $f(v_n)=v'_{n}$. It follows from our assumption $w_n=1$ that $v_n$ has degree $1$ in $G_w$ and $v_{n-1}$ is its unique neighbour. Since $f$ is an isomorphism and $f(v_n)=v'_{n}$ we infer that $v'_{n}$ has degree $1$ in $G_{w'}$. It follows from this and $n\geq 3$ that $w'_n=1$. Hence, $v'_{n-1}$ is the unique neighbour of $v'_{n}$ in $G_{w'}$. Since $f$ is an isomorphism we must have $f(v_{n-1})=v'_{n-1}$. The proof of the lemma follows by induction on $n\geq 3$.\\
%\end{proof}

\begin{lemma}\label{lem:embed-left-left}Let $\mu$ be a $0$-$1$ sequence on an interval $J$ of $\NN$. Let $I:=\{i_0,i_1,\ldots, i_n\}$ be a finite  interval of $\NN$ with $n\geq 2$ and let $w$ be a $0$-$1$ sequence on $I\setminus \{i_0\}$. Suppose $G_w$ embeds into $G_\mu$ and let $f$ be such an embedding. If $f(i_n)=\max(f(I))$, then $f(\{i_2,\ldots, i_n\})$ is an interval of $J$ and $f$ is strictly increasing on $\{i_2,\ldots, i_n\}$ and $\mu_{\restriction f(\{i_3,\ldots, i_n\})}=w_3\ldots w_n$.
\end{lemma}
\begin{proof}Let $w=w_1\ldots w_n$. We notice at once that we can assume without loss of generality that $w_n=1$. Indeed, if $w_n=0$, then we consider $\overline w$ and $\overline \mu$ and recall that $G_{\overline{w}}$ is the complement of $G_w$. Furthermore, two graphs embed in each other if and only if their corresponding complements embed in each other.\\
Let $f$ be an embedding of $G_w$ into $G_{\mu}$ such that $f(i_n)=\max(f(I))$. If $n=2$, there is nothing to prove. Next we suppose $n\geq 3$. It follows from our assumption $w_n=1$ that $i_n$ has degree $1$ in $G_w$ and $i_{n-1}$ is its unique neighbour. Since $f$ is an embedding we infer that $f(i_n)$ has degree $1$ in $f(G_{w})$. It follows from this and $n\geq 3$ and $f(i_n)=\max(f(I))$  that $\mu(f(i_n))=1$. Hence, $f(i_n)-1$ is the unique neighbour of $f(i_n)$ in $G_\mu$ satisfying $f(i_n)-1<f(i_n)$, and therefore in $f(G_{w})$. Since $f$ is an embedding we must have $f(i_{n-1})=f(i_n)-1$. The proof of the lemma follows by induction on $n\geq 3$.
\end{proof}

It should be noted that the lemma is best possible. Indeed, $f(i_1)<f(i_0)$ is possible in general.

\begin{lemma}\label{lem:embed-left-right}Let $I:=\{i_0,i_1,\ldots, i_n\}$ be a finite  interval of $\NN$ with $n\geq 7$ and let $w$ be a $0$-$1$ sequence on $I\setminus \{i_0\}$ so that  $G_w$ is prime. Let $\mu$ be a $0$-$1$ sequence on an interval $J$ of $\NN$. Suppose $G_w$ embeds into $G_\mu$ and let $f$ be such an embedding. Let $f(I):=\{j_0,j_1,\ldots,j_n\}$ so that $j_0<j_1<\ldots <j_n$. If $f(i_n)\in \{j_0,j_1\}$, then $w$ and $\mu$ have $0^{n-7}$ or $1^{n-7}$ as a factor.
\end{lemma}
\begin{proof}
Let $w:=w_1\ldots w_n$. As in the proof of Lemma \ref{lem:embed-left-left} we may assume without loss of generality that $w_n=1$. Then $i_n$ has degree $1$ in $G_w$ and $i_{n-1}$ is its unique neighbour. Let $f$ be an embedding of $G_w$ into $G_\mu$ and suppose $f(i_n)\in \{j_0,j_1\}$. Since $f$ is an embedding we infer that $f(i_n)$ has degree $1$ in $f(G_w)$. Furthermore, since  $G_w$ is prime, $f(G_w)$ is prime too and therefore $\{j_1,\ldots,j_n\}$ is an interval of $\NN$ (Corollary \ref{lem:5}).
 \begin{enumerate}[{Case} 1.]
\item $f(i_n)=j_0$. \\
Let $k\in \NN$ be such that $j_k:=f(i_{n-1})$. It follows from Corollary \ref{cor:6'} that we may assume $k\not \in \{1,n\}$. Since $f$ is an embedding and $j_0=f(i_n)$ it follows  that $j_k$ is the unique neighbour of $j_0$ in $f(G_w)$. It follows from this and $k\not \in \{1,n\}$ that $\mu(j_k)=0$ and $\mu$ is constant on  $\{j_2,\ldots,j_{n}\}\setminus \{j_k\}$ and takes the value $1$. In particular, $j_k$ has at least $j_0$ and $j_{k+1}$ as neighbours.\\
\textbf{If} $w_{n-1}=1$, then $i_{n-1}$ has degree $2$ in $G_w$ and since $f$ is an embedding $j_k=f(i_{n-1})$ has degree $2$ in $f(G_w)$. It follows from $k\not \in \{1,n\}$ and $\mu(j_k)=0$ that $k=2$. In particular, $\mu_{\restriction \{j_3,\ldots, i_{n}\}}=1^{n-2}$ and $f(G_w)$ embeds $P_{n-1}$. Since $f$ is an embedding we infer that $G_w$ embeds $P_{n-1}$. It follows from Lemma \ref{lem:pk} that $w$ has $1^{n-4}$ as a factor.\\
\textbf{Else if} $w_{n-1}=0$, then $i_{n-1}$ has degree $n-1$. Since $f$ is an embedding we infer that $j_k$ has degree $n-1$ in $f(G_w)$. This forces $k=n-1$. It follows from Lemma \ref{lem:embed-left-left} applied to $w'=w_1\ldots w_{n-1}$ and $\mu$ that $f$ is strictly increasing on $\{i_2,\ldots, i_n\}$ and $\mu_{\restriction f(\{i_3,\ldots, i_{n-1}\})}=w_3\ldots w_{n-1}$. In particular, $w$ has $1^{n-3}$ as a factor.

\item $f(i_n)=j_1$.\\
Let $k\in \NN$ be such that $j_k=f(i_{n-1})$. It follows from Corollary \ref{cor:6'} that we may assume $k\not \in \{1,n\}$. Since $f$ is an embedding and $j_1=f(i_n)$ it follows  that $j_k$ is the unique neighbour of $j_1$ in $f(G_w)$. It follows from this and $k\not \in \{1,n\}$ that:\\
$(a)$ $k=2$ and $\mu$ is constant on  $\{j_2,\ldots,j_{n}\}$ and takes the value $1$, or\\
$(b)$ $k>2$ and $\mu(i_2)=\mu(j_k)=0$ and $\mu$ is constant on  $\{j_3,\ldots,j_{n}\}\setminus \{i_k\}$ and takes the value $1$.\\
\textbf{If} $w_{n-1}=1$, then $i_{n-1}$ has degree $2$ in $G_w$ and since $f$ is an embedding $j_k=f(i_{n-1})$ has degree $2$ in $f(G_w)$. Then only case $(a)$ holds. Indeed, if not $j_k$ would be adjacent to $j_{k+1}$, $j_1$ and $j_0$ and hence has degree $3$ which is impossible. Thus $\mu_{\restriction \{j_2,\ldots, j_{n}\}}=1^{n-1}$. In particular, $f(G_w)$ has an induced $P_n$ and since $f$ is an embedding we infer that $G_w$ has an induced $P_n$. It follows from Lemma \ref{lem:pk} that $w$ has $1^{n-3}$ as a factor.\\
\textbf{Else if} $w_{n-1}=0$, then $i_{n-1}$ has degree $n-1$ in $G_w$. Since $f$ is an embedding we infer that $j_k=f(i_{n-1})$ has degree $n-1$ in $f(G_w)$. Then only case $(b)$ holds. Indeed, if not $j_k$ would be adjacent only to $j_1$, $j_3$ and hence has degree $2$ which is impossible. This forces $k=n-1$. It follows that $\mu_{\restriction \{j_3,\ldots, i_{n-3}\}}=1^{n-5}$. In particular, $f(G_w)$ has an induced $P_{n-4}$ and since $f$ is an embedding we infer that $G_w$ has an induced $P_{n-4}$. It follows from Lemma \ref{lem:pk} that $w$ has $1^{n-7}$ as a factor.
\end{enumerate}
\end{proof}

\begin{proposition}\label{prop:embed-lmu-finite}Let $\mu$ be a recurrent word on $\NN$ such that $l(\mu)$ is finite. Let $w:=w_0\ldots w_{n-1}$ be a finite word such that $n>l(\mu)+7$. If $G_w$ embeds into $G_\mu$  and $f$ is such an embedding, then $f(-1), f(0) <f(1)<f(2)<\ldots<f(n-1)$ and either $\{f(-1), f(1),f(2),\ldots,f(n-1)\}$  or $\{f(0), f(1),f(2),\ldots,f(n-1)\}$is an interval of $\NN$ and $w_2\ldots w_n$ is a factor of $\mu$.
\end{proposition}
\begin{proof} It follows from Corollary \ref{cor:large-prime-restr} that $G_w$ is prime. It follows from our assumption $n>l(\mu)+7$ and Corollary \ref{cor:6'} that $G_w\setminus \{-1\}$ and $G_w\setminus \{0\}$ and $G_w\setminus \{n-1\}$ are also prime. Since $f$ is an embedding it follows that in $f(G_w)$ removal of one of the vertices $f(-1)$ or $f(0)$ or $f(n-1)$ leaves a prime graph. Since $G_w$ is prime and $f$ is embedding it follows that $f(G_w)$ is also prime. It follows from Corollary \ref{lem:5} that $I:= f(V(G_w))\setminus \{\min(f(V(G_w))\}$ is an interval  of $\NN$. It follows from Corollary  \ref{cor:module} that $f(n-1)\in \{\min(f(V(G_w)),\min(I),\max(I)\}$. It follows from Lemma \ref{lem:embed-left-right} that $f(n-1)= \max(f(V(G_w)))$. The required conclusion follows then from Lemma \ref{lem:embed-left-left}.
\end{proof}

\begin{corollary}\label{cor:embed-lmu4}Let $\mu$ be a recurrent word on $\NN$ such that $l(\mu)<4$. Let $u, v$ be finite words such that $\vert v\vert \geq 3$ and $n>l(mu)+4$ and $G_{vu}$ is prime. If $G_{vu}$ embeds into $G_\mu$, then $u\in \fac(\mu)$.
\end{corollary}
\begin{proof}Follows from Proposition  \ref{prop:embed-lmu-finite} applied to $w:=vu$.
\end{proof}
%\begin{corollary}\label{cor:embed-lmu4}Let $\mu$ be a recurrent word on $\NN$ such that $l(\mu)<4$. Let $u, v$ be finite words such that $\vert v\vert \geq 3$ and $G_{vu}$ is prime. If $G_{vu}$ embeds into $G_\mu$, then $u\in \fac(\mu)$.
%\end{corollary}
%\begin{proof}Let $u:=u_0\ldots u_{n-1}$ and let $f$ be an embedding of $G_{vu}$ into $G_\mu$.  It follows from Lemma \ref{lem:embed-left-right} that $f$ maps the vertex of $G_{vu}$  corresponding to $u_{n-1}$ onto $\max(f(V(G_{vu}))$. It follows from Lemma \ref{lem:embed-left-left} that $u\in \fac(\mu)$.
%\end{proof}

\begin{lemma} \label{lem:prime-extension}  If  $\mu$ is recurrent word and $u\in \fac(\mu)$, then  there exists  $v\in \{0,1\}^*$ such that $\vert v\vert \geq 4$ and  $vu\in \fac (\mu)$  and $G_{vu}$ is prime.
\end{lemma}
\begin{proof}We consider several cases.
\begin{enumerate}[{Case} 1.]
\item $\mu$ has $1^4$ as a factor.\\
We can write $\mu=\alpha 1^4\mu'$ where $\alpha$ is a finite word and $\mu'$ is an infinite $0$-$1$ sequence. Since $\mu$ is recurrent $\fac(\mu)=\fac(1^4\mu')=\fac(\mu')$. Hence, we may assume without loss of generality that  $\alpha$ is the empty word. Let $u\in \fac(\mu')$. There exists then $\beta \in \fac(\mu')$ such that $1^4\beta u\in \fac(\mu')$. It follows from Proposition \ref{lem:6'} that $G_{1^4 \beta u}$ is prime. Choose $v:=1^4 \beta$.
\item $\mu$ has $0^4$ as a factor.\\
We apply Case 1 to $\overline{\mu}$ and $\overline{u}$.
\item $l(\mu)<4$.
Let $u\in \fac(\mu)$. Since $\mu$ is recurrent there exists $v\in \fac(\mu)$ such that $vu\in \fac(\mu)$ and $\vert v \vert\geq 4$ and $\vert vu \vert>l(\mu)+4$. It follows from  Corollary \ref{cor:large-prime-restr} that $G_{vu}$ is prime.
\end{enumerate}
\end{proof}

\begin{lemma}\label{lem:embed-1111}Let $\mu$ be a word on an interval $I$ of $\NN$ and let $w:=w_1\ldots w_n$ be any finite word. If $G_{1^4 w}$ embeds into $G_\mu$, then $1w$ is a factor of $\mu$.
\end{lemma}
\begin{proof}We notice at once that it follows from Proposition \ref{lem:6'} that $G_{1^4 w}$ is prime. Let $f$ be an embedding of $G_{1^4 w}$ into $G_\mu$. Then the image of $G_{1^4 w}$ under $f$ is prime. We write $f(V(G_{1^4 w}))=\{i_0, i_1,i_2,i_3,i_4,j_1,\ldots,j_n\}$ so that $i_0<\ldots<i_4<j_1<\ldots<j_n$. It follows from Corollary \ref{lem:5} that $\{i_1,i_2,i_3,i_4,j_1,j_2,\ldots,j_n\}$ is an interval of $\NN$.

We use induction on the length $n\geq 1$ of $w$ to prove the following statement: $\mu_{\restriction \{i_4,j_1,\ldots,j_{n}\}}=1w_1\ldots w_{n}$ and if $w\neq 1^n$, then for all $i\in \{1,\ldots,n\}$, $f$ maps the vertex of $G_{1^4 w}$ corresponding to $w_i$ to the vertex $v_i$.

For the basis case suppose $w\in \{0,1\}$. If $w=1$, then $G_{1^4 w}=G_{1^5}$ is a path on six vertices. Since $f$ is an embedding we infer that $f(G_{{\bf 1}^4 w})$ is a path on six vertices. It follows from Lemma \ref{lem:pk} that $\mu(u_4)=\mu(v_1)=1$ and hence $1w=11$ is a factor of $\mu$ as required. Now suppose $w=0$ and note that $G_{1^4 w}$ has exactly one vertex of degree four.
We prove that $\mu(j_1)=0$. Suppose for a contradiction that $\mu(j_1)=1$. Then $\mu(i_4)=0$ because otherwise  $f(G_{1^4 w})$ wont have a vertex of degree four and since $f$ is an embedding neither will $G_{1^4 w}$ which is impossible. But then in $f(G_{1^4 w})$ the vertex $i_4$ which has degree four is adjacent to the vertex $j_1$ which has degree one and hence in $G_{1^4 w}$ the vertex of degree four is adjacent to a vertex of degree one and this is not possible. A contradiction. Hence, our supposition that $\mu(v_1)=1$ is false, that is $\mu(j_1)=0$ as required. Now since $f(G_{1^4 w})\setminus \{j_1\}$ is a path on five vertices it follows from Lemma \ref{lem:pk} that $\mu(i_4)=1$ and hence $1w=10$ is a factor of $\mu$ as required.

Next we consider the inductive case. We first note that if $w = 1^n$, then $G_{1^4 w}$ is a path on $n + 5$ vertices. We apply Lemma \ref{lem:pk} with $k=n+5$ and deduce that $\mu_{\restriction \{i_3,i_4,j_1,\ldots,j_{n}\}}=1^{n+2}$ and hence $1w$ is a factor of $\mu$. We now assume that $w\neq 1^n$. Suppose that $w_1\ldots w_{n-1}=1^{n-1}$. Then $G_{1^4 w_1\ldots w_{n-1}}$ is a path on $n+4$ vertices. It follows from  Lemma \ref{lem:pk} that $\mu_{\restriction \{i_3i_4,j_1,\ldots,j_{n-1}\}}= 1^{n+1}$. From our assumption that $w\neq 1^n$ we deduce that $w_n=0$. Hence, $G_{1^4 w}$ has a unique vertex of degree $n+4$ and this vertex is associated to $w_n$. Since $f$ is an embedding and $\mu_{\restriction \{i_3i_4,j_1,\ldots,j_{n-1}\}}= 1^{n+1}$ it follows that the image under $f$ of the vertex associated to $w_n$ must be $j_n$ and $j_n$ has degree $n+4$. This shows that $\mu(j_n)=0$ and hence $1w$ is a factor of $\mu$.

Next we suppose that $w_1\ldots w_{n-1}\neq 1^{n-1}$. By the induction hypothesis $\mu_{\restriction \{i_4,j_1,\ldots,j_{n-1}\}}=1w_1\ldots w_{n-1}$ and for all $i\in \{1,\ldots,n-1\}$, $f$ maps the vertex of $G_{1^4 w}$ corresponding to $w_i$ to the vertex $j_i$. We note that $j_{n-1}$ is the unique neighbour or the unique non neighbour of $j_n$ in $f(G_{1^4 w})$. Since $f$ is an embedding it follows that $j_n$ is the image under $f$ of the vertex of $G_{1^4 w}$ corresponding to $w_n$ and $\mu(j_n)=w_n$. This completes the proof of the lemma.
\end{proof}

\begin{corollary}\label{cor:embed-0000}Let $\mu$ be a word on an interval $I$ of $\NN$ and let $w:=w_1\ldots w_n$ be any finite word. If $G_{0^4 w}$ embeds into $G_\mu$, then $0w$ is a factor of $\mu$.
\end{corollary}
\begin{proof}We apply Lemma \ref{lem:embed-1111} to $\overline{\mu}$ and $\overline{w}$ and recall that $G_{1^4 \overline{w}}$ embeds into $G_{\overline{\mu}}$ if and only if the complement of $G_{1^4 \overline{w}}$, which is $G_{0^4 w}$,  embeds into the complement of $G_{\overline{\mu}}$, which is $G_\mu$.
\end{proof}

\begin{corollary}\label{lem:1111-bound}
Let $\mu$ be a word on $\NN$.
\begin{enumerate}[$(1)$]
  \item If $w$ is a bound of $\mu$, then $G_{1^4 w}$ and $G_{0^4 w}$ do not embed into $G_\mu$.
  \item If $(w_i)_{i\in I}$, $I\subseteq \NN$,   is an antichain (with respect to the factor ordering) of finite words such that no $w_i$ starts with $1$, then $(G_{1^4 w_i})_{i\in I}$ is an antichain of (permutation) graphs.
\end{enumerate}
\end{corollary}
\begin{proof}
\begin{enumerate}[$(1)$]
  \item The fact that $G_{1^4 w}$ does not embed into $G_\mu$ follows from Lemma \ref{lem:embed-1111}. The fact that $G_{0^4 w}$ does not embed into $G_\mu$ follows from Corollary \ref{cor:embed-0000}.
  \item Suppose for a contradiction that there exists $i\neq j$ be such that $G_{1^4 w_i}$ embeds into $G_{1^4 w_j}$. It follows from Lemma \ref{lem:embed-1111} that $1w_i$ is a factor of $1^4 w_j$. Since $w_i$ does not start with $1$ we infer that $w_i$ is a factor of $w_j$. This is impossible since by assumption the sequence $(w_i)_{i\in I}$ is an antichain of words.
  \end{enumerate}
\end{proof}

\section{A proof of Theorem \ref{thm:recurrent-word}}\label{sec:proof-thm:recurrent-word}

We prove the following strengthening of Theorem \ref{thm:recurrent-word}. For that we introduce first the following notation: if $X$ is a set of finite $0$-$1$ words we set $G_{X}:= \{G_{w} : w\in X\}$ and
 \[\downarrow G_{X}:= \{ H: H\;  \text{embeds into some}\;  G_w\in G_{X}\}.\]

\begin{theorem}\label{thm:recurrent-word-2}Let $\mu$ be a recurrent word and $X$ be an initial segment of $\{0, 1\}^*$ for the factor ordering. If $\age (G_{\mu}) \subseteq  \downarrow G_{X}$ then $\fac (\mu)\subseteq X$.
\end{theorem}
\begin{proof}
Let $u\in \fac(\mu)$. We prove that $u\in X$. According to Lemma \ref{lem:prime-extension}, since $\mu$ is recurrent and $u\in \fac(\mu)$ there is some $v\in \{0,1\}^*$ with $\vert v\vert \geq 3$ such that $vu\in \fac (\mu)$  and $G_{vu}$ is prime.
Since $G_{vu}\in \age (G_{\mu})\subseteq  \downarrow G_{X}$, $G_{vu}$ embeds in  $G_{w}$ for some $w\in X$. If  $l(\mu)<4$ it follows from Corollary  \ref{cor:embed-lmu4} that $u$ is a factor of $w$. If $l(\mu)\geq 4$ then there is $u'\in \fac(\mu)$ such that $u$ is a factor of $u'$ and either $0^4 u'$ or $1^4 u'$ is a factor of $\mu$. It follows from Lemma \ref{lem:embed-1111} and  Corollary \ref{cor:embed-0000}  applied to either  $v=0^4$ or $v= 1^4$ that $u'$ is a factor of $w$, and so is $u$. Hence,  $u\in X$.
\end{proof}

%\begin{proof}
%Let $u\in \fac(\mu)$. We prove that $u\in X$. According to Lemma \ref{lem:prime-extension}, since $\mu$ is recurrent and $u\in \fac(\mu)$ there is some $v\in \{0,1\}^*$ such that $vu\in \fac (\mu)$  and $G_{vu}$ is prime.  Since $G_{vu}\in \age (G_{\mu})\subseteq  \downarrow G_{X}$, $G_{vu}$ embeds in  $G_{w}$ for some $w\in X$. It follows from Lemma \ref {lem:embed-1111}  that $u$ is a factor of $w$. Hence $u\in X$.
%\end{proof}

Theorem \ref {thm:recurrent-word} now follows by observing  that $\age (G_{\mu'})=\downarrow G_{\fac ({\mu'}) }$  and then applying Theorem \ref {thm:recurrent-word-2} to $X:= \fac(\mu')$.
\section{A proof of Theorem \ref{thm:uniformly-ages}}\label{sec:proof-thm:uniformly-ages}
%
%\begin{theorem}Let $\mu$ be a $0$-$1$ sequence on $\NN$. The following propositions are equivalent.
%\begin{enumerate}[$(i)$]
%  \item $\mu$ is uniformly recurrent.
%  \item $\mu$ is recurrent and $\age(G_\mu)$ is minimal prime
%\end{enumerate}
%\end{theorem}
\begin{proof}
$(i)\Rightarrow (ii)$.
Let $\mu$ be uniformly recurrent.  Then trivially, $\mu$ is recurrent. Since for two infinite sequence $\tau$ and $\tau'$, the equality $\fac (\tau)= \fac(\tau')$ implies $\age (G_{\tau})=\age (G_{\tau'})$, it follows from Theorem \ref{thm:rec-age-word}  that we may assume that $\mu$ is a word on $\NN$. It follows from Corollary \ref{cor:prime-recurrent} that $G_\mu$ is prime. Hence, from  Theorem \ref{ille-theorem} it follows that the set of prime graphs in $\age(G_{\mu})$ is cofinal in $\age(G_{\mu})$ hence infinite. Now let $\mathcal{C}$ be a proper age of $\age(G_{\mu})$. We prove that $\mathcal{C}$ contains only a finite number of prime graphs. If $\mathcal{C}$ contains restrictions on intervals of $\NN$ of arbitrarily large length, then according to Corollary \ref{cor:large-prime-restr},  $\mathcal{C}$ contains finite prime graphs of arbitrarily large length and therefore $\mathcal{C}=\age(G_\mu)$. Else, $\mathcal{C}$ contains only restrictions to factors of $\mu$ of bounded length. It follows from Corollary \ref{lem:5} that every prime member of $\mathcal C$  of cardinality $m$ induces an interval of $\NN$ of cardinality $m-1$. Therefore prime members of $\mathcal{C}$ have bounded cardinality. That is, there are only finitely many prime members of  $\mathcal{C}$. \\
$(ii)\Rightarrow (i)$.  First $\fac(\mu)$ is infinite since $\mu$ is recurrent. Next, let $X$ be an infinite initial segment of $\fac(\mu)$. We claim  that $X= \fac(\mu)$.  Corollary \ref {cor:primeX} asserts that the set $X':= \{ u\in X: G_u\;  \text{is prime}\}$ is infinite.
Since $\downarrow G_{X}$ contains infinitely many prime and $\age (G_{\mu})$ is minimal prime, $\downarrow G_{X}= \age (G_{\mu})$.  Since $\age (G_{\mu})\subseteq \downarrow G_{X}$,    Theorem \ref {thm:recurrent-word-2} asserts that $\fac(u) \subseteq X$. This proves our claim.
\end{proof}

%\begin{proof}It follows from Pouzet that $\mathcal{A}(\mu)$ has infinitely many bounds if $\mu$ is non periodic.
%\end{proof}

%\section{A proof of Theorem \ref {thm:minprimeages-bichains}} \label{minimalprimeage}

%The following result is Theorem 67 from \cite{pouzet-zaguia-w.q.o. 22}.
%
%\begin{theorem}\label{thm:minimalprime-graph-poset}
%\begin{enumerate}[$(1)$]
%\item Let $P:=(V,\leq)$ be a  poset.  Then  $\age(\inc (P))$ is minimal prime if and only if
%$\age(\comp(P))$ is minimal prime. Furthermore,  $\age(P)$ is minimal prime if and only if $\age (\ainc(P))$ is minimal prime and $\downarrow  \prim (\age(P))= \age (P)$.
%\item  Let $B:=(V,(\leq_1, \leq_2))$ be a bichain and $o(B):= (V, \leq_1\cap \leq_2)$. Then  $\age (B)$ is minimal prime if and only if $\age (o(B))$ is minimal prime and $\downarrow  \prim (\age (B))= \age (B)$.
%\end{enumerate}
%\end{theorem}

%Theorem \ref{thm:minprimeages-bichains} now follows since for every word $\mu$ the graph $G_\mu$ is a prime comparability and incomparability graph.

\section{Bounds of $0$-$1$ graphs: a proof of Theorem \ref{thm:bound-uniform}.}\label{sec:proof-thm:bound-uniform}

%and \ref{thm:bound-uniformages}}\label{sec:proof-thm:bound-uniform}

%A property $\mathrm{P}$ is a \emph{critical} for a binary relation $R=(X,\rho)$ if $R$ has property $\mathrm{P}$ but for every $x\in X$, $R_{\restriction X\setminus  \{x\}}$ does not have property $\mathrm{P}$.

Let $\mu$ be a $0$-$1$ sequence. Then every bound of $\age(G_\mu)$ is one of the following types:
\begin{enumerate}[$(1)$]
  \item Finite graphs that are not comparability graphs and that are minimal  with this property.
  \item Finite comparability graphs of critical posets  of dimension three (see subsection \ref{subsubsection:posets}).
  \item Finite comparability graphs of posets of dimension two, that is finite permutation graphs.
\end{enumerate}

For example if $\mu=11111...$, then the bounds of $\age(G_\mu)$ listed according to their type are:
\begin{enumerate}[$(a)$]
  \item Odd cycles of length at least 5. These are of type (1).
  \item Even cycles of length at least 6. These are of type (2).
  \item The complete bipartite graph $K_{1,3}$ and the complete graph $K_3$. These are of type (3).
\end{enumerate}

%We recall that the class of finite comparability graphs has infinitely many bounds and these were characterized by Gallai \cite{gallai} (see \cite{maffray} for an English translation). The list can also be found in \cite{trotter-moore} Figures 4(a) and 4(b). Critical posets of dimension three (and hence finite comparability graphs of posets critical of dimension three) were characterized by Kelly \cite{kelly77}.

Let $\mu$ be a $0$-$1$ sequence. If $\mu$ contains factors of 1's of arbitrarily length, then it follows from Lemma \ref{lem:pk} that $G_\mu$ embeds $P_k$ for infinitely many $k$'s, hence $\age(G_{\mu})$ contains the age of an  infinite path. The cycles $C_k$ are  bounds of the infinite path and form an infinite antichain. Since  cycles of length at least five are not permutation graphs, these cycles  are bounds of $\age(G_{\mu})$. Now suppose that neither $P_k$ nor $\overline{P_k}$ embed in $G_\mu$. In particular, $\mu$ has infinitely many 1's and 0's, that is $G_\mu$ has an infinite independent set and an infinite clique. If we put an upper bound on the length of paths and of complement of paths in members of the lists of  Gallai \cite{gallai} and Kelly \cite{kelly77}, there are only finitely many such members, hence $\age (G_\mu)$ has only finitely many bounds of type (1) and  finitely many bounds of type (2). Hence,
\begin{theorem}If the age of $G_\mu$ does not contain the age of the infinite path nor of its complement, then it has only finitely many bounds which are not permutation graphs.
\end{theorem}

%\begin{lemma}If $\mu$ is uniformly recurrent non periodic $0$-$1$ sequence, then the $\fac(\mu)$ has infinitely many bounds.
%\end{lemma}

It is tempting to think that candidates for bounds of $\age(G_\mu)$ of type $(3)$ are graphs of the form $G_w$ where $w$ is a bound of $\fac(\mu)$. This is false.

\begin{lemma}\label{lem:bound-mu-not-bound-gmu}Let $\mu$ be a recurrent $0$-$1$ sequence on $\NN$ and $w:=w_1\ldots w_n$ be a finite word. If $w_2 \ldots w_n$ is a factor of $\mu$, then $G_w$ embeds into $G_\mu$.
\end{lemma}
\begin{proof}Suppose $w_2\ldots w_n$ is a factor of $\mu$. Let $\{j_2,\ldots,j_n\}\subseteq \NN$ be such that $\mu(j_k)=w_k$ for all $2\leq k\leq n$. Since $\mu$ is recurrent we may assume that there are at least three elements of $\NN\cup\{-1\}$ before $j_2$. Let $j_1:=j_2-1$ and $w'_1:=\mu(j_1)$. If $w'_1= w_1$, then $w$ is a factor of $\mu$ and hence $G_w$ embeds into $G_\mu$. Else if $w'_1\neq w_1$, then we set $j_0:=j_1-2$. It follows that $G_w$ is isomorphic to ${G_\mu}_{\restriction \{j_0,j_1\ldots,j_n\}}$.
\end{proof}

\begin{corollary}\label{cor:bound-mu-not-bound-gmu}Let $\mu$ be a recurrent $0$-$1$ sequence on $\NN$ and $w$ be a finite word. If $w$ is a bound of $\fac(\mu)$, then $G_w$ embeds into $G_\mu$.
\end{corollary}

\subsection{Proof of $(1)$ of Theorem \ref{thm:bound-uniform}.}
We show first how to construct a bound of $\age(G_\mu)$ using a bound of $\mu$.

\begin{lemma}\label{lem:bound-gmu}Let $\mu$ be a recurrent $0$-$1$ sequence on $\NN$ with $l(\mu)$ finite.  Let $w=w_1\ldots w_n$ be a finite word such that $n>l(\mu)+7$.
\begin{enumerate} [$(1)$]
\item If $w:=w_1\ldots w_n$ is a bound of $\mu$ and  $w_0\in \{0,1\}$ is such that $w_0\ldots w_{n-1}$ is a factor of $\mu$ and  $w':=w_0w_1\ldots w_n$,   then $G_{w'}$ is a bound of $\age (G_\mu)$.

\item If  $G_{w}$ is a bound of $\age(G_\mu)$, then $w_2\ldots w_{n}$ is a bound of $\fac(\mu)$.
 \end{enumerate}
\end{lemma}
\begin{proof}
$(1)$ We need to prove that $G_{w'}$ does not embed into $G_\mu$ and that deleting any vertex from $G_{w'}$ yields a graph that embeds into $G_\mu$. We notice at once that it follows from our assumption $n>l(\mu)+7$ that $G_{w'}$ is prime. We first prove that $G_{w'}$ does not embed into $G_\mu$. Suppose not and let $f$ be an embedding of $G_{w'}$ into $G_\mu$. Then $f(i_n)=\max(f(V(G_{w'})))$ because otherwise it follows from Corollary \ref{cor:6'} that $w'$ has $0^{n-6}$ or $1^{n-6}$. Hence, $w$ has $0^{n-7}$ or $1^{n-7}$ as a factor. Since $w_0\ldots w_{n-1}$ is a factor of $\mu$ we infer that $n-7<l(\mu)$ contradicting our assumption that $n>l(\mu)+7$. This proves that $f(i_n)=\max(f(V(G_{w'})))$. It follows then for Lemma \ref{lem:embed-left-left} that $w$ is a factor of $\mu$ contradicting our assumption that $w$ is a bound of $\fac(\mu)$. This proves that $G_{w'}$ does not embed into $G_\mu$.\\
Next we prove that deleting any vertex from $G_{w'}$ yields a graph that embeds into $G_\mu$. Set $V(G_{w'})=\{-1,0,\ldots,n\}$. First we consider the graph $G_{w'}\setminus \{-1\}$ and observe that it is isomorphic to $G_w$. It follows from Corollary \ref{cor:bound-mu-not-bound-gmu} that $G_w$ embeds into $G_\mu$. We now consider the graph $G_{w'}\setminus \{n\}$ and observe that it is isomorphic to $G_{w_0\ldots w_{n-1}}$. Since $w_0\ldots w_{n-1}$ is a factor of $\mu$ we infer  that $G_{w_0\ldots w_{n-1}}$ is an induced subgraph of $G_\mu$. Let $k\not \in \{-1,n\}$ and consider the graph $G_{w'}\setminus \{k\}$. Then ${G_{w'}}_{\restriction \{-1,\ldots,k-1\}}$ is the graph $G_{w_{0}\ldots w_{k-1}}$ and  ${G_{w'}}_{\restriction \{k+1,\ldots,n\}}$ is the graph $G_{w_{k+2}\ldots w_{n}}$. Since $w_{0}\ldots w_{k-1}$ and $w_{k+2}\ldots w_{n}$ are factors of $\mu$ the graphs $G_{w_{0}\ldots w_{k-1}}$ and $G_{w_{k+2}\ldots w_{n}}$ are induced subgraphs of $G_\mu$, and hence, so is $G_{w'}\setminus \{k\}$. This completes the proof of  $(1)$.\\

\noindent $(2)$ Suppose $G_{w}$ is a bound of $\age(G_\mu)$. Then $w$ cannot be a factor of $\mu$ and it follows from Lemma \ref{lem:bound-mu-not-bound-gmu} that $w$ is not a bound of $\fac(\mu)$. Hence, $w$ has a factor which is a bound of $\fac(\mu)$. We prove that $w_2\ldots w_{n}$ is a bound of $\fac(\mu)$, that is $w_2\ldots w_{n}$ is a not a factor of $\mu$ and both words $w_3\ldots w_{n}$ and $w_2\ldots w_{n-1}$ are factors of $\mu$. The fact that $w_2\ldots w_{n}$ is not a factor of $\mu$ follows from Lemma \ref{lem:bound-mu-not-bound-gmu} and the fact that $G_w$ does not embed in $G_\mu$. Next we prove that $w_3\ldots w_{n}$ and $w_2\ldots w_{n-1}$ are factors of $\mu$. It follows from our assumption $n>l(\mu)+7$ and Corollary \ref{cor:large-prime-restr} that $G_w$ is prime. Next we set $V(G_w):=\{i_0,i_1,\ldots,i_n\}$ so that $w$ is a word on $\{i_1,\ldots,i_n\}$. It follows from Corollary \ref{cor:6'} that $G_w \setminus \{i_0\}$ and  $G_w \setminus \{i_n\}$ are prime. It follows from our assumption that $G_w$ is a bound of $\age(G_\mu)$ that $G_w \setminus \{i_0\}$ and $G_w \setminus \{i_n\}$ embed in $G_\mu$. It follows from Lemma \ref{lem:embed-left-right} and our assumption $n>l(\mu)+7$ that if $f$ and $g$ are such embeddings then $f(i_n)=\max(f(G_w \setminus \{i_0\}))$ and $g(i_{n-1})=\max(g(G_w \setminus \{i_n\}))$. Lemma \ref{lem:embed-left-left} yields that $\mu_{\restriction f(\{i_3,\ldots, i_n\})}=w_3\ldots w_n$ and $\mu_{\restriction g(\{i_2,\ldots, i_{n-1}\})}=w_2\ldots w_{n-1}$. This proves that $w_3\ldots w_{n}$ and $w_2\ldots w_{n-1}$ are factors of $\mu$ as required.
\end{proof}

%\begin{corollary}\label{cor:bound-uniform}Let $\mu$ be a uniformly recurrent $0$-$1$ sequence on $\NN$. Then $\age(G_\mu)$ has infinitely many bounds.
%\end{corollary}
%\begin{proof}

The proof of $(1)$ of Theorem \ref{thm:bound-uniform} follows from Theorem \ref{thm:inf-bound-uniform-seq} and $(1)$ of Lemma \ref{lem:bound-gmu}.

\subsection{Proof of $(2)$ of Theorem \ref{thm:bound-uniform}}

We notice at once that if $\mu$ is periodic and $u$ is a period, then $\overline{\mu}$ is periodic and $\overline{u}$ is a period.

\begin{lemma}\label{lem:constant-final-interval}Let $\mu$ be a $0$-$1$ word on $\NN$, let $I:=\{i_0,i_1,\dots,i_{n-1}\}\subseteq \NN\cup \{-1\}$ so that $i_0<i_1<\dots <i_{n-1}$ and $H:= G_{\mu \restriction I}$ be an induced subgraph of $G_\mu$. Let $j<k<n-1$. If $i_j$ is adjacent to all vertices in $\{i_k,\ldots,i_{n-1}\}$, then $\mu$ is constant on $\{i_{k+1},\ldots,i_{n-1}\}$ and takes the value $0$. In particular, if $l(\mu)$ is finite and $\{i_{k+1},\ldots,i_{n-1}\}$  is an interval of $\NN$, then $n-l(\mu)-1\leq k$.
\end{lemma}
\begin{proof}
Straightforward.
\end{proof}

\begin{lemma}\label{lem:embed-indec}
Let $\mu$ be a $0$-$1$ word on $\NN$ such $l(\mu)$ is finite. Let $J:= \{j_0,j_1,\ldots,j_{k}\}\subseteq \NN$ be such that $j_0<j_1<\ldots<i_{k}$ and $\{j_1,\ldots,j_{k}\}$ is an interval of $\NN$ and $k>l(\mu)+5$. Then $G:=G_{\mu}{\restriction J}$ is prime.
\end{lemma}
\begin{proof}

Suppose for a contradiction that $G$ is not prime.  Let $M$ be a nontrivial module of $G$. Then $M\cap \{j_1,\ldots,j_{k}\}$ is a module of $G\setminus \{j_0\}$. It follows from our assumption that $k>l(\mu)+5$ and Corollary \ref{cor:large-prime-restr} that $G\setminus \{i_0\}$ is prime. Hence, $M\cap \{j_1,\ldots,j_{k}\}$ is either empty, reduced to a singleton or is equal to $\{j_1,\ldots,j_{k}\}$. Since $M$ is nontrivial  we infer that  $M=\{j_1,\ldots,j_{k}\}$ or $M\cap \{j_1,\ldots,j_{k}\}$ is a singleton. If $M=\{j_1,\ldots,j_{k}\}$, then $j_0$ must be either adjacent to all elements of $M$ or adjacent to none. Thus $\mu$ is constant on $M$, that is $k\leq l(\mu)<k-5$. A contradiction. Else if $M\cap \{j_1,\ldots,j_{k}\}$ is a singleton, then $M=\{j_0,j_m\}$ for some $1\leq m\leq k$. Necessarily $m=k$, because otherwise $j_{m+1}$ separates $j_m$ from $j_0$. That is $M=\{j_0,j_k\}$. Suppose $\mu(j_k)=1$. Then no vertex in $\{j_2,\ldots,j_{k-2}\}$ is adjacent to $j_k$. Since $M$ is a module, no vertex in $\{j_2,\ldots,j_{k-2}\}$ is adjacent to $j_0$ and therefore $\mu$ is constant on $\{j_2,\ldots,j_{k-2}\}$ and takes the value $1$. Thus $\mu$ has $1^{k-3}$ as factor. If $\mu(j_k)=0$, then we obtain that $\mu$ has $0^{k-3}$ as factor. Therefore, $k-3\leq l(\mu)$ and from our assumption $k>l(\mu)+5$ we get $k-3<k-5$ which is impossible.
\end{proof}

\begin{corollary}\label{cor:threevertices} Let $\mu$ be a $0$-$1$ word on $\NN$ such $l(\mu)$ is finite. Let $J:= \{j_0,j_1,\ldots,j_{k}\}\subseteq \NN$ be such that $j_0<j_1<\ldots<i_{k}$ and $\{j_1,\ldots,j_{k}\}$ is an interval of $\NN$ and $k>l(\mu)+6$. Let  $G:=G_{\mu}{\restriction J}$ and  $x\in J$. Then  $G\setminus \{x\}$ is prime if and only if $x\in \{j_0,j_1,j_k\}$.
\end{corollary}

 \begin{lemma}\label{lem:embed-bound-periodic}
Let $\mu$ be a $0$-$1$ word on $\NN$ such that $l(\mu)$ is finite. Let $\{i_0,i_1,\ldots,i_{n-1}\}\subseteq \NN$ be such that $i_0<i_1<\ldots<i_{n-1}$ and $\{i_1,\ldots,i_{n-1}\}$ is an interval of $\NN$ and $n>l(\mu)+8$. Let $x\not \in \NN$ and $H$ be the graph whose vertex set is $\{i_0,i_1,\ldots,i_{n-1}\}\cup \{x\}$ and edge set $E:=E({G_\mu}_{\restriction \{i_0,i_1,\ldots,i_{n-1}\}})\cup \{\{i_1,x\},\{i_2,x\},\ldots,\{i_{n-1},x\}\}$. Then $H$ does not embed into $G_\mu$.
\end{lemma}
\begin{proof}
%We first prove that the graph $H\setminus \{x\}={G_\mu}_{\restriction \{i_0,i_1,\ldots,i_{n-1}\}}$ is prime.  This proves that $H\setminus \{x\}$ is prime.
Suppose for a contradiction that $H$ embeds into $G_\mu$ and let $f$ be such an embedding. Then $f$ induces an embedding of $H\setminus\{x\}$ into $G_\mu$.
It follows from Lemma \ref {lem:embed-indec} that $H\setminus\{x\}$ is prime. According to Corollary \ref{lem:5} the image of $H\setminus\{x\}$ under $f$ decomposes into a point $y$ and an interval $J$ to its right. It follows from Corollary \ref{cor:threevertices} that $f(\{i_0, i_1, i_n-1\})= \{y, \min (J), \max (J)\}$. It follows from Lemma \ref{lem:embed-left-right} that $f(i_{n-1})= \max (J)$. Hence, $f(\{i_0, i_1\})= \{y, \min (J)\}$.
%apply Proposition  \ref{prop:embed-lmu-finite} to $w:= \mu_{\restriction \{i_2, \dots i_{n-1}\}}$ and note that $G_w$ is isomorphic to $H\setminus \{x, i_0\}$. We deduce that   $f(i_1), f(i_2)< \ldots<f(i_{n-1})$ and $\{f(i_3),\ldots,f(i_{n-1})\}$ is an interval of $\NN$.
Now, we argue on the possible position of $f(x)$. Suppose that $f(x)$ is to the left of $f(i_{n-1})$. Since $\{f(i_2),\ldots,f(i_{n-1})\}$ is an interval of $\NN$ we infer that $f(x)$ is to the left of $f(i_2)$. Since $f(x)$ is adjacent to all vertices in $\{f(i_1),f(i_2),\ldots,f(i_{n-1})\}$ it follows from Lemma \ref{lem:constant-final-interval} that $\mu$ is constant on $\{f(i_2),\ldots,f(i_{n-1})\}$. Hence, $n-2\leq l(\mu)$. From our assumption that $n>l(\mu)+8$ we get $n<n-8$ which is impossible. Now suppose that $f(x)$ is to the right of $f(i_{n-1})$. Since $f(x)$ is adjacent to all vertices in $\{f(1),\ldots,f(n-1)\}$ it is adjacent to $f(i_{n-2})$ and $f(i_{n-1})$. It follows that  $\mu(f(x))=0$. Thus   $f(x)$ is adjacent to $f(i_0)$, hence $x$ is adjacent to $i_0$ in $H$.  A contradiction. This proves that our supposition $H$ embeds into $G_\mu$ is false.
\end{proof}

A vertex $x$ of a graph $G$ is $-1$-\emph{extremal} if either $x$ is not adjacent to at most one vertex of $V(G)\setminus \{x\}$ or if $x$ is  adjacent to at most one vertex of $V(G)\setminus \{x\}$. Note that if $x$ is $-1$-extremal in $G$, then $x$ is also $-1$-extremal in $\overline{G}$.

\begin{lemma}\label{lem:argument-pouzet}Let $\mathcal{C}$ be a hereditary class of finite graphs which is $1^-$-well-quasi-ordered. Then $\mathcal{C}$ has only finitely many bounds having a $-1$-extremal vertex.
\end{lemma}
\begin{proof}
Since $\mathcal{C}$ is w.q.o. there are only finitely many bounds of $\mathcal{C}$ having a vertex adjacent to all other vertices. Let $(G_n)_{n\in \NN}$ be a sequence of bounds of $\mathcal{C}$ such that each $G_n$ has a  $-1$-extremal vertex $x_n$. We may suppose  that there is a  unique vertex $y_n$ distinct from $x_n$ and not adjacent to $x_n$. Let $H_n:={G_n}_{\restriction V(G_n)\setminus \{x_n\}}$. Since $\mathcal{C}$ is $1^-$-well-quasi-ordered from the sequence $(H_n, y_n)$ we can extract an increasing subsequence. Clearly, if $(H_n, y_n)$ embeds into $(H_m, y_m)$, then $G_n$ embeds into $G_m$. This contradicts the fact that $\{G_n: n\in \NN\}$ forms an antichain.
\end{proof}

\begin{corollary}\label{cor:argument-pouzet}
Let $\mu$ be a periodic $0$-$1$ sequence on $\NN$. Then there are only finitely many bounds of $\age(G_\mu)$ having a $-1$-extremal vertex.
\end{corollary}
\begin{proof}Follows from Lemma \ref{lem:argument-pouzet} and the fact that $\age(G_\mu)$ is $1^-$-well-quasi-ordered.
\end{proof}

\begin{lemma}\label{lem:bound-prime-periodic}
Let $\mu$ be a periodic $0$-$1$ sequence on $\NN$. Then the number of non prime bounds of $\age(G_\mu)$ is finite.
\end{lemma}
\begin{proof}Let $(G_n)_{n\in \NN}$ be a sequence of bounds of  $\age(G_\mu)$. Suppose $G_n$ is not prime.  Let $M_n$ be a nontrivial module of $G_n$ and  $x_n$ any vertex  of $M_n$. Let $H_n:={G_n}_{\restriction (V(G_n)\setminus M_n)\cup \{x_n\}}$. Since $M_n$ is nontrivial and $G_n$ is a bound we infer that $H_n$ and $M_n$ are elements of $\age(G_\mu)$. Since $\age(G_\mu)$ is w.q.o there exists an infinite subset $I$ of $\NN$ so that the sequence $({G_n}_{\restriction M_n})_{n\in I}$ is increasing with respect to embeddability. Since $\age(G_\mu)$ is $1^-$-well-quasi-ordered we infer that we can extract from the sequence $(H_n,x_n)_{n\in I}$ an increasing subsequence. Then note that if $(H_n,x_n)$ embeds into $(H_m,x_m)$ and $M_n$ embeds into $M_m$, then $G_n$ embeds into $G_m$.
\end{proof}

We now prove $(2)$ of Theorem \ref{thm:bound-uniform}. Let $\mu$ be a non constant and periodic $0$-$1$ sequence on $\NN$ and let $H$ be a bound of $G_\mu$. It follows from Lemma \ref{lem:bound-prime-periodic} that we may assume that $H$ is prime. Since the examples of critically prime graphs of Schmerl and Trotter \cite{S-T} split into two totally ordered sets with respect to embeddability we may assume that $H$ is not critically prime. There exists then $x\in V(H)$ such that $H\setminus \{x\}$ is prime. Since $H$ is a bound of $G_\mu$ we infer that $H\setminus \{x\}$ embeds into $G_\mu$.  Let $f_x$ be such an embedding. We write $f_x(V(H\setminus \{x\}) :=\{i_0,i_1,\ldots,i_n\}$ so that $i_0<i_1<\ldots<i_n$. Since $H\setminus \{x\}$ is prime it follows from Corollary \ref{lem:5} that $\{i_1,\ldots,i_n\}$ is an interval of $\NN$. Since $\mu$ is periodic $\l(\mu)$ is finite. For $n>l(\mu)+5$,  ${G_{\mu}}{\restriction \{ i_1<\ldots<i_n \}}$ is prime, hence $H\setminus \{x, f_x^{-1}(i_0)\}$ is prime. We may  assume that  $\mu(i_n)=0$ (if not consider $\overline{G_\mu}=G_{\overline{\mu}}$ and $\overline{H}$ and note that $\overline{\mu}$ is also periodic). By Lemma \ref{lem:argument-pouzet} we may assume that $H$ has no $-1$-extremal vertices. It follows that  $\{x, f_x^{-1}(i_n)\}$ is not an edge of $H$ (otherwise $f_x^{-1} (i_n)$ is $-1$-extremal in $H$). We now consider the graph $H\setminus \{f_x^{-1}(i_0)\}$. Let  $g_{i_0}$ be an embedding of $H\setminus \{f_x^{-1}(i_0)\}$ into $G_\mu$. For $1\leq k\leq n$, we define $i'_k:= g_{i_0}(f^{-1}_{x}(i_k))$.

Suppose $n>l(\mu)+7$.
It follows then from Proposition \ref{prop:embed-lmu-finite} that every embedding of $H\setminus \{f_x^{-1}(i_0)\}$ in $G_\mu$ maps
$\{i'_2,\ldots,i'_n\}$ into an interval and in that order. Hence, such an embedding agree with $f_x$ and $g_{i_0}$. From our assumption that $\mu(i_n)=0$ and $\{x,f_x^{-1}(i_n)\}$ is not an edge of $H$ we deduce that $g_{i_0}(x)$ is to the right of $i'_n$. Indeed, if $g_{i_0}(x)$ is to the left of $i'_n$, then since $\mu(i'_n)=\mu(i_n)=0$ we infer that $g_{i_0}(x)$ is on the left of $i'_3$. But then $\{g_{i_0}(x), i'_{n}\}$ is an edge, therefore $\{x, f_x^{-1}(i_0) \}$ is an edge  of $H$ hence $x$ is $-1$-extremal, which is not possible . Thus, $g_{i_0}(x)$ is to the right of $i'_n$. It follows then that $g_{i_0}(x)$ is either adjacent to all vertices in $\{i'_1,\ldots,i'_{n-1}\}$ or adjacent to none. This last case is not possible, otherwise $x$ would be $-1$-extremal. So we are left with the case that $g_{i_0}(x)$ is  adjacent to all vertices in $\{i'_1,\ldots,i'_{n-1}\}$. Since $x$ is not $-1$-extremal,   $x$ is adjacent to all vertices of $H \setminus \{f_x^{-1}(i_0),  f_x^{-1}(i_n)\}$ and not adjacent to either $f^{-1}(i_0)$ or   $f^{-1}(i_n)$. It follows from Lemma \ref{lem:embed-bound-periodic} that $H\setminus \{i_n\}$ does not embed into $G_\mu$. This contradicts our assumption that $H$ is a bound of $\age(G_\mu)$.

\section{Conclusion}

This work on hereditary classes of finite graphs containing relatively few primes put a light on hereditary classes which are  well-quasi-ordered and also on those made  of  permutation graphs. The result of \cite{chudnovsky} was crucial in proving that our list of hereditary classes of graphs which are minimal prime was complete. Kim \cite{kim} obtained for tournaments a result similar to Chudnovski and al \cite{chudnovsky}. It remains to see if
similar results to ours can be obtained in the case of tournaments;  and also, if  they shed light on the case of binary relations and binary relational structures and allow to solve the problems mentioned in the text about minimal prime hereditary classes. Among question which interest us are first the rank  of minimal prime classes of permutation graphs; in this respect note that  it is unknown if there are hereditary well-quasi-ordered classes of graphs with arbitrary countable rank (see \cite{pouzet-sobrani-SP}). Next, the question to know wether or not well-quasi-
ordered hereditary classes of finite graphs are better-quasi-ordered. \\
A consequence of our study is the existence  of an uncountable antichain of well-quasi-ordered ages of permutation graphs. The existence of uncountably many  well-quasi-ordered ages of binary structures was obtained in 1978 \cite {pouzettr}. This was obtained by means of a coding via uniformly recurrent sequences. The same existence  for graphs, permutation graphs or posets, is a non trivial fact which requires work.  The same coding than the one we use in this paper was used first in 1992 \cite{sobranithesis} and in 2002 \cite{sobranietat}. In Chapter 5 of \cite{oudrar} the first author proved with a simpler coding the existence of uncountably many hereditary classes of oriented graphs which are  minimal prime.  We conclude by mentioning the existence of uncountably many  well-quasi-ordered ages of permutation graphs with distinct enumeration functions (alias profile) due to Brignall and Vatter \cite{brignall-vatter2}.

\section*{Acknowledgements}
The authors would like to sincerely thank  Robert Brignall for bringing to their attention several informations, notably on pin sequences and labelled classes of permutations.

\end{document}